\documentclass[12pt,reqno]{amsart}

\usepackage[colorlinks=true,urlcolor=blue,
bookmarks=true,bookmarksopen=true,citecolor=blue
]{hyperref}

\usepackage{color,soul}
\setstcolor{red}

\usepackage{pdfsync}
\usepackage[all, cmtip]{xy}
\usepackage{graphicx}

\usepackage{epsfig}
\usepackage{amscd}
\usepackage[mathscr]{eucal}
\usepackage{amssymb}
\usepackage{bm}
\usepackage{amsxtra}
\usepackage{amsmath}
\usepackage{latexsym} 
\usepackage[all]{xy}
\usepackage{enumerate}

\usepackage{enumitem}

\usepackage{color}
\usepackage{xcolor, soul}
\sethlcolor{green}

\usepackage{bbm}

\input{widebar}

\theoremstyle{plain}

\newtheorem{mainthm}{Theorem}
\newtheorem{thm}{Theorem}[subsection]

\newtheorem{lem}[thm]{Lemma}
\newtheorem{prop}[thm]{Proposition}

\newtheorem*{cnj-nonum}{Conjecture}

\theoremstyle{definition}
\newtheorem{dfn}[thm]{Definition}

\newtheorem*{claim-nonum}{Claim}

\theoremstyle{remark}
\newtheorem{rem}[thm]{Remark}
\newtheorem*{remnonum}{Remark}
\newtheorem{rems}[thm]{Remarks}

\theoremstyle{plain}

\newcommand{\Qed}{\hfill \qedsymbol \medskip}

\newcommand{\id}{\textnormal{id}}

\newcommand{\R}{\mathbb{R}}

\newcommand{\fuk}{\mathcal{F}uk}

\newcommand{\thmb}{T^{^{\uparrow}}}
\newcommand{\tthmb}{\mathcal{T}^{^{\uparrow}}}

\newcommand{\vbl}{L}

\newcommand{\ocnote}[1]{#1}

\newcommand{\pbaddress}{biran@math.ethz.ch}
\newcommand{\ocaddress}{cornea@dms.umontreal.ca}

\begin{document}

\title{Bounds on the Lagrangian spectral metric in cotangent bundles}

\date{\today}

\thanks{The second author was supported by an individual NSERC
  Discovery grant.}

\author{Paul Biran and Octav Cornea}

\address{Paul Biran, Department of Mathematics, ETH-Z\"{u}rich,
  R\"{a}mistrasse 101, 8092 Z\"{u}rich, Switzerland}
\email{\pbaddress}
 
\address{Octav Cornea, Department of Mathematics and Statistics,
  University of Montreal, C.P. 6128 Succ.  Centre-Ville Montreal, QC
  H3C 3J7, Canada} \email{\ocaddress}

\bibliographystyle{alphanum}

%

%

\begin{abstract}
  
Let $N$ be a closed manifold and $U \subset T^*(N)$ a bounded domain in
the cotangent bundle of $N$, containing the zero-section. A conjecture
due to Viterbo asserts that the spectral metric for Lagrangian
submanifolds that are exact-isotopic to the zero-section is
bounded. In this paper we establish an upper bound on the spectral
distance between two such Lagrangians $L_0, L_1$, which depends
linearly on the boundary depth of the Floer complexes of $(L_0, F)$
and $(L_1, F)$, where $F$ is a fiber of the cotangent bundle.

\end{abstract}

\maketitle

%
%




\section{Introduction and main results} \label{s:intro}

Let $N$ be a closed manifold and $T^*(N)$ its cotangent bundle,
endowed with its standard symplectic structure. A domain
$U \subset T^*(N)$ is called bounded if there is a Riemannian metric
$g$ on $N$ such that $U$ is contained inside the unit-ball cotangent
bundle of $T^*(N)$ with respect to the metric associated to $g$ on the
fibers of $T^*(N)$. More specifically,
$U \subset \{ v \in T^*(N) \mid |v| \leq 1\}$, where $| \cdot |$ is
the norm on the fibers of $T^*(N)$ corresponding to the metric $g$ via
the isomorphism $T^*(N) \cong T(N)$ induced by $g$. Since $N$ is
compact, the boundedness of $U$ is independent of the choice of $g$.

For a domain $W \subset T^*(N)$ we denote by
$\mathcal{L}_{\text{ex}}(W)$ the collection of closed exact Lagrangian
submanifolds of $W$ (exactness is considered here with respect to the
canonical Liouville form) and by
$\mathcal{L}_{\text{ex}, N}(W) \subset \mathcal{L}_{\text{ex}}(W)$ the
collection of Lagrangians that are exact isotopic (within $T^*(N)$) to
the zero section $N \subset T^*(N)$.

There are several $\textnormal{Ham}$-invariant metrics on
$\mathcal{L}_{\text{ex},N}(T^*(N))$. For example, the Hofer metric on
$\textnormal{Ham}(T^*(N))$ descends to a non-degenerate metric
$d_{\text{Hof}}$ on $\mathcal{L}_{\text{ex}, N}(T^*(N))$. Another
important metric, due to Viterbo, is the spectral metric. This was
originally defined for $\mathcal{L}_{\text{ex}, N}(T^*(N))$, but
thanks to more recent developments can be extended to the entire of
$\mathcal{L}_{ex}(T^*(N))$. (See
Remark~\ref{r:spec-choices}~-~(\ref{i:extend-metric}) and
Remark~\ref{r:exact-lags}~-~(\ref{i:FSS-exact-lags}) for more on
this.) The spectral distance $\gamma(L_0, L_1)$ between two elements
$L_0, L_1 \in \mathcal{L}_{ex}(T^*(N))$ is define as:
\begin{equation} \label{eq:gamma} \gamma(L_0, L_1) = c([N]; L_0,L_1) -
  c([\textnormal{pt}]; L_0,L_1),
\end{equation}
where $c([N]; L_0,L_1)$, $c([\textnormal{pt}]; L_0, L_1)$ stand for
the spectral invariants associated to $(L_0, L_1)$, for the
fundamental class $[N] \in H_n(N)$ and for the class of a point
$[\textnormal{pt}] \in H_0(N)$,
correspondingly. See~\S\ref{s:spectral-1} (and more
specifically~\S\ref{sb:fil-HF} and~\S\ref{sbsb:action-filtrations})
below for the precise definitions.

It is well known that $\gamma(L_0, L_1) \leq d_{\text{Hof}}(L_0,L_1)$
for all $L_0, L_1 \in \mathcal{L}_{\text{ex}, N}(T^*(N))$. However
beyond this inequality, little is known about the relation between
these two metrics.

Let $U \subset T^*(N)$ be a bounded domain. It is also known that, at
least for some $N$'s, the Hofer metric on
$\mathcal{L}_{\text{ex}, N}(U)$ is unbounded. This has been proved for
several cases like $N=S^1$ by Khanevsky~\cite{Kh:diameters} and is
conjectured to hold for all $N$'s.

In contrast to the Hofer metric, there is the following conjecture
regarding the spectral metric:
\begin{cnj-nonum}[Viterbo] \label{p:vit-conj}
  The spectral metric on $\mathcal{L}_{\text{ex},N}(U)$ is bounded.
\end{cnj-nonum}
This was conjectured by Viterbo in~\cite{Vi:homogen} for the case
$N=\mathbb{T}^n$, and is expected to hold for all closed manifolds
$N$. Recently Shelukhin~\cite{Shl:vit-1, Shl:vit-2} proved this
conjecture for several classes of manifolds $N$ (including
$\mathbb{T}^n$).

Our main result, which applies to all closed manifolds $N$, is the
following.
\begin{mainthm} \label{t:main-a} Let $N$ be any closed manifold and
  $U \subset T^*(N)$ a bounded domain. There exist constants $A, B >0$
  that depend only on $U$ such that for every
  $L_0, L_1 \in \mathcal{L}_{\text{ex},N}(U)$ we have:
  \begin{equation} \label{eq:bound-main-thm} \gamma(L_0, L_1) \leq B
    \bigl( \beta(CF(L_0, F_q)) + \beta(CF(L_1, F_q)) \bigr) + A.
  \end{equation}
  Here $F_q = T_q^*(N)$ is the fiber of the cotangent bundle at an
  arbitrary point $q \in N$, viewed as a Lagrangian submanifold of
  $T^*(N)$, and $\beta(CF(L_i, F_q))$ is the boundary depth of the
  Floer complex of the pair $(L_i, F_q)$, $i=0,1$, defined with
  coefficients in $\mathbb{Z}_2$. See~\S\ref{sb:filtered} for the
  definition of boundary depth.
\end{mainthm}

\begin{rems} \label{r:main-thm}
  \begin{enumerate}
  \item \label{i:beta-vit} Clearly, the above conjecture of Viterbo
    would follow from Theorem~\ref{t:main-a} if we can show that the
    boundary depth $\beta(CF(L, F_q))$ is uniformly bounded in
    $L \in \mathcal{L}_{\text{ex},N}(U)$.
  \item \label{i:vit-beta} \ocnote{The converse to the statement made
      at point~(\ref{i:beta-vit}) above turns out to be also true.
      Namely, if the conjecture of Viterbo holds true then the
      boundary depth $\beta(CF(L, F_q))$ is uniformly bounded in
      $L \in \mathcal{L}_{\text{ex},N}(U)$. This follows by a
      relatively simple argument that we summarize
      in~\S\ref{sb:specm-betta-bound}.}
  \item The chain complex $CF(F_q, L_i)$ depends on the point
    $q \in N$ and so does its boundary depth $\beta(CF(L_i,
    F_q))$. However, as we will see in~\S\ref{s:prf-main},
    Lemma~\ref{l:diff-fibers}, the difference
    $|\beta(CF(L, F_{q'})) - \beta(CF(L, F_{q''}))|$ is bounded,
    uniformly in $q', q'' \in N$,
    $L \in \mathcal{L}_{\text{ex},N}(U)$. Therefore the formulation of
    inequality~\eqref{eq:bound-main-thm} with constants $A, B$ that do
    not depend on $q$, makes sense.
  \item While the chain complex $CF(L, F_q)$ might be complicated and
    have arbitrary large rank, its homology is very simple:
    $HF(L, F_q) \cong \mathbb{Z}_2$ for every $q \in N$,
    $L \in \mathcal{L}_{\text{ex},N}(U)$.
  \end{enumerate}
\end{rems}


\subsection{Strategy and main ideas in the proof} \label{sb:main-i}
The starting point of the proof is borrowed from~\cite{FSS:ex} - we
embed a tubular neighborhood $\mathcal{U}$ of the zero section of
$T^*(N)$ into a real affine algebraic manifold $E$ which also serves
as the total space of a Lefschetz fibration
$\pi: E \longrightarrow \mathbb{C}$ endowed with a real structure. The
embedding can be arranged such that the zero section is sent to (one
of the components of) the real part of $E$.

The 2'nd step appeals to our previous work~\cite{Bi-Co:lefcob-pub}
which establishes canonical presentations of Lagrangians $K$ in
Lefschetz fibrations as iterated cone decompositions with standard
factors. These iterated cone decompositions take place in the category
of modules over the Fukaya category of $E$ and hold up to
quasi-isomorphisms. The factors in the decomposition of $K$ consist of
the Yoneda modules of certain Lefschetz thimbles emanating from the
critical points of $\pi$ along $N$, as well as some factors that
involve the Floer complexes of pairs of thimbles and pairs of the type
$(\text{Thimble}, K)$. This makes it possible to express $CF(L,K)$ for
every exact Lagrangians $L$, as an iterated cone involving chain
complexes of the types $CF(L, \text{Thimble})$, $CF(\text{Thimble},K)$
and $CF$ of pairs of thimbles. Note that the 2'nd and 3'rd types do
not involve $L$.

By specializing to the case $K=N$ and taking the $L$ to correspond to
a Lagrangian in the neighborhood $\mathcal{U}$ of the zero-section,
the previous cone decomposition of $CF(L,N)$ reduces now to terms of
the type $CF(L,F_{q})$, for different critical points $q \in N$ of
$\pi$, and some other fixed chain complexes that do not depend on
$L$. The terms $F_q$ appear here because the previously mentioned
thimbles coincide within $\mathcal{U}$ with the fibers $F_q$ of the
cotangent bundle. A ``local to global'' argument in Floer theory shows
that replacing the thimble emanating from a critical point $q \in N$
of $\pi$ by $F_q$ does not change the respective Floer complexes.

The next step is to analyze the spectral metric using the above cone
decompositions. This requires a refinement of the cone decomposition
in the realm of filtered Floer theory. It turns out that the above
cone decomposition continues to hold in the filtered sense up to a
bounded action shift. Therefore, in principle once can recover (up to
a bounded shift) the filtered Floer homology of $(L,N)$ from the
filtered Floer homology of the factors mentioned above and the
knowledge of the chain maps between the factors which form the
cones. In practice this not so effective, as these chain maps are in
general hard to describe explicitly. Fortunately, this obstacle can be
overcome by algebraic means which are described next.

The next step in the proof is purely algebraic. Here we obtain a
coarse uniform upper bounds on the spectral range of filtered mapping
cones $C = \text{Cone}(C' \xrightarrow{\; f \;} C'')$ between two
filtered chain complexes $C'$ and $C''$. By ``spectral range'' of a
filtered chain complex we mean the difference between the highest and
the lowest spectral invariants of that complex. It turns out that one
can derive such a bound on the spectral range of $C$ which involves
only the following pieces of data: the spectral ranges of $C'$ and
$C''$, the boundary depths of $C'$ and $C''$ and the amount of
filtration shift in the map $f$. A crucial point here is that our
bound is uniform in $f$ in the sense that it does not involve specific
information on the map $f$, except of the extent by which it shifts
the filtrations.  We also establish an analogous upper bound for the
boundary depth of $C$.  Having these two algebraic ingredients at
hand, we can derive similar upper bounds for the spectral range and
boundary depth of iterated cones.

The final step puts the geometry and algebra together. We apply the
algebraic estimates on the spectral range to the previously mentioned
cone decomposition of $CF(L,N)$. While it is possible to describe relatively precisely 
the chain maps between the terms in this decomposition, this is delicate. Fortunately, this is not
needed here as we can easily bound the amount by which these maps shift filtrations. Consequently we
obtain an upper bound on the spectral range of $CF(L,N)$ as the sum of
two terms: one of them is a constant $A$ that comes from the spectral
ranges of the factors in our cone decomposition (these are
straightforward to determine) and some uniformly bounded errors that
come from our coarse estimates. This constant depends on $\mathcal{U}$
but {\em not} on $L$ since the only appearance of $L$ in the cone
decomposition of $CF(L,N)$ is in terms of the type
$CF(L,F_q)$. However, the spectral range of such terms is $0$ because
$HF(L, F_q)$ is $1$-dimensional. The second summand in our bound looks
like $B \beta(CF(L,F_q))$, where $B$ is a constant and
$\beta(CF(L,F_q))$ is the boundary depth of $CF(L,F_q)$. Our main
result now easily follows from these bounds.

The above is only an outline of the main ideas in the proof. Along the
way there are several additional ingredients required for the proof to
work. These have to do with technicalities in Floer theory, Lefschetz
fibrations and filtered homological algebra.

\subsection{Organization of the paper} \label{sb:org-p} The rest of
the paper is organized as follows. Section~\ref{s:spectral-1} reviews
necessary preliminaries on filtered Floer theory in the framework of
exact Lagrangian submanifolds in Liouville manifolds. We also prove in
Section~\ref{sb:Fl-subdomains} a general ``local vs.~global'' result,
comparing the Lagrangian Floer persistent homologies in a Liouville
subdomain with the same type of homology in the entire Liouville
manifold.

Section~\ref{s:lef} is devoted to Lefschetz fibrations and their
relevance to our problem. We go over real Lefschetz fibrations in
general and then review a construction from~\cite{FSS:ex} which gives
an embedding of a neighborhood of the zero-section in $T^*(N)$ into a
real Lefschetz fibration $E$. We then go over a construction coming
from~\cite{Bi-Co:lefcob-pub} which alters the Lefschetz fibration $E$
into an an extended Lefschetz fibration $E'$ containing a collection
of matching spheres that will be useful for our purposes. Part of this
section is devoted to showing that the construction of $E'$ can be
made while preserving a geometric setting amenable to Floer theory
like exactness etc.

Section~\ref{s:Fl-E-E'} is dedicated to comparison between the
filtered Floer theory inside $E$ and the same theory viewed in $E'$.
In particular we show there that the matching spheres from $E'$,
constructed in Section~\ref{s:lef}, correspond in $E$ to some
Lefschetz thimbles emanating from $N$. These in turn coincide near $N$
with cotangent fibers of $T^*(N)$. We show that these correspondences
hold also in a Floer-theoretic sense.

Section~\ref{s:cone-decomp} is central for the proof of the main
theorem. There we discuss iterated cone decompositions in the Fukaya
categories of $E$ and $E'$. In particular we show how to represent
Lagrangian submanifolds in $E'$ as iterated cones with standard terms
coming from the matching spheres from Section~\ref{s:lef}. Moreover,
in Section~\ref{sb:filt-cd-Lef} and~\ref{sb:prfs-filt-cd} we extend
these decompositions to the realm of Fukaya categories endowed with
action filtrations. In particular we also derive a filtered version of
the Seidel exact triangle associated to a Dehn-twist.

Section~\ref{s:prf-main} combines the geometric contents of the
previous sections together with some filtered homological algebra
(developed in Section~\ref{s:fhomalg}) to conclude the proof of the
main theorem. We also sketch the argument for the converse.

The algebraic ingredients necessary for the paper are concentrated in
Section~\ref{s:fhomalg}. This is a purely algebraic section in which
we study spectral invariants and boundary depth of filtered chain
complexes. Special attention is given to filtered mapping cones and we
establish estimates on the spectral range and boundary depth in that
case.

The paper can be read linearly, with the exception of
Section~\ref{s:fhomalg} which is the last one, but is being referred
to at many instances along the paper. At the same time,
Section~\ref{s:fhomalg} is independent of the rest the paper and can
be read separately.

\subsection{Acknowledgments} \label{sb:ack} \ocnote{We thank Egor
  Shelukhin for suggesting this project to us and also pointed out the
  relevance of \cite{FSS:ex} in this context.}  This work was
initiated and partially carried out during our two weeks stay at the
Mathematical Research Institute of Oberwolfach in May 2017, in the
framework of the Research in Pairs program. We would like to thank the
Oberwolfach Institute for the wonderful hospitality and working
conditions during our visit. We would like to thank Sobhan Seyfaddini
for useful discussions related to~\S\ref{sb:specm-betta-bound}.

\tableofcontents 

\section{Lagrangian Floer theory and spectral invariants}
\label{s:spectral-1}

Here we briefly recall the definitions of spectral invariants,
boundary depth and the spectral metric on the space of Lagrangian
submanifolds. We refer the reader to~\cite{PRSZ:top-pers,
  Po-Sh-St:pers-op, Oh:symp-top-act-1, DKM:spectral, KMN:spec,
  Leclercq:spectral, Le-Za:spec-inv, Usher-Zhang:perh, Usher1, Usher2,
  Vi:generating-functions-1} for more details on the general theory of
these concepts.

\subsection{Filtered chain complexes and their
  invariants} \label{sb:filtered} Fix a unital ring $R$ and let $C$ be
a chain complex of $R$-modules.  By a filtration on $C$ we mean an
increasing filtration of subcomplexes of $R$-modules, indexed by the
real numbers. More specifically, for every $\alpha \in \mathbb{R}$ we
are given a subcomplex $C^{\leq \alpha} \subset C$ of $R$-modules and
for every $\alpha \leq \beta$ we have
$C^{\leq \alpha} \subset C^{\leq \beta}$.  For simplicity we will
assume from now on that the filtration on $C$ is exhaustive, i.e.
$\cup_{\alpha \in \mathbb{R}} C^{\leq \alpha} = C$.

The inclusions $C^{\leq \alpha} \subset C^{\leq \beta}$,
$\alpha \leq \beta$, and $C^{\leq \alpha} \subset C$ induce maps in
homology which we denote by:
$$i^{\beta, \alpha}: H_*(C^{\leq \alpha}) \longrightarrow H_*(C^{\leq
  \beta}), \quad i^{\alpha}: H_*(C^{\leq \alpha}) \longrightarrow
H_*(C).$$

Given a homology class $a \in H_*(C)$ we define its spectral invariant
$\sigma(a) \in \mathbb{R} \cup \{-\infty\}$ to be
\begin{equation} \label{eq:sig-a}
  \sigma(a) := \inf \{ \alpha \in \mathbb{R} \mid  a \in
  \textnormal{image\,} i^{\alpha} \}.
\end{equation}
Note that $\sigma(0) = -\infty$.

Another important measurement for our purposes is the boundary depth
$\beta(C)$ of a filtered chain complex $C$, which is defined as
follows:
$$\beta(C) := \inf \{ r \geq 0 \mid \, \forall \alpha,
\forall c \in C^{\leq \alpha} \text{ which is a boundary in } C,
\exists \, b \in C^{\leq \alpha + r} \text{ s.t. } c = d(b)\}.$$

We will elaborate more on spectral invariants, boundary depth and
other measurements of filtered chain complexes in~\S\ref{s:fhomalg}.

\subsection{Filtered Lagrangian Floer theory} \label{sb:fil-HF}

In what follows all symplectic manifolds and their Lagrangian
submanifolds will be implicitly assumed to be connected, unless
otherwise mentioned. And all Hamiltonian functions
$[0,1] \times W \longrightarrow \mathbb{R}$ will be implicitly assumed
to be compactly supported.

\subsubsection{Liouville and Stein manifolds} \label{sbsb:LiSt} In the
following we will be mainly concerned with symplectic manifolds of two
types: Liouville domains and manifolds that are Stein at infinity. We
refer the reader to~\cite{El-Ci:Weinstein-book} for the foundations of
the theory of such manifolds and much more. Below we briefly recall
the basic notions needed for our purposes.

A compact Liouville domain $(W, \omega = d\lambda)$ consists of a
compact manifold $W$ with boundary $\partial W$ and an exact
symplectic structure $\omega$, with a given primitive $1$-form
$\lambda$ (called the Liouville form) such that the following holds:
the Liouville vector field $X_{\lambda}$, defined by
$i_{X_{\lambda}} \omega = \lambda$, is outward transverse to
$\partial W$. Under this assumption the restriction
$\lambda_{\partial W} := \lambda|_{\partial W}$ is a contact form and
we denote by $\xi_{\lambda} := \ker \lambda_{\partial W}$ the contact
structure defined by $\lambda_{\partial W}$ on $\partial W$. We write
$\psi_t: W \longrightarrow W$, $t \leq 0$, for the flow of
$X_{\lambda}$ (which exists for all $t \leq 0$). We have
$\psi_t^* \lambda = e^t \lambda$ and $\psi_t^*\omega = e^t \omega$.

For a Liouville domain $(W, \omega = d\lambda)$, consider the
embedding $\Psi: (-\infty, 0] \times \partial W \longrightarrow W$,
$(s, x) \longmapsto \psi_s(x)$. We have:
$\Psi^*\lambda = e^s \lambda_{\partial W}$,
$\Psi^* \omega = d(e^s \lambda_{\partial W})$. Define an almost
complex structure \label{p:J-lambda-1} $J^{\lambda}$ on
$(-\infty, 0] \times \partial W$ as follows. Fix an almost complex
structure $J_{\xi_{\lambda}}$ on $\xi_{\lambda}$ which is compatible
with $\omega|_{\xi_{\lambda}}$.  Denote by
$R_{\lambda_{\partial W}} \in T(\partial W)$ the Reeb vector field
corresponding to $\lambda_{\partial W}$. Define
$J^{\lambda}|_{\xi_{\lambda}} := J_{\xi_{\lambda}}$ and
$J^{\lambda}(\tfrac{\partial}{\partial s}) := R$. Note $J^{\lambda}$
is compatible with $\Psi^*\omega$ and moreover the function
$\phi: (-\infty, 0] \times \partial W \longrightarrow \mathbb{R}$,
$\phi(s,x) := e^s$, is a potential for $\Psi^*\omega$, i.e.~
$\Psi^* \omega = -d d^{J^{\lambda}} \phi$ (in fact we have
$d^{J^{\lambda}} \phi = - e^s\lambda$). In particular, $\phi$ is
$J$-plurisubharmonic (or $J$-convex). Using the map $\Psi$ we can
endow $\textnormal{image}(\Psi)$ with the almost complex structure
$\Psi_*(J^{\lambda})$ which, by abuse of notation, will also be
denoted by $J^{\lambda}$. (Note that in general $J^{\lambda}$ does not
extend from $\textnormal{image}(\Psi)$ to the entire of $W$.)


Sometimes it will be useful to work with the completion
$(\widehat{W}, \widehat{\omega} = d\widehat{\lambda})$ of a compact
Liouville domain $(W, \omega = d\lambda)$. More precisely, set
$$\widehat{W} := W \cup_{\Psi}
\bigl( [-\epsilon,\infty) \times \partial W \bigr),$$ where the gluing
identifies $[-\epsilon, 0] \times \partial W$ with a collar
neighborhood of $\partial W$ in $W$ via the map $\Psi$. The Liouville
form $\widehat{\lambda}$ is defined by extending $\lambda$ from $W$ to
the cylindrical part $[0,\infty) \times \partial W$ by
$\widehat{\lambda} = e^s \lambda_{\partial W}$, where
$s \in [0,\infty)$. We denote the corresponding symplectic structure
by $\widehat{\omega} := d \widehat{\lambda}$.

All the previous structures, like $X_{\lambda}$, $\psi_t$, $\phi$ and
$J^{\lambda}$, extend in an obvious way to the completion. More
specifically, the Liouville vector field $X_{\widehat{\lambda}}$
(defined by
$i_{X_{\widehat{\lambda}}} \widehat{\omega} = \widehat{\lambda}$)
extends $X_{\lambda}$ by $\tfrac{\partial}{\partial s}$ along the
cylindrical part. We denote the flow of $X_{\widehat{\lambda}}$ by
$\widehat{\psi}_t$. Note that this flow is complete (i.e.~exists for
all times $t$, both positive and negative). Next, we extend the almost
complex structure $J^{\lambda}$ from $\textnormal{image\,}\Psi$ to an
almost complex structure $\widehat{J}^{\lambda}$ on
$(\textnormal{image\,}\Psi) \cup_{\Psi} \bigl( [-\epsilon,\infty)
\times \partial W \bigr) \subset \widehat{W}$ by the same recipe
defining $J^{\lambda}$, namely: $\widehat{J}^{\lambda} := J^{\lambda}$
on $\textnormal{image\,}\Psi$, and
$\widehat{J}^{\lambda}_{(s,x)}|_{\xi_{\lambda}} := J_{\xi_{\lambda}}$,
$J^{\lambda}_{(s,x)}(\tfrac{\partial}{\partial s}) :=
R_{\lambda_{\partial W}}$, for every
$(s,x) \in [0,\infty) \times \partial W$ (where here we view
$\xi_{\lambda} \subset T_{(s,x)}( s \times \partial W)$) . Finally
note that the plurisubharmonic function
$\phi: \textnormal{image\,}\Psi \longrightarrow \mathbb{R}$ extends to
the cylindrical part $[0,\infty) \times \partial W$ by
$\widehat{\phi}(s,x) = e^s$ and
$\widehat{\lambda} = -d^{\widehat{J}^{\lambda}} \widehat{\phi}$,
$\widehat{\omega} = -dd^{\widehat{J}^{\lambda}} \widehat{\phi}$.


Another type of symplectic manifolds that we will encounter are Stein
manifolds, which are very much related to the above. By a Stein
manifold we mean a triple $(V, J_V, \varphi)$, where $(V, J_V)$ is an
open complex manifold (with integrable $J_V$) and
$\varphi: V \longrightarrow \mathbb{R}$ is an exhaustion
plurisubharmonic function. Exhaustion means that $\varphi$ is proper
and bounded from below, and plurisubharmonic means that the $2$-form
$\omega_{\varphi} := -dd^{J_V} \varphi$ is compatible with $J_V$
(i.e.~$\omega_{\varphi}(u, J_Vu)>0$, $\forall u$ and
$\omega_{\varphi}(J_V u, J_V v) = \omega_{\varphi}(u, v)$,
$\forall u, v$). Denote $\lambda_{\varphi} := -d^{J_V} \phi$ and for
$R \in \mathbb{R}$,
$V_{\varphi \leq R} := \{ x \in V \mid \varphi(x)\leq R\}$. (Similarly
we have $V_{\varphi < R}$, $V_{\varphi \geq R}$ etc.)  Below we will
implicitly assume that $(V, J_V, \varphi)$ is of finite type, namely
that $\varphi$ has a finite number of critical points.  Note that if
$R$ is a regular value of $\varphi$ then
$(V_{\varphi \leq R}, \omega_{\varphi} = d\lambda_{\varphi})$ is a
compact Liouville domain.

Another variant is symplectic manifolds that are Stein at infinity:
$(V, J_V, \varphi, R_0, \omega)$. Here $V$ is a symplectic manifold,
endowed with a (possibly non-exact) symplectic structure $\omega$.
Next we have $\varphi: V \longrightarrow \mathbb{R}$, an exhaustion
function with finitely many critical points. The parameter
$R_0 \in \mathbb{R}$ is a regular value of $\varphi$, and $J_V$ is an
integrable complex structure defined on $V_{\varphi \geq R_0}$, and
the following holds along $V_{\varphi\geq R_0}$:
$\omega = -dd^{J_V} \varphi$ is compatible with $J_V$.  Thus $\varphi$
is $J_V$-convex on $V_{\varphi\geq R_0}$.

Symplectic manifolds that are Stein at infinity admit a slightly
different variant of completion, which we now briefly recall
(see~\cite{El-Gr:convex, Bi-Ci:Stein, El-Ci:Weinstein-book} for more
details). Let $(V, J_V, R_0, \varphi, \omega )$ be a symplectic
manifold manifold which is Stein at infinity. Let $R \geq R_0$ and
assume that $\textnormal{Crit}(\varphi) \subset V_{\varphi < R}$. Then
there exists a function $\varphi_R: V \longrightarrow \mathbb{R}$ with
the following properties:
\begin{enumerate}
\item $\varphi_R$ is an exhaustion function and
  $V_{\varphi_R \leq R} = V_{\varphi \leq R}$. Moreover,
  $\varphi_R = \varphi$ on $V_{\varphi \leq R}$.
\item $\varphi_R$ has no critical points in $V_{\varphi_R \geq R}$.
\item $\varphi_R$ is plurisubharmonic on $V_{\varphi \geq R_0}$,
  i.e.~$-dd^{J_V} \varphi_R$ is compatible with $J_V$ along
  $V_{\varphi \geq R_0}$.
\item Define the $1$-form $\widehat{\lambda}_R:= -d^{J_V} \varphi_R$
  on $V_{\varphi \geq R_0}$. Define $\widehat{\omega}_R$ on $V$ by
  setting it to be $\omega$ on $V_{\varphi \leq R_0}$ and
  $\widehat{\omega}_R := d \widehat{\lambda}_R$ on
  $V_{\varphi \geq R_0}$. Let $X_{\widehat{\lambda}_R}$ be the
  Liouville vector field, defined along $V_{\varphi \geq R_0}$, by
  $i_{X_{\widehat{\lambda}_R}} \widehat{\omega}_R =
  \widehat{\lambda}_R$. Then the flow
  $\widehat{\psi}^R_t : V_{\varphi \geq R_0} \longrightarrow
  V_{\varphi \geq R_0}$ of $X_{\widehat{\lambda}_R}$ exists for all
  $t \geq 0$.
\end{enumerate}
We will call $(V, J_V, R_0, \varphi_R, \widehat{\omega}_R)$ a
completion of $(V, J_V, R_0, \varphi, \omega)$.

Finally, we will also need the notion of Liouville manifolds that are
Stein at infinity. These are symplectic manifolds that are Stein at
infinity, $(V, J_V, \varphi, R_0, \omega = d\lambda)$, but now we
assume in addition that the symplectic structure $\omega$ is globally
exact with a prescribed primitive $\lambda$. Moreover, $\lambda$ is
assumed to satisfy $\lambda = -d^{J_V} \varphi$ along
$V_{\varphi\geq R_0}$.

Note that, as for the case of Stein manifolds, if $R \geq R_0$ is a
regular value of $\varphi$ then
$(V_{\varphi \leq R}, \omega = d\lambda)$ is a compact Liouville
domain.

Note also that for the completion of Liouville manifolds that are
Stein at infinity, the Liouville vector field
$X_{\widehat{\lambda}_R}$ is defined all over $V$ and moreover, its
flow exists for all $t \in \mathbb{R}$.

\subsubsection{Floer theory} \label{sbsb:HF} We will work here with
Floer homology and singular homology, both taken with coefficients in
$\mathbb{Z}_2$. We will generally omit the $\mathbb{Z}_2$ from the
notation (e.g.~writing $H_*(L)$ for $H_*(L; \mathbb{Z}_2)$). Our
setting is almost identical to~\cite[Chapter~III,
Section~8]{Se:book-fukaya-categ}, with two slight
differences. Firstly, we work with homological conventions rather than
with cohomological ones. Secondly, we work in an ungraded setting.

Let $(V, \omega = d\lambda)$ be an exact symplectic manifold with a
given primitive $\lambda$ for the symplectic structure. We assume
further that this symplectic manifold is of one of the following three
types:
\begin{enumerate}
\item $(V, d\lambda)$ is a compact Liouville
  domain. \label{i:compact-LD}
\item $(V, d\lambda)$ is the completion
  $(\widehat{V'}, \widehat{\omega'} = d\widehat{\lambda'})$ of a
  compact Liouville domain $(V', \omega' =
  d\lambda')$. \label{i:complete-LD}
\item $(V, d\lambda)$ can be endowed with a structure
  $(V, J_V, R_0, \varphi, \omega = d\lambda)$ of a Liouville manifold
  which is Stein at infinity. In that case we also fix the additional
  structures $J_V, \varphi, R_0$. \label{i:stein-infty}
\end{enumerate}
We denote by $\textnormal{Int\,}V$ the interior of $V$. (Note that
only in case~\eqref{i:compact-LD}, we have
$\textnormal{Int\,}V \subsetneqq V$.) Denote by $\mathcal{J}_V$ the
space of $\omega$-compatible almost complex structures on $V$ which
coincide with, $J^{\lambda}$ near the boundary of $V$ in
case~\eqref{i:compact-LD}, or with $\widehat{J}^{\lambda}$ at infinity
in case~\eqref{i:complete-LD}, or coincide with $J_V$ on
$V_{\varphi\geq R}$ for some $R \geq R_0$ in
case~\eqref{i:stein-infty}.

Let $L_0, L_1 \subset \textnormal{Int\,} V$ be two closed exact
Lagrangian submanifolds. (Exactness of a Lagrangian $L$ will be
generally considered with respect to the given Liouville form
$\lambda$. In case we want to emphasize the form with respect to which
$L$ is exact we will call $L$ a $\lambda$-exact Lagrangian.)  We fix
primitive functions $h_{L_i}: L_i \longrightarrow \mathbb{R}$ to
$\lambda|_{L_i}$, $i=0, 1$.

Let $H: [0,1] \times V \longrightarrow \mathbb{R}$ be a Hamiltonian
function. Write $H_t(x) = H(t,x)$. Henceforth we will implicitly
assume that there exists a compact subset
$K \subset \textnormal{Int\,} V$ such that for all $t \in [0,1]$, the
function $H_t$ is constant outside of $K$. The Hamiltonian vector
field $X_t^{H} = X^{H_t}$ of $H$ is given by
$\omega(X_t^H, \, \cdot \,) = -dH_t(\, \cdot \,)$.

Denote by
$\mathcal{P}_{L_0, L_1} = \bigl\{ \gamma:[0,1] \longrightarrow V \mid
\gamma(0) \in L_0, \, \gamma(1) \in L_1\bigr\}$ the space of paths
with end points on $L_0$, $L_1$. The action functional
$\mathcal{A}_{H}: \mathcal{P}_{L_0, L_1} \longrightarrow \mathbb{R}$
is defined as follows:
\begin{equation} \label{eq:action-funct} \mathcal{A}_{H} (\gamma) :=
  \int_0^1 H(t, \gamma(t))dt - \int_0^1 \lambda(\dot{\gamma}(t)) dt +
  h_{L_1}(\gamma(1)) - h_{L_0}(\gamma(0)).
\end{equation}

Denote by
$\mathcal{O}(H) = \mathcal{O}_{L_0, L_1}(H) \subset \mathcal{P}_{L_0,
  L_1}$ the set of Hamiltonian chords with endpoints on $(L_0, L_1)$,
namely the set of orbits $\gamma:[0,1] \longrightarrow W$ of $X_t^H$
with $\gamma(0) \in L_0$, $\gamma(1) \in L_1$.

Let $\mathscr{D} = (H, J)$ be a regular Floer datum, consisting of a
Hamiltonian function $H:[0,1] \times W \longrightarrow \mathbb{R}$ and
a time-dependent almost complex structure $J = \{J_t\}_{t \in [0,1]}$,
with $J_t \in \mathcal{J}_V$ for every $t$. Sometimes we will write
$\mathcal{O}_{L_0,L_1}(\mathscr{D})$ (or $\mathcal{O}(\mathscr{D})$)
for $\mathcal{O}_{L_0, L_1}(H)$.

The negative gradient flow of $\mathcal{A}_H$ (with respect to a
metric on $\mathcal{P}_{L_0,L_1}$ induced by $J$) gives rise to the
Floer equation associated to $\mathscr{D}$:
\begin{equation} \label{eq:floer-eq-1}
  \begin{aligned}
    & u: \mathbb{R} \times [0,1] \longrightarrow M, \quad u(\mathbb{R}
    \times 0) \subset L_0, \; u(\mathbb{R} \times 1) \subset L_1, \\
    & \partial_s u + J_t (u) \partial_t u =
    J_t X_t^H(u), \\
    & E(u) := \int_{-\infty}^{\infty} \int_{0}^{1} \lvert \partial_su
    \rvert^2 dt ds < \infty.
  \end{aligned}
\end{equation}
where $(s,t) \in \mathbb{R} \times [0,1]$. The quantity $E(u)$ in the
last line of~\eqref{eq:floer-eq-1} is the energy of a solution $u$ and
we consider only finite energy solutions. (Note also that the norm
$\lvert \partial_s u \rvert$ in the definition of $E(u)$ is calculated
with respect to the Riemannian metric associated to $\omega$ and
$J_t$.)  Solutions $u$ of~\eqref{eq:floer-eq-1} are also called Floer
trajectories.

For $\gamma_{-}, \gamma_{+} \in \mathcal{O}(H)$ we have the space of
{\em parametrized} Floer trajectories $u$ connecting $\gamma_{-}$ to
$\gamma_+$:
\begin{equation} \label{eq:floer-traj-space-1} \mathcal{M}(\gamma_-,
  \gamma_+; \mathscr{D}) = \Big\{ u \mid u \text{
    solves~\eqref{eq:floer-eq-1} and} \lim_{s \to \pm \infty} u(s,t) =
  \gamma_{\pm}(t) \Big\}.
\end{equation}
Note that $\mathbb{R}$ acts on this space by translations along the
$s$-coordinate. This action is generally free, with the only exception
being $\gamma_- = \gamma_+$ and the stationary solution
$u(s,t) = \gamma_-(t)$ at $\gamma_-$.

Whenever, $\gamma_- \neq \gamma_+$ we denote by
\begin{equation}
  \mathcal{M}^*(\gamma_-,
  \gamma_+; \mathscr{D}) := \mathcal{M}(\gamma_-,
  \gamma_+; \mathscr{D}) \big/ \mathbb{R}
\end{equation}
the quotient space (i.e. the space of non-parametrized solutions).

For a generic choice of Floer datum $\mathscr{D}$ the space
$\mathcal{M}^*(\gamma_-, \gamma_+; \mathscr{D})$ is a smooth manifold
(possibly with several components having different
dimensions). Moreover, its $0$-dimensional component
$\mathcal{M}^*_0(\gamma_{-}, \gamma_{+}; \mathscr{D})$ is compact
hence a finite set.

The Floer complex $CF(L_0, L_1; \mathscr{D})$ is the vector space,
over $\mathbb{Z}_2$, with a basis formed by the set $\mathcal{O}(H)$:
\begin{equation} \label{eq:CF-1} CF(L_0, L_1; \mathscr{D}) =
  \bigoplus_{\gamma \in \mathcal{O}(H)} \mathbb{Z}_2 \gamma.
\end{equation}
Its differential
$d: CF(L_0, L_1; \mathscr{D}) \longrightarrow CF(L_0, L_1;
\mathscr{D})$ is defined by counting solutions of the Floer equation:
\begin{equation} \label{eq:mu-1} d(\gamma_{-}) := \sum_{\gamma_{+} \in
    \mathcal{O}(H)} \#_{\mathbb{Z}_2} \mathcal{M}^*_0(\gamma_{-},
  \gamma_{+}; \mathscr{D}) \gamma_+, \quad \forall \; \gamma_{-} \in
  \mathcal{O}(H),
\end{equation}
and extending linearly over $\mathbb{Z}_2$. The homology of
$CF(L_0, L_1; \mathscr{D})$ is denoted by $HF(L_0,L_1; \mathscr{D})$ -
the Floer homology of $(L_{0},L_{1})$.

The Floer homology is independent of the choice of the Floer datum in
the sense that for every two regular choices of Floer data
$\mathscr{D}$, $\mathscr{D}'$ there is a quasi-isomorphism, canonical
up to chain homotopy,
$\psi_{\mathscr{D}, \mathscr{D}'}: CF(L_0,L_1; \mathscr{D})
\longrightarrow CF(L_0,L_1; \mathscr{D}')$, called a continuation map.
The (now canonical) isomorphisms induced in homology
$H(\psi_{\mathscr{D}, \mathscr{D}'}): HF(L_0,L_1; \mathscr{D})
\longrightarrow HF(L_0,L_1; \mathscr{D}')$ form a directed system and
we can regard the collection of vector spaces
$HF(L_0,L_1; \mathscr{D})$, parametrized by regular Floer data
$\mathscr{D}$, as one vector space and denote it by $HF(L_0,L_1)$.

\subsubsection{PSS and naturality} \label{sbsb:pss-nat} Given a
Hamiltonian function $F: [0,1] \times V \longrightarrow \mathbb{R}$,
denote by $\widebar{F}(t,x) := -F(t, \phi_t^F(x))$ and
$\widehat{F}(t,x) = -F(1-t, x)$. The flows of these functions are
$\phi_t^{\widebar{F}} = (\phi_t^F)^{-1}$ and
$\phi_t^{\widehat{F}} = \phi_{1-t}^F \circ (\phi_1^F)^{-1}$
respectively. Note that both these flows have the same time-$1$ map:
$\phi_1^{\widebar{F}} = \phi_1^{\widehat{F}} = (\phi_1^F)^{-1}$.  For
two Hamiltonian functions
$F, G: [0,1] \times V \longrightarrow \mathbb{R}$, denote by
$G \# F: [0,1] \times V \longrightarrow \mathbb{R}$ the function
$(G \# F)(t,x) = G(t,x) + F(t, (\phi_t^{G})^{-1}(x))$. Its Hamiltonian
flow is $\phi_t^{G \# F} = \phi_t^G \circ \phi_t^F$. Given a Floer
datum $\mathscr{D} = (F, J)$ and a Hamiltonian flow $\phi_t^G$
generated by $G$ we denote by
$\phi^G_* \mathscr{D} := (G \# F, \phi^G_* J)$ the push-forward Floer
datum, where
$(\phi^G_* J)_t := D \phi_t^G \circ J_t \circ (D \phi_t^G)^{-1}$.

Let $L_0, L_1 \subset \textnormal{Int\,} V$ be two exact Lagrangians
and assume that the Floer datum $\mathscr{D} = (F, J)$ is regular. Let
$G$ be another Hamiltonian function. There is a {\em naturality map}
\begin{equation} \label{eq:nat-G}
  \begin{aligned}
    & \mathcal{N}_G : CF(L_0, L_1; \mathscr{D}) \longrightarrow CF(L_,
    \phi_1^G(L_1); \phi^G_* \mathscr{D}), \\
    & \mathcal{N}_G(\gamma)(t):= \phi_t^G \gamma(t), \; \; \forall
    \gamma \in \mathcal{O}_{L_0,L_1}(F).
  \end{aligned}
\end{equation}
The map $\mathcal{N}_G$ is a chain isomorphism.

Consider now a Lagrangian $L'_1$ which is exact isotopic to $L_1$. Fix
a Hamiltonian function $G$ such that $\phi_1^G(L_1) = L'_1$. The map
induced in homology by $\mathcal{N}_G$ is compatible with the
homological maps induced by continuation. Therefore $\mathcal{N}_G$
induces as well defined isomorphism
$HF(L_0, L_1) \longrightarrow HF(L_0, L'_1)$. Moreover, this
isomorphism is independent of the choice of $G$ (among Hamiltonian
functions $G$ with $\phi_1^G(L_1) = L'_1$). We thus obtain a system of
canonical isomorphisms
$\mathcal{N}^{L_0}_{L'_1, L_1}: HF(L_0, L_1) \longrightarrow HF(L_0,
L'_1)$, defined for every pair of exact isotopic Lagrangians
$L_1, L'_1$.  Moreover,
$$\mathcal{N}^{L_0}_{L_1, L_1} = \id, \quad
\mathcal{N}^{L_0}_{L''_1, L'_1} \circ \mathcal{N}^{L_0}_{L'_1, L_1} =
\mathcal{N}^{L_0}_{L''_1, L_1}.$$

\begin{rems} \label{r:N-maps}
  \begin{enumerate}
  \item For the latter statement to hold it is important that the
    Lagrangians are exact, or more generally weakly exact. Indeed, in
    the presence of holomorphic disks (e.g.~for monotone Lagrangians)
    the isomorphisms $\mathcal{N}^{L_0}_{L'_1, L_1}$ might depend on
    the homotopy class of the path between $L_1$ and $L'_1$ inside the
    space of exact Lagrangians.
  \item Denote by
    $*: HF(L_0, L_1) \otimes HF(L_1, L'_1) \longrightarrow HF(L_0,
    L'_1)$ the product induced by the chain level $\mu_2$-operation.
    Then there exists a class $c_{L_1, L'_1} \in HF(L_1, L'_1)$ such
    that $\mathcal{N}^{L_0}_{L'_1, L_1}(a) = a * c_{L_1, L'_1}$ for
    every $a \in HF(L_0, L_1)$. In fact,
    $c_{L_1, L'_1} = \mathcal{N}^{L_1}_{L'_1, L_1}(e_{L_1})$, where
    $e_{L_1} \in HF(L_1, L_1)$ is the unity.
  \end{enumerate}
\end{rems}

Similarly to the maps $\mathcal{N}^{L_0}_{L'_1, L_1}$ we also have
canonical isomorphisms
$\mathcal{N}^{L'_0, L_0}_{L_1}: HF(L_0, L_1) \longrightarrow
HF(L_0',L_1)$, defined in an analogous way.

We now turn to the PSS isomorphism. Let
$L \subset \textnormal{Int\,} V$ be an exact Lagrangian. Let
$\mathfrak{m} = (f, \rho)$ be a Morse datum, consisting of a Morse
function $f: L \longrightarrow \mathbb{R}$ and a Riemannian metric
$\rho$ on $L$. Denote by $\mathcal{C}(L; \mathfrak{m})$ the Morse
complex associated to $\mathfrak{m}$. Let $\mathscr{D} = (H, J)$ be a
regular Floer datum for the pair $(L,L)$. The PSS map is a
quasi-isomorphism
\begin{equation} \label{eq:PSS} PSS_{\frak{m}, \mathscr{D}}:
  \mathcal{C}(L; \mathfrak{m}) \longrightarrow CF(L,L; \mathscr{D})
\end{equation}
canonical up to chain homotopy. Moreover, the maps
$PSS_{\mathfrak{m}, \mathscr{D}}$, defined for different
$\mathfrak{m}$, $\mathscr{D}$, are compatible with the corresponding
continuation maps up to chain homotopy. Consequently, the isomorphism
induced by $PSS$ in homology
$$PSS: H_*(L) \longrightarrow HF(L,L),$$ which we also denote by $PSS$, is
independent of the data $\mathfrak{m}$, $\mathscr{D}$. Moreover, this
map is multiplicative (with respect to the intersection product on
$H_*(L)$ and the triangle product induced by $\mu_2$ on $HF(L,L)$) and
it sends the fundamental class $[L]$ to the unit $e_L \in HF(L,L)$. We
refer the reader to~\cite{Kat-Mil:PSS, Alb:PSS} for the definition and
properties of this map.

\begin{rems} \label{r:exact-lags}
  \begin{enumerate}
  \item \label{i:exact-monod} Let
    $L_0, L_1 \subset \textnormal{Int\,} V$ be two exact Lagrangians
    that are exact isotopic. Choose any exact isotopy
    $\phi_t: L_0 \longrightarrow \textnormal{Int\,} V$, $t \in [0,1]$,
    with $\phi_0 = \text{inclusion of } L_0 \subset V$ and
    $\phi_1(L_0) = L_1$. By a result of
    Hu-Lalonde-Leclercq~\cite{HLL:monodromy} the map
    ${\phi_1}_* : H_*(L_0; \mathbb{Z}_2) \longrightarrow H_*(L_1;
    \mathbb{Z}_2)$, induced in homology by $\phi_1$, is independent of
    the choice of the isotopy $\{ \phi_t \}$. Therefore there is a
    canonical map
    $\phi_*: H_*(L_0; \mathbb{Z}_2) \longrightarrow H_*(L_1;
    \mathbb{Z}_2)$ between any two exact isotopic exact Lagrangians in
    $\textnormal{Int\,} V$. The map $\phi_*$ is compatible with Floer
    theory in the following sense. First note that if $\{\phi_t\}$ is
    an exact isotopy as above its time-$1$ map induces a map in Floer
    homology
    ${\phi_1}^{HF}: HF(L_0, L_0) \longrightarrow HF(L_1,
    L_1)$. Moreover, this map is independent of the choice of the
    isotopy (in fact,
    $\phi_1^{HF} = \mathcal{N}_{L_1}^{L_1, L_0} \circ
    \mathcal{N}^{L_0}_{L_1, L_0}$). Write $\phi^{HF} :=
    \phi_1^{HF}$. Standard arguments then show that $\phi_*$ equals
    the composition
    \begin{equation*}
      H_*(L_0;\mathbb{Z}_2) \xrightarrow{PSS} HF(L_0, L_0)
      \xrightarrow{ \phi^{HF}} HF(L_1,L_1) \xrightarrow{ PSS^{-1}}
      H_*(L_1;\mathbb{Z}_2).
    \end{equation*}
  \item \label{i:FSS-exact-lags} In general the space of exact
    Lagrangians in $V$ might be disconnected (and even contain
    Lagrangians of different topological types). However, in certain
    situation this is not expected to be so. For example, a version of
    the nearby Lagrangian conjecture asserts that if $V = T^*(N)$ is
    the cotangent bundle of a closed manifold $N$ then all exact
    Lagrangians are exact isotopic to the zero-section. While this is
    still open in general, a result of
    Fukaya-Seidel-Smith~\cite{FSS:ex, FSS:cbndl} and independently of
    Nadler~\cite{Na:cotangent}, says that under mild topological
    assumptions on $N$ the following holds. Every exact Lagrangian
    $L \subset T^*(N)$ is canonically isomorphic, when viewed as an
    objects in the (compact) derived Fukaya category of $T^*(N)$, to
    the zero-section. Moreover, this isomorphism induces the same map
    $HF(L,L) \longrightarrow HF(N,N)$ as the one induced by the
    projection $\textnormal{pr}: T^*(N) \longrightarrow N$ on homology
    $H_*(L) \longrightarrow H_*(N)$, under the canonical
    identifications $HF(L,L) \cong H_*(L)$ and $HF(N,N) \cong H_*(N)$.
  \end{enumerate}
\end{rems}

\subsubsection{Action filtrations and Floer persistent
  homology} \label{sbsb:action-filtrations} We begin by recalling the
fundamentals of filtered Lagrangian Floer theory in the exact setting.
Much of the general theory has been developed
in~\cite{Oh:symp-top-act-1, Oh:symp-top-act-2, Leclercq:spectral,
  Le-Za:spec-inv, DKM:spectral, KMN:spec}, though in somewhat
different frameworks like monotone (and weakly exact) Lagrangians. The
essence however remains the same and a considerable part of these
papers applies with minor changes to the exact case too.

In order to define the action functional and its induced filtrations
in Floer theory we need to endow each exact Lagrangian $L$ with a
primitive $h_L: L \longrightarrow \mathbb{R}$ of the exact form
$\lambda|_L$. We will refer to $h_L$ as a {\em marking of $L$} and to
the pair $(L, h_L)$ as a {\em marked Lagrangian}. However, for
simplicity of notation we will often continue to denote marked
Lagrangians by a single letter, e.g.~$L$, with the understanding that
the primitive $h_L$ has been fixed.

Let $L_0, L_1 \subset \textnormal{Int\,} V$ be two marked Lagrangians.
Let $\mathscr{D} = (H, J)$ be a regular Floer datum for
$(L_0,L_1)$. For $\alpha \in \mathbb{R}$ denote
\begin{equation} \label{eq:CF-alpha}
  CF^{\leq \alpha}(L_0, L_1; \mathscr{D}) :=
  \bigoplus_{\gamma \in \mathcal{O}(H), \, \mathcal{A}_H(\gamma) \leq
    \alpha} \mathbb{Z}_2 \gamma.
\end{equation}
For convenience we extend $\mathcal{A}_H$ to all elements of
$CF(L_0, L_1; \mathscr{D})$ by defining it on
$\lambda = \sum_{i=1}^k a_i \gamma_i$, $a_i \in \mathbb{Z}_2$, to be:
$$\mathcal{A}_H(\lambda) =
\max \{ \mathcal{A}_H(\gamma_i) \mid a_i \neq 0\} = \inf \bigl\{
\alpha \mid \lambda \in CF^{\leq \alpha}(L_0,L_1; \mathscr{D})
\bigr\}.$$ Here we use the convention that $\max \emptyset = -\infty$,
so that $\mathcal{A}_H(0) = -\infty$.

The subspaces $CF^{\leq \alpha} \subset CF$ are in fact subcomplexes.
This is so because for every Floer trajectory
$u \in \mathcal{M}(\gamma_{-}, \gamma_{+}; \mathscr{D})$ we have the
following action-energy relation:
$\mathcal{A}_H(\gamma_{+}) = \mathcal{A}_{H}(\gamma_{-}) - E(u) \leq
\mathcal{A}_{H}(\gamma_{-})$. Therefore
$\mathcal{A}_H(d\gamma) \leq \mathcal{A}_H(\gamma)$, hence
$d (CF^{\leq \alpha}(L_0, L_1; \mathscr{D})) \subset CF^{\leq
  \alpha}(L_0, L_1; \mathscr{D})$.

We write
$HF^{\leq \alpha}(L_0, L_1; \mathscr{D}) := H_*(CF^{\leq \alpha}(L_0,
L_1; \mathscr{D}))$ and for $\alpha \leq \beta \leq \infty$ we denote
by
$i_{\beta, \alpha} : HF^{\leq \alpha}(L_0, L_1; \mathscr{D})
\longrightarrow HF^{\leq \beta}(L_0, L_1; \mathscr{D})$ the map
induced by the inclusion
$CF^{\leq \alpha}(L_0, L_1; \mathscr{D}) \subset CF^{\leq \beta}(L_0,
L_1; \mathscr{D})$. For $\beta = \infty$ we abbreviate
$i^{\alpha} := i^{\infty, \alpha}$.

The homologies $HF^{\leq \alpha}(L_0,L_1; \mathscr{D})$,
$\alpha \in \mathbb{R}$, and the maps $i_{\beta, \alpha}$,
$\alpha \leq \beta$, fit together into a persistence module which we
denote by $HF^{\leq \bullet}(L_0,L_1; \mathscr{D})$ and call the Floer
persistent homology.

Next, we briefly discuss to what extent the Floer persistent homology
depends on the Floer data. The continuation maps
$\psi_{\mathscr{D}', \mathscr{D}}$ do not preserve action-filtrations
in general, hence there is no meaning to write
$H(CF^{\leq \alpha}(L_0,L_1))$ without specifying the Floer datum.
Nevertheless, if $\mathscr{D}' = (H, J')$ and
$\mathscr{D}'' = (H, J'')$ are two regular Floer data with the same
Hamiltonian function $H$, then one can choose the continuation map
$\psi_{\mathscr{D}'', \mathscr{D}'} : CF(L_0, L_1; \mathscr{D}')
\longrightarrow CF(L_0, L_1; \mathscr{D}'')$ to be action
preserving. Moreover, for such Floer data, the chain homotopies
between
$\psi_{\mathscr{D}', \mathscr{D}''} \circ \psi_{\mathscr{D}'',
  \mathscr{D}'}$ and $\text{id}$ can be also chosen to preserve
action. It follows that $\psi_{\mathscr{D}'', \mathscr{D}'}$ induces
an isomorphism between the persistence modules
$HF^{\leq \bullet}(L_0, L_1; \mathscr{D}')$ and
$HF^{\leq \bullet}(L_0, L_1; \mathscr{D}'')$. Moreover, standard
arguments imply that this isomorphism is canonical (in the sense that
there is a preferred such isomorphism). Thus the Floer persistent
homology of $(L_0, L_1)$ depends only on the Hamiltonian function in
the Floer data, hence will sometimes be denoted by
$HF^{\leq \bullet}(L_0, L_1;H)$. In case $L_0 \pitchfork L_1$ we can
take the Hamiltonian function to be $0$, and the Floer persistent
homology using this choice will be abbreviated as
$HF^{\leq \bullet}(L_0, L_1)$.

The persistence modules $HF^{\leq \bullet}(L_0, L_1; \mathscr{D})$
give rise to a variety of numerical invariants. The most important for
us will be spectral invariants and boundary depth. 

Given $a \in HF(L_0, L_1; \mathscr{D})$ we denote by
$\sigma(a; L_0, L_1; \mathscr{D})$ the spectral invariant of $a$,
defined by the recipe in~\eqref{eq:sig-a} of~\S\ref{sb:filtered} for
the chain complex $CF(L_0, L_1; \mathscr{D})$. By the preceding
discussion the spectral invariants $\sigma(a; L_0, L_1; (H,J))$ as
well as boundary depth $\beta(CF(L_0, L_1; (H,J))$ do not depend on
$J$, hence we will sometimes denote them by $\sigma(a; L_0, L_1; H)$
and $\beta(CF(L_0, L_1; H))$ respectively.

Next we discuss the version of spectral invariants involved in the
definition of the spectral metric, namely $c(a; L_0, L_1)$, where
$L_0 \subset \textnormal{Int\,} V$ is a marked exact Lagrangian,
$a \in H_*(L_0)$, and $L_1 \subset \textnormal{Int\,} V$ is another
marked Lagrangian which is exact isotopic to $L_0$. (Here the marking
on $L_1$ is arbitrary and is not assumed to be related in any way to
the given marking of $L_0$ via any isotopy going from $L_0$ to $L_1$.)
Consider the following composition of isomorphisms
$$H_*(L_0) \xrightarrow{\; PSS \;} HF(L_0, L_0)
\xrightarrow{\; \mathcal{N}^{L_0}_{L_1, L_0} \;} HF(L_0, L_1).$$
Assume first that $L_1$ intersects $L_0$ transversely. Choose an
almost complex structure $J$ such that the Floer datum $(0, J)$ is
regular. Consider the chain complex $CF(L_0, L_1; (0, J))$ endowed
with the action filtration, as defined
at~\eqref{eq:CF-alpha}. Consider also the class
$\mathcal{N}^{L_0}_{L_0, L_1} \circ PSS (a)$ viewed as an element of
$H_*(CF(L_0, L_1; (0, J))) = HF(L_0,L_1)$. We then define
\begin{equation} \label{eq:ca} c(a; L_0, L_1) := \sigma
  \bigl(\mathcal{N}^{L_0}_{L_1, L_0} \circ PSS (a); L_0, L_1; 0
  \bigr).
\end{equation}
In case $L_0$ and $L_1$ do not intersect transversely, we
define
$$c(a; L_0, L_1) = \lim_{\|H\| \to 0} \sigma \bigl(\mathcal{N}^{L_0}_{L_1,
  L_0} \circ PSS (a); L_0, L_1; H \bigr),$$ where
$\|H\| := \int_0^1 \bigl(\max_{x \in V} H(t, x) - \min_{x \in V}
H(t,x)\bigr) dt$, and $\|H\| \to 0$ through Hamiltonian functions for
which $\phi_1^H(L_0) \pitchfork L_1$. The fact that the limit exists
and is finite follows from Lipschitz continuity of the spectral
invariants $\sigma$ with respect to the Hofer norm (see
e.g.~\cite{Leclercq:spectral}).

Finally, given an exact Lagrangian $L \subset \textnormal{Int\,} V$ we
define the spectral distance $\gamma$ on the space
$\mathcal{L}_{\text{ex}, L}(\textnormal{Int\,} V)$ of Lagrangians in
$\textnormal{Int\,} V$ which are exact isotopic to $L$ by
\begin{equation} \label{eq:gamma-dist-def} \gamma(L_0, L_1) = c([L_0];
  L_0, L_1) - c([\textnormal{pt}]; L_0, L_1), \quad \, \forall L_0,
  L_1 \in \mathcal{L}_{\text{ex}, L}(\textnormal{Int\,}V).
\end{equation}

\begin{rems} \label{r:spec-choices}
  \begin{enumerate}
  \item The primitives $h_{L_i}: L_i \longrightarrow \mathbb{R}$ for
    the exact $1$-forms $\lambda|_{L_i}$, $i=1,2$, are uniquely
    determined only up to additions of constants. Similarly, one can
    add to the Hamiltonian function $H$ a (time dependent) constant
    $C(t)$. Different such choices have no effect on Floer complex
    $CF(L_0, L_1; \mathscr{D})$ and its homology, but they do add a
    constant to the action functional, hence shift the filtration on
    $CF(L_0, L_1; \mathscr{D})$ by an overall constant. Consequently,
    the spectral numbers $\sigma(a; L_0,L_1; H)$ and $c(a; L_0, L_1)$
    get shifted by a constant which is independent of $a$. Let
    $a', a'' \in HF(L_0, L_1)$, $b', b'' \in H_*(L_0)$. It follows
    that each of the differences
    $$\sigma(a''; L_0, L_1; H) - \sigma(a'; L_0, L_1; H), \quad
    c(b''; L_0, L_1) - c(b'; L_0, L_1)$$ is {\em independent} of the
    preceding choices. In particular, the spectral distance
    $\gamma(L_0, L_1)$ is independent of any choice of marking on
    $L_0$ and $L_1$.
  \item \label{i:spec-lambda} The action functional and the spectral
    invariants depend on the choice $\lambda$ of the Liouville
    form. However, altering $\lambda$ by an exact $1$-form has no
    effect on these quantities. More specifically, let
    $f: V \longrightarrow \mathbb{R}$ be a smooth function and
    consider $\lambda' = \lambda + df$. The latter is also a primitive
    of the symplectic form $\omega$.

    Clearly, a Lagrangian in $V$ is $\lambda$-exact if and only if it
    is $\lambda'$-exact. Let $L_0, L_1 \subset V$ be two
    $\lambda$-exact Lagrangians and fix primitives
    $h_{L_i}: L_i \longrightarrow \mathbb{R}$ for $\lambda|_{L_i}$,
    $i=0,1$. Then $h'_{L_i} := h_{L_i} + f|_{L_i}$ is a primitive for
    $\lambda'|_{L_i}$. Denote by
    $\mathcal{A}'_{H}: \mathcal{P}_{L_0,L_1} \longrightarrow
    \mathbb{R}$ the action functional defined using $\lambda'$ and the
    primitives $h'_{L_i}$, and by $\mathcal{A}_H$ the one defined
    using $\lambda$ and the $h_{L_i}$'s.  A simple calculation shows
    that $\mathcal{A}'_H = \mathcal{A}_H$.  It follows that the
    spectral invariants $\sigma$ and $c$ remain the same when
    replacing $\lambda$ by $\lambda'$ (provided we use the primitives
    $h'_{L_i}$ as above). Consequently, the spectral metric $\gamma$
    remains unchanged too (the latter does not even depend on the
    choices of the primitive functions $h_{L_i}$ or $h'_{L_i}$).
  \item \label{i:extend-metric} In case $V = T^*(N)$ is the cotangent
    bundle of a closed manifold $N$, one can extend the definition of
    the spectral invariants $c(a; L_0, L_1)$ as well as the spectral
    metric $\gamma(L_0, L_1)$ to arbitrary pairs of exact Lagrangians
    (i.e.~including also pairs that are, hypothetically, not isotopic
    one to the other). This follows from
    point~(\ref{i:FSS-exact-lags}) of Remark~\ref{r:exact-lags}.
  \end{enumerate}
\end{rems}

Another source of numerical invariants comes from the barcode
$\mathcal{B}(HF^{\leq \bullet}(L_0, L_1; \mathscr{D}))$ of the
persistence module $HF^{\leq \bullet}(L_0, L_1; \mathscr{D})$,
see~\cite{PRSZ:top-pers} for the definition. Of main interest for our
considerations is the boundary depth $\beta(L_0, L_1; \mathscr{D})$,
which by definition is the length of the longest finite bar in the
barcode $\mathcal{B}(HF^{\leq \bullet}(L_0, L_1; \mathscr{D}))$. We
will discuss this invariant in more detail
in~\S\ref{subs:filtered-complexes}, and give alternative equivalent
definitions of it.

\subsection{Weakly filtered Fukaya categories}
\label{sb:wf-f} Occasionally it will be convenient to view all exact
Lagrangian submanifold as objects in a Fukaya category, taking into
account action filtrations.

Denote by $\fuk(V)$ the Fukaya category whose objects are the closed
marked Lagrangian submanifolds $L \subset V$ (see the beginning
of~\S\ref{sbsb:action-filtrations} for the definition).  Note that
each underlying Lagrangian appears in this category with all its
possible markings. $\fuk(V)$ is an $A_{\infty}$-category whose
realization requires additional auxiliary structures, namely Floer
data for all pairs of objects as well as coherent perturbation data
for every tuple of objects. We will suppress these choices from the
notation, whenever these choices are clear (or irrelevant). We refer
to~\cite{Se:book-fukaya-categ} for the foundations of Fukaya
categories. In contrast to this (and most) references on the subject,
our Fukaya categories (and all Floer complexes in general) will be
ungraded.

The Fukaya category $\fuk(V)$ has the structure of a so called weakly
filtered $A_{\infty}$-category. This means that
$\hom_{\fuk(V)}(L_0, L_1)=CF(L_0,L_1)$ between every pair of objects
$(L_0, L_1)$ is a filtered chain complex, and moreover each of the
higher order operations $\mu_d$, $d \geq 2$, preserves these
filtrations up to a uniformly bounded error (i.e.~the error for
$\mu_d$ depends only on $d$, and not on the objects involved in it).
We refer the reader to~\cite[~\S2]{Bi-Co-Sh:lshadows-long} for more
details on this theory.

\subsection{Local and global Floer theory} \label{sb:Fl-subdomains}
Let $(V, J_V, \varphi, R_0, \omega = d\lambda)$ be a Liouville
manifold which is Stein at infinity. Let $W_0 \subset V$ be a compact
Liouville subdomain, endowed with the structures $\lambda$ and
$\omega$ coming from $V$. Let
$L_0, L_1 \subset \textnormal{Int\,} W_0$ be two closed marked
$\lambda$-exact Lagrangian submanifolds. Consider Hamiltonian
functions $H: [0,1] \times W_0 \longrightarrow \mathbb{R}$, compactly
supported in $[0,1] \times \textnormal{Int\,}W_0$, such that
$\phi_1^H(L_0) \pitchfork L_1$. We will view these also as Hamiltonian
functions on $V$ by extending them to be $0$ outside $W_0$.

The following proposition compares the local and global Floer
invariants of $(L_0, L_1)$. It says that the Floer homologies as well
as filtered numerical invariants of $(L_0, L_1; H)$, when viewed
either in $W_0$ (``local'') or in $V$ (``global''), coincide.

\begin{prop} \label{p:loc-glob} There exist isomorphisms of
  persistence modules
  $$j^{\leq \bullet} :
  HF^{\leq \bullet} \bigl(L_0, L_1; H; (W_0, \omega=d\lambda)\bigr)
  \longrightarrow HF^{\leq \bullet}\bigl(L_0, L_1; H; (V,
  \omega=d\lambda)\bigr)$$ defined for every pair of closed marked
  Lagrangians $(L_0, L_1)$ and $H$ as above. Moreover, the
  corresponding isomorphisms
  $j:= j^{\leq \infty}: HF \bigl(L_0, L_1; (W_0,
  \omega=d\lambda)\bigr) \longrightarrow HF\bigl(L_0, L_1; (V,
  \omega=d\lambda)\bigr)$ on the total homologies are independent of
  $H$ and have the following further properties:
  \begin{enumerate}
  \item \label{i:loc-glob-1} They are compatible with the triangle
    products.
  \item \label{i:loc-glob-2} They are compatible with the naturality
    maps $\mathcal{N}^{L_0}_{L'_1, L_1}$ from~\S\ref{sbsb:pss-nat} (in
    case $L'_1$ and $L_1$ are exact-isotopic) as well as with PSS (in
    case $L_0=L_1$).
  \item \label{i:loc-glob-3} They preserve spectral invariants, namely
    $$\sigma(j(a); L_0,L_1; H; (V, \lambda)) = \sigma(a; L_0,L_1; H; 
    (W_0, \lambda)), \; \forall a \in HF(L_0,L_1; (W_0, \omega)).$$
  \end{enumerate}
\end{prop}

\begin{rem}
  Proposition~\ref{p:loc-glob} does not hold without the assumption
  that $L_0, L_1$ are exact. For example, take $L_0=L_1$ to be a
  circle in $V=\mathbb{R}^2$ endowed with the standard symplectic
  structure $\omega_{\textnormal{std}}$, and let $W_0$ be a small
  tubular neighborhood of this circle. Then
  $HF(L_0,L_1; W_0, \omega_{\textnormal{std}}) \cong H_*(S^1)$ but
  $HF(L_0,L_1; V, \omega_{\textnormal{std}}) = 0$.
\end{rem}

\begin{proof}[Proof of Proposition~\ref{p:loc-glob}]
  The main idea in the proof is based on a rescaling (or shrinking)
  argument from~\cite[Section~5]{FSS:ex} which we adapt here to our
  setting.

  We will assume without loss of generality that $L_0 \pitchfork L_1$
  and that $H \equiv 0$. This simplifies notation and the proof of
  the general case is very similar to the one we will present below.
  
  Fix $R>R_0$ such that $R> \max_{\overline{W}_0} \varphi$ (so that
  $\overline{W}_0 \subset V_{\varphi < R}$). Consider the completion
  $(V, J_V, R_0, \varphi_R, \widehat{\omega}_R = d
  \widehat{\lambda}_R)$ of
  $(V, J_V, R_0, \varphi, \omega = d\lambda)$ as described
  in~\S\ref{sbsb:LiSt}. Put $\lambda_0 := \lambda|_{\overline{W}_0}$.

  Denote by $\widehat{\psi}_t: V \longrightarrow V$ the Liouville
  flow corresponding to the completion, and by
  $\widehat{\Psi}: \mathbb{R} \times \partial W_0 \longrightarrow V$
  the embedding $\Psi(s,x) := \widehat{\psi}_s(x)$. Note that
  $\widehat{\psi}_t|_{W_0} = \psi_t|_{W_0}$ for every $t\leq 0$,
  where $\psi_t$ is the Liouville flow corresponding to the
  uncompleted Liouville manifold. For an interval
  $I \subset \mathbb{R}$ we write
  $\mathcal{N}(I) := \widehat{\Psi}(I \times \partial W_0) \subset
  V$.

  Recall from~\S\ref{sbsb:LiSt} the model almost complex structure
  $\widehat{J}^{\lambda_0}$ on $\mathbb{R} \times \partial
  W_0$. Consider now its push forward
  $\widehat{\Psi}_*\widehat{J}^{\lambda_0}$ defined on
  $\mathcal{N}(\mathbb{R})$. Slightly abusing notation we will
  continue to denote this almost complex structure by
  $\widehat{J}^{\lambda_0}$. Note that $\mathcal{N}(\mathbb{R})$ is
  invariant under the flow $\widehat{\Psi}_t$, and moreover
  $\widehat{\Psi}_t$ is $\widehat{J}^{\lambda_0}$-holomorphic along
  $\mathcal{N}(\mathbb{R})$.

  Fix $\delta>0$ small enough such that
  $L_0, L_1 \subset W_0 \setminus \mathcal{N}([-\delta,0])$. For every
  $T>0$, consider the space $\mathcal{J}_{(T)}$ of almost complex
  structures $J$ on $V$ that have the following properties:
  \begin{enumerate}
  \item $J$ is compatible with $\widehat{\omega}_R$.
  \item $J = \widehat{J}^{\lambda_0}$ on $\mathcal{N}([-\delta, T])$.
  \item $J = J_V$ at infinity.
  \end{enumerate}
  Denote the space of time-dependent almost complex structure
  $J = \{J_t\}_{t \in [0,1]}$ with $J_t \in \mathcal{J}_{(T)}$ for
  every $t$, by $\mathcal{J}_{(T)}^{[0,1]}$.

  \begin{lem} \label{p:T_0-im-u} There exists $T_0>0$ such that the
    following holds for every $T \geq T_0$: for every regular Floer
    datum $\mathscr{D} = (0,J)$ with $J \in \mathcal{J}_{(T)}^{[0,1]}$
    and every Floer strip
    $u: \mathbb{R} \times [0,1] \longrightarrow V$ corresponding to
    $(L_0, L_1; \mathscr{D})$ we have
    $\textnormal{image\,}u \subset W_0$.
  \end{lem}
  \begin{proof}[Proof of Lemma~\ref{p:T_0-im-u}] Consider the
    Lagrangian submanifolds $L_0^{-T}:= \psi_{-T}(L_0)$,
    $L_1^{-T}:= \psi_{-T}(L_1)$ of $W_0$. Note that $L_0^{-T}$,
    $L_1^{-T}$ are both $\lambda$-exact and
    $L_0^{-T} \pitchfork L_1^{-T}$. For $x \in L_0 \cap L_1$ write
    $x_{-T} := \psi_{-T}(x)$. Denote by $\mathcal{A}^{(L_0, L_1)}$ and
    by $\mathcal{A}^{(L^{-T}_0, L^{-T}_1)}$ the action functionals of
    $(L_0,L_1)$ and of $(L_0^{-T}, L_1^{-T})$ respectively, both
    defined with the Hamiltonian perturbation term $H \equiv 0$.  A
    simple calculation shows that
    \begin{equation} \label{eq:AT-A0} \mathcal{A}^{(L^{-T}_0,
        L^{-T}_1)}(x_{-T}) = e^{-T} \mathcal{A}^{(L_0, L_1)}(x).
    \end{equation}
    Let $u: \mathbb{R} \times [0,1] \longrightarrow \mathbb{R}$ be a
    Floer strip associated to $(L_0, L_1; (0,J))$ with
    $J \in \mathcal{J}_{(T)}^{[0,1]}$. Put
    $v_{-T} := \widehat{\psi}_{-T} \circ u$. Then $v_{-T}$ is a Floer
    strip corresponding to
    $(L_0^{-T}, L_1^{-T}; (0, (\widehat{\psi}_{-T})_*J)$. Note that
    $(\widehat{\psi}_{-T})_*J$ is compatible with
    $\widehat{\omega}_R$. Moreover, by the definition of
    $\mathcal{J}_{(T)}$ we have $(\widehat{\psi}_T)_*J = J$ on
    $\mathcal{N}([-\delta-T, 0])$. Recall also that by definition
    $J \equiv J^{\lambda_0}$ on $\mathcal{N}([\delta, 0])$.

    Denote the energy of Floer strips by $E$. We have:
    \begin{equation} \label{eq:E-T} E(v_{-T}) = e^{-T} E(u) \leq
      e^{-T} \bigl( \max_{x \in L_0 \cap L_1}
      \mathcal{A}^{(L_0,L_1)}(x) - \min_{y \in L_0 \cap L_1}
      \mathcal{A}^{(L_0, L_1)} \bigr).
    \end{equation}

    By a standard energy-length (a.k.a.~monotonicity) estimate for
    pseudo-holomorphic curves (see e.g.~\cite[Section~5.a]{FSS:ex}) we
    have that $\textnormal{image\,} v_{-T} \subset W_0$ provided that
    the right-hand side of~\eqref{eq:E-T} is small enough, which in
    turn can be assured by taking $T$ to be large enough.

    Now
    $L_0^{-T}, L_1^{-T} \subset W_0 \setminus \mathcal{N}([-\delta-T,
    0])$, hence by the maximum principle (applied to the
    $J^{\lambda_0}$-convex function
    $\phi: \mathcal{N}(\mathbb{R}) \longrightarrow \mathbb{R}$,
    $\phi(s,x) = e^s$) we in fact have:
    $$\textnormal{image\,} v_{-T} \subset W_0 \setminus
    \mathcal{N}([-\delta-T, 0]).$$ It follows that
    $u = \widehat{\psi}_T \circ v_{-T}$ has its image inside
    $W_0 \setminus \mathcal{N}([-\delta, 0]) \subset W_0$. This
    concludes the proof of Lemma~\ref{p:T_0-im-u}.
  \end{proof}

  We proceed now with the proof of Proposition~\ref{p:loc-glob}.  Fix
  $J^{\lambda_0}$ on $\mathcal{N}([-\delta, 0])$. Consider Floer data
  of the type $\mathscr{D} = (H \equiv 0, J)$ with
  $J \in \mathcal{J}_{(T)}^{[0,1]}$. By standard transversality
  arguments, for every $T>0$ there exists $J$ as above which makes
  $\mathscr{D}$ regular. Lemma~\ref{p:T_0-im-u} implies that
  there exists $T_0>0$ such that for every $T \geq T_0$ and every
  $J \in \mathcal{J}_{(T)}^{[0,1]}$ with $(0,J)$ regular, the identity
  map
  \begin{equation} \label{eq:i-CF-W_0-V} i: CF \bigl(L_0, L_1; (0,
    J|_{W_0}); (W_0, \omega = d\lambda)\bigr) \longrightarrow CF
    \bigl(L_0, L_1; (0,J); (V, \widehat{\omega}_R=d
    \widehat{\lambda}_R)\bigr)
  \end{equation}
  is a chain map. Clearly $i$ preserves action, hence induces an
  isomorphism of persistence modules
  $$i^{\leq \bullet} :
  HF^{\leq \bullet}(L_0, L_1; H=0; (W_0, \omega)) \longrightarrow
  HF^{\leq \bullet}(L_0, L_1; H=0; (V, \widehat{\omega}_R)).$$
  Finally, note that by the maximum principle the persistence modules
  $HF^{\leq \bullet}(L_0, L_1; H=0; (V, \widehat{\omega}_R= d
  \widehat{\lambda}_R))$ and
  $HF^{\leq \bullet}(L_0, L_1; H=0; (V, \omega=d\lambda))$
  coincide. Thus the isomorphism $i^{\leq \bullet}$ induces the
  isomorphism $j^{\leq \bullet}$ claimed by the proposition. It
  implies also the statement at point~(\ref{i:loc-glob-3}).

  As mentioned at the beginning of the proof, the arguments above can
  be easily adapted to the case of Floer data of the type
  $\mathscr{D} = (H,J)$ with $J \in \mathcal{J}_{(T)}^{[0,1]}$ and
  $H: [0,1] \times V \longrightarrow \mathbb{R}$ compactly supported
  inside $[0,1] \times (W_0 \setminus \mathcal{N}([-\delta, 0])$.

  Moreover, very similar arguments to the above imply that if
  $L_0, \cdots, L_d \subset W_0 \setminus \mathcal{N}([-\delta, 0])$
  are exact Lagrangians then there exists $T_0>0$ such that for every
  disk $S$ (with $(d+1)$ boundary punctures) and for every choice of
  perturbation data $\mathscr{D}_{L_0, \ldots, L_d} = (K,J)$ with
  Hamiltonian term $K$ such that $K_z$ is compactly supported in
  $W_0 \setminus \mathcal{N}([-\delta, 0])$ for every $z \in S$ and
  with almost complex structure $J = \{J\}_{z \in S}$ such that
  $J_z \in \mathcal{J}_{(T)}$ for every $z \in S$, the following
  holds: every Floer polygon $u: S \longrightarrow V$ corresponding to
  $(L_0, \ldots, L_d; (K,J))$ satisfies
  $\textnormal{image\,}u \subset W_0$.

  The statements at points~(\ref{i:loc-glob-1})
  and~(\ref{i:loc-glob-2}) readily follow.
\end{proof}


\section{Cotangent bundles and real Lefschetz fibration}
\label{s:lef}

\subsection{Real Lefschetz fibrations} \label{sb:rlef} In this paper
we will adopt the following definition of Lefschetz fibrations,
essentially as in~\cite{FSS:ex}. By a Lefschetz fibration
$\pi: E \longrightarrow \mathbb{C}$ we mean a symplectic manifold $E$,
endowed with a symplectic structure $\omega_E$ as well as an
$\omega_E$-compatible almost complex structure $J_E$ such that the
following holds:
\begin{enumerate}
\item $\pi$ is $(J_E, i)$-holomorphic and has a finite number of
  critical points. Moreover, we assume that every critical value of
  $\pi$ corresponds to precisely one critical point of $\pi$. We
  denote the set of critical points of $\pi$ by
  $\textnormal{Crit}(\pi)$ and by
  $\textnormal{Critv}(\pi) \subset \mathbb{C}$ the set of critical
  values of $\pi$. For every $z \in \mathbb{C}$ we denote by
  $E_z = \pi^{-1}(z)$ the fiber over $z$.
\item All the critical point of $\pi$ are ordinary double points in
  the following sense. For every $p \in \textnormal{Crit}(\pi)$ there
  exist a $J_E$-holomorphic chart around $p$ (hence $J_{E}$ is
  integrable on this chart) with respect to which $\pi$ is a
  holomorphic Morse function.
\item There exists and exhaustion function
  $\varphi_E: E \longrightarrow \mathbb{R}$ and $R_0 \in \mathbb{R}$
  such that $(E, J_E, R_0, \varphi_E, \omega_E)$ is a symplectic
  manifold which is Stein at infinity. (See~\S\ref{sbsb:LiSt}.)
\item
  We assume that for every compact subset $K \subset \mathbb{C}$ there
  exists $R_K \geq R_0$ such that each level set $\varphi_E^{-1}(R)$,
  $R \geq R_K$, intersects each fiber $E_z$, $z \in K$,
  transversely. Note that this implies that for every $z \in K$,
  $\textnormal{Crit}(\varphi_E|_{E_z}) \subset E_{\varphi_E \leq R}$.
  Thus $(E_z, J_E|_{E_z}, R_K, \varphi_E|_{E_z}, \omega_E|_{E_z})$ is
  a symplectic manifold which is Stein at infinity, for every
  $z \in K \setminus \textnormal{Critv}(\pi)$.
\item Denote by $\Gamma$ the symplectic connection on
  $E \setminus \textnormal{Crit}(\pi)$, associated to $\omega_E$.
  (Recall that the horizontal distribution of this connection is the
  $\omega_E$-complement of the tangent spaces of the fibers of $\pi$.)
  Let $\gamma: [0,1] \longrightarrow \mathbb{C}$ be a smooth
  curve. Then the parallel transport
  $\Pi_{\gamma}: E_{\gamma(0)} \longrightarrow E_{\gamma(1)}$ along
  $\gamma$ is well defined at infinity.
\end{enumerate}
  
We now turn to real Lefschetz fibrations. By a real structure on a
Lefschetz fibration $\pi: E \longrightarrow \mathbb{C}$ we mean an
involution $c_E: E \longrightarrow E$ which is anti
$\omega_E$-symplectic and covers (with respect to $\pi$) the standard
complex conjugation
$c_{\mathbb{C}} : \mathbb{C} \longrightarrow \mathbb{C}$. We will
assume in addition that $c_E$ is anti $J_E$-holomorphic. We denote by
$E_{\mathbb{R}} \subset E$ the fixed locus of $c_E$ and call it the
{\em real part of $E$}.  Note that $E_{\mathbb{R}}$ is automatically a
smooth Lagrangian submanifold of $E$ (of course, it might be void).

It turns out that every smooth connected closed manifold can be
realized as the real part of a Lefschetz fibration. This is proved
in~\cite[Section~3]{FSS:ex}. More precisely, in that paper the
following is proved. Given a connected closed $n$-manifold $N$ and a
Morse function $f:N \longrightarrow \mathbb{R}$ with the property that
the level set of each critical value contains precisely one critical
value, there exist the following:
\begin{enumerate}
\item A smooth affine variety $E$, endowed with a complex structure
  denote by $J_E$.
\item A proper holomorphic function
  $\pi: E \longrightarrow \mathbb{C}$.
\item A plurisubharmonic function
  $\varphi : E \longrightarrow \mathbb{R}$ which is proper and bounded
  below. Denote by $\omega_E = -dd^{\mathbb{C}} \varphi$ the
  associated symplectic structure on $E$. Put also
  $\lambda_E = -d^{\mathbb{C}} \varphi$, so that
  $\omega_E = d\lambda_E$. (Here and in what follows, for a real
  valued function $\varphi$ on a complex manifold with complex
  structure $J$ we denote by $d^{\mathbb{C}}\varphi$ the $1$-form
  $dh \circ J$.)

\item An anti-$J_E$-holomorphic involution $c_E: E \longrightarrow E$.
\end{enumerate}
with the following properties:
\begin{enumerate}
\item The function $\varphi$ is $c_E$-invariant. In particular $c_E$
  is anti-$\omega_E$-symplectic.
\item $\pi \circ c_E = c_{\mathbb{C}} \circ \pi$, i.e.~$c_E$ covers
  the standard complex conjugation $c_{\mathbb{C}}$.
\item $\pi: E \longrightarrow \mathbb{C}$ is a Lefschetz fibration (in
  the sense of the definition from the beginning of~\S\ref{sb:rlef})
  with respect to the structures $\omega_E$ and $J_E$. Moreover, when
  endowed with $c_E$, $\pi:E \longrightarrow \mathbb{C}$ is a real
  Lefschetz fibration according to the preceding definition.
\item The real part $E_{\mathbb{R}} \subset E$ (with respect to $c_E$)
  is diffeomorphic to $N$.
\end{enumerate}
Moreover, $E$ and its associated structures above can be chosen such
that there is a diffeomorphism
$\vartheta : N \longrightarrow E_{\mathbb{R}}$ with
$\pi|_{E_{\mathbb{R}}} \circ \vartheta : N \longrightarrow \mathbb{R}$
arbitrarily close to $f$ in the $C^2$-topology.


Note that $\textnormal{Critv}(\pi)$ is invariant under the conjugation
$c_{\mathbb{C}}$, hence the points of
$\textnormal{Critv}(\pi) \setminus \mathbb{R}$ come in pairs of
conjugate points. Further, we have
$\pi(E_{\mathbb{R}}) \subset \mathbb{R}$ and
$\textnormal{Critv}(\pi|_{\mathbb{R}}) = \textnormal{Critv}(\pi) \cap
\mathbb{R}$.

For a point $x \in \textnormal{Critv}(\pi) \cap \mathbb{R}$ denote by
$\thmb_x \subset E$ the Lefschetz thimble associated to the curve
$[0,\infty) \ni t \mapsto itx \in \mathbb{C}$. \label{p:thimble-x}

\subsection{Embedding the ball cotangent bundle into a real Lefschetz
  fibration} \label{sb:bcbundle-rlef}

A simple calculation shows that $\lambda_E|_{E_{\mathbb{R}}} = 0$,
hence $E_{\mathbb{R}} \subset E$ is a $\lambda_E$-exact Lagrangian
submanifold.

Fix a Riemannian metric on $N$ and denote by $| \cdot |$ the norm on
the fibers of $T^*(N)$ corresponding to the Riemannian metric via the
isomorphism $T^*(N) \cong T(N)$ induced by the same metric. We
denote $$T^*_{\leq r}(N) = \{ v \in T^*(N) \mid |v| \leq r\}$$ the
radius-$r$ ball cotangent bundle. Similarly we have $T^*_{<r}(N)$,
$T^*_{\geq r}(N)$ etc.~and more generally for any subset
$I \subset \mathbb{R}$ we write
$T^*_{I}(N) = \{ v \in T^*(N) \mid |v| \in I\}$. Denote by
$\lambda_{\text{can}} = pdq$ the standard Liouville form on $T^*(N)$
and let $\omega_{\text{can}} = d \lambda_{\text{can}}$ be the
canonical symplectic structures. We identify $N$ with the zero section
of $T^*(N)$.

Fix a diffeomorphism $\vartheta : N \longrightarrow E_{\mathbb{R}}$ as
provided by the previous construction of the real Lefschetz fibration
$\pi: E \longrightarrow \mathbb{C}$. By the Darboux-Weinstein theorem
there exists $r_0>0$ and a symplectic embedding
$\kappa: T^*_{\leq r_0}(N) \longrightarrow E$ such that
$\kappa(x) = \vartheta (x)$ for every $x \in N$. Moreover, by possibly
decreasing $r_0>0$ we can arrange that the embedding $\kappa$ sends
the cotangent fibers $T^*_{x}(N) \cap T^*_{\leq r_0}(N)$,
$x \in \vartheta^{-1}(\textnormal{Critv}(\pi|_{\mathbb{R}}))$, to the
thimbles $\thmb_{\vartheta(x)} \cap \textnormal{image\,}(\kappa)$ in
$E$.  We write from now on
$\mathcal{U}_{\leq r} = \kappa(T^*_{\leq r}(N))$ for $r\leq r_0$, and
as before we have the analogous subsets $\mathcal{U}_{<r}$,
$\mathcal{U}_{>r}$ and $\mathcal{U}_{I}$.

Next, as explained in~\cite{FSS:ex} the symplectic embedding $\kappa$
is exact. More precisely, there exists a function
$f:T^*_{\leq r_0}(N) \longrightarrow \mathbb{R}$ such that
$\kappa^* \lambda_E = \lambda_{\text{can}} + df$. Moreover, we may
assume that $f|_N = 0$. (These statements follow from the fact that
$\kappa^*\lambda_{E} - \lambda_{\text{can}}$ is closed and vanishes
along the zero-section $N \subset T^*_{\leq r_0}(N)$.) In view of
point~(\ref{i:spec-lambda}) of Remark~\ref{r:spec-choices} we can
replace $\lambda_{\text{can}}$ by $\lambda := \kappa^*\lambda_E$ and
work from now on with the form $\lambda_E$ for defining the action
functional, spectral invariants and the spectral metric for exact
Lagrangians in $\mathcal{U}_{\leq r_0} \subset E$.

Henceforth we will identify $T^*_{\leq r_0}(N)$ with
$\mathcal{U}_{\leq r_0}$ and write $T^*_{\leq r_0}(N)$ and
$\mathcal{U}_{\leq r_0}$ (resp.~$N$ and $E_{\mathbb{R}}$)
interchangeably for the same thing.

In the following we will need a slight extension of
Proposition~\ref{p:loc-glob} that holds also for the thimbles
$\thmb_{x_j}$. Clearly the thimbles $\thmb_{x_j}$ are
$\lambda_E$-exact Lagrangians and we fix a marking for them. Note that
$\mathcal{U}_{\leq r_0} \subset E$ is a compact Liouville subdomain.
The following shows that Proposition~\ref{p:loc-glob} essentially
holds also for pairs of Lagrangians of the type $(L,
\thmb_{x_j})$. For simplicity in this proposition we take the
Hamiltonian terms in the Floer data to be $0$.
\begin{prop} \label{p:loc-glob-t} There exist isomorphisms of
  persistence modules
  $$j^{\leq \bullet}:
  HF^{\leq \bullet}(L, \thmb_{x_j}; (\mathcal{U}_{\leq r_0},
  \omega_E)) \longrightarrow HF^{\leq \bullet}(L, \thmb_{x_j}; (E,
  \omega_E))$$ defined for all closed marked $\lambda_E$-exact
  Lagrangians $L \subset \mathcal{U}_{\leq r_0}$. Moreover, the
  corresponding isomorphisms
  $j:= j^{\leq \infty}: HF(L, \thmb_{x_j}; (\mathcal{U}_{\leq r_0},
  \omega_E)) \longrightarrow HF(L, \thmb_{x_j}; (E, \omega_E))$ on the
  total homologies have the following properties:
  \begin{enumerate}
  \item They are compatible with the triangle products (among closed
    Lagrangians).
  \item They are compatible with the naturality maps
    $\mathcal{N}^{L'_0, L_0}_{\thmb_{x_j}}$ from~\S\ref{sbsb:pss-nat}
    (in case $L'_0$ and $L_0$ are exact-isotopic).
  \item They preserve spectral invariants, namely
    $$\sigma(j(a); L, \thmb_{x_j}; (\mathcal{U}_{\leq r_0}, \lambda_E))
    = \sigma(a; L, \thmb_{x_j}; (E, \lambda_E)), \; \forall a \in
    HF((L,\thmb_{x_j}); (\mathcal{U}_{\leq r_0}, \omega_E)).$$
  \end{enumerate}
  Completely analogous statements to the above continue to hold also
  for pairs of the type $(\thmb_{x_j}, L)$ with
  $L \subset \mathcal{U}_{\leq r_0}$ closed $\lambda_E$-exact
  Lagrangians.
\end{prop}

We will omit the proof, as it is based on very similar ideas as the
proof of Proposition~\ref{p:loc-glob}.

\subsection{The extended Lefschetz fibration}
\label{sb:extended} In order to use the theory developed
in~\cite{Bi-Co:lefcob-pub} we consider yet another Lefschetz fibration
$\pi':E' \longrightarrow \mathbb{C}$, which we call the extended
fibration of $E$. The construction is taken
from~\cite{Bi-Co:lefcob-pub} and goes as follows.  Write the critical
values of $\pi$ as
$\textnormal{Critv}(\pi) = \{x_1, \ldots, x_k, z_1, \widebar{z}_1,
\ldots, z_l, \widebar{z}_l\}$, where $x_i \in \mathbb{R}$ are the real
critical values and $z_j, \widebar{z}_j$ are pairs of non-real complex
conjugate critical values of $\pi$. Let $p_i \in E_{x_i}$ be the
critical point corresponding to $x_i$. Let $\nu>0$ be large enough
such that $\nu > |\text{Im} \, z_j|$ for every $j$.
\begin{prop} \label{p:extended-1} There exists a Lefschetz fibration
  $\pi' : E' \longrightarrow \mathbb{C}$ with the following
  properties:
  \begin{enumerate}
  \item \label{i:E'} $(E', \pi', J_{E'}, \omega_{E'})$ coincides with
    $(E, \pi, J_E, \omega_E)$ over
    $\{z \in \mathbb{C} \mid -\nu < \text{Im} \, z\}$.  Moreover,
    $\textnormal{Critv}(\pi') = \{x_1, \ldots, x_k, x'_1, \ldots,
    x'_k, z_1, \widebar{z}_1, \ldots, z_l, \widebar{z}_l\}$, namely
    every real critical value $x_i$ has now a corresponding critical
    value $x'_i$ (which is not assumed to be real anymore). The new
    critical values $x'_i$ have $\text{Im}\, x'_i < -\nu$, and they
    are placed as depicted in Figure~\ref{f:ext-lef-1}.
  \item \label{i:match-S} Denote by $\gamma_i \subset \mathbb{C}$,
    $i=1, \ldots, k$, the paths connecting $x_i$ with $x'_i$, as in
    figure~\ref{f:ext-lef-1} and denote by $p'_i \in E'_{x'_i}$ the
    critical point corresponding to $x'_i$. The Lefschetz thimbles
    emanating from $p_i$ and from $p'_i$ along the two opposite ends
    of $\gamma_i$ form a matching sphere $S_i \subset E'$, lying over
    $\gamma_i$. (Put in different words, the vanishing cycles
    emanating from $p_i$ along $\gamma_i$ converge over the other end
    of $\gamma_i$ to the point $p'_i$ and their union forms a smooth
    Lagrangian sphere $S_i$.)
  \item \label{i:omega-E'-exact} The symplectic structure
    $\omega_{E'}$ is exact. Moreover, it admits a primitive
    $\lambda_{E'}$ which coincides with $\lambda_{E}$ over
    $E|_{-\nu < \text{Im} \, z}$.
  \item \label{i:E'-Stein} There exists an exhaustion function
    $\varphi': E' \longrightarrow \mathbb{R}$ and $R_0 \in \mathbb{R}$
    such that $(E', J_{E'}, \varphi', R_0, \omega_{E'})$ is a
    symplectic manifold which is Stein at infinity.
  \item \label{i:S-exact} The matching spheres $S_i$
    from~\eqref{i:match-S} are $\lambda_{E'}$-exact.
  \end{enumerate}
\end{prop}

\begin{rem}
  We do not require that the exact $1$-form $\lambda_{E'}$ from
  point~\eqref{i:omega-E'-exact} of the proposition coincides with
  $-d^{J_{E'}}\varphi'$ at infinity. While it seems that this can be
  arranged, we will not need such a statement in the following.
\end{rem}

\begin{figure}[htbp]
  \begin{center}
    \includegraphics[trim=0 0 120 0, scale=0.68]{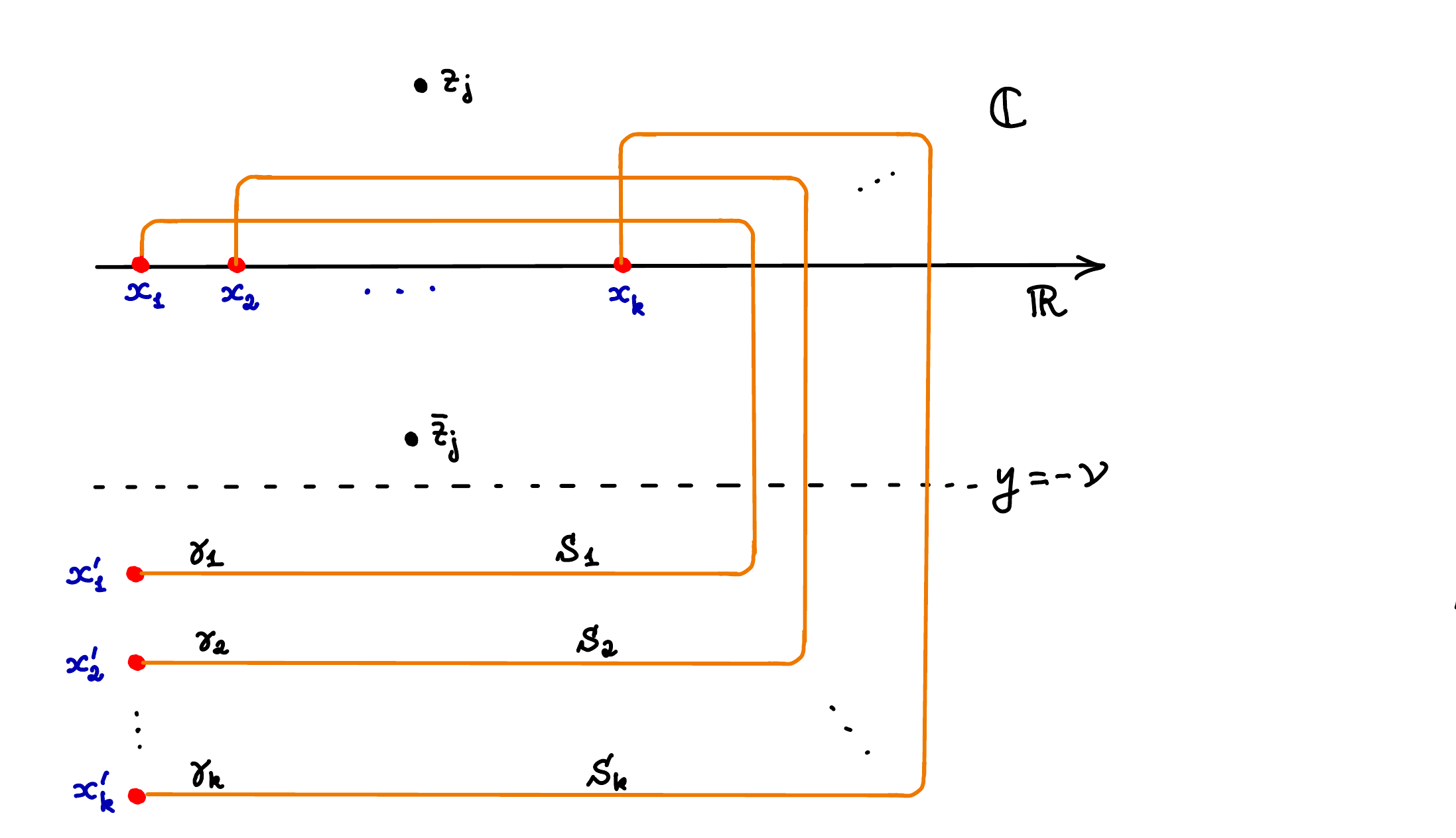}
  \end{center}
  \caption{The extended Lefschetz fibration $E'$ and the matching
    spheres $S_j$, projected to $\mathbb{C}$.}
  \label{f:ext-lef-1}
\end{figure}

\begin{proof}[Proof of Proposition~\ref{p:extended-1}]
  Statements~\eqref{i:E'},~\eqref{i:match-S} and~\eqref{i:E'-Stein}
  follow from the theory developed
  in~\cite[Sections~15d,~16e]{Se:book-fukaya-categ}.

  To prove~\eqref{i:omega-E'-exact} we begin by showing that
  $\omega_{E'}$ is exact. Denote
  $E^+ := E|_{\{-\nu < \text{Im} \, z\}}$.  Let
  $\gamma'_i \subset \mathbb{C}$ be the path obtained from $\gamma_i$
  by chopping a little neighborhood of its second end near $x'_i$,
  namely $\gamma'_i = \gamma_i \setminus D'_i$, where $D'_i$ is a
  little open disk around $x'_i$.
  Fix also another point $y_i \in \gamma_i \cap E^+$ which is
  different from $x_i$.

  Denote by $T_{x'_i} \subset E'$ the Lefschetz thimble emanating from
  $p_i$ along the path $\gamma'_i$ and by $T_{y_i} \subset T_{x'_i}$
  the part of that thimble lying over $\gamma'_i$, between $x_i$ and
  $y_i$. Denote by $\partial T_{x'_i}$ and $\partial T_{y_i}$ the
  boundaries of these ``partial'' thimbles. These are Lagrangian
  spheres in the fibers of $E'$ over $x'_i$ and $y_i$ respectively.

  By standard topological arguments there is a canonical isomorphism
  \begin{equation} \label{eq:iso-kappa}
    \kappa: H_2(E^+, \cup_{i=1}^k \partial T_{y_i})
    \longrightarrow H_2(E'),
  \end{equation}
  where the homologies are taken with any given coefficient
  group. This isomorphism is induced from the following chain-level
  map. Let $C$ be a relative cycle of
  $(E^+, \cup_{i=1}^k \partial T_{y_i})$. For $w \in \gamma_i$ denote
  by $\Pi_{\gamma_i}^{y_i, w} $ the parallel transport (with respect
  to the connection induced by $\omega_{E'}$) along $\gamma_i$ from
  $E'_{y_i}=E_{y_i}$ to $E'_{w}$. Take the part of $\partial C$ lying
  in $\partial T_{y_i}$ and consider its trail under this parallel
  transport from $y_i$ till $x'_i$, namely the union of
  $\Pi_{\gamma_i}^{y_i, w}(\partial C \cap \partial T_{y_i})$, where
  $w$ runs along $\gamma_i$ between $y_i$ and $x'_i$. Note that while
  $\Pi_{\gamma_i}^{y_i, w}$ is in general not defined for the end
  point $w=x'_i$, here we apply $\Pi_{\gamma_i}^{y_i, x'_i}$ to
  $\partial C \cap \partial T_{y_i}$ which yields the point $p'_i$.
  Therefore the trail of $\partial C \cap \partial T_{y_i}$ along
  $\gamma_i$ between $y_i$ and $x'_i$ is well defined and gives
  another relative cycle in $(E', \partial T_{y_i})$, which we denote
  by $\text{Tr}_{y_i, x'_i}(\partial C)$. Note that
  $\partial \text{Tr}_{y_i, x'_i}(\partial C) = -(\partial C \cap
  \partial T_{y_i})$.

  We can now cap the trails $\text{Tr}_{y_i, x'_i}(\partial C)$,
  $i=1, \ldots, k$, to $C$ along $\partial C \cap \partial T_{y_i}$,
  and obtain at the end an absolute cycle $C'$ in $E'$. The map
  $\kappa$ is induced by the chain level map $C \longmapsto C'$.

  In order to show that $\omega_{E'}$ is exact, we will use the
  isomorphism $\kappa$, with coefficients in $\mathbb{R}$. It is
  enough to prove that $\langle [\omega_{E'}], \kappa(A) \rangle = 0$
  for every
  $A \in H_2(E^+, \cup_{i=1}^k \partial T_{y_i}; \mathbb{R})$. For
  this end, note that $\omega_{E'}$ vanishes over each of the trails
  $\text{Tr}_{y_i, x'_i}(\partial C)$, hence
  $$\langle [\omega_{E'}], \kappa(A) \rangle = \langle [\omega_{E'}],
  A \rangle = \langle [\omega_{E}], A \rangle,$$ where the last
  equality holds because $\omega_{E'}|_{E^+} = \omega_{E}|_{E^+}$.
  Now $\omega_{E} = d\lambda_{E}$, hence
  \begin{equation} \label{eq:lambda-del-i} \langle [\omega_{E}], A
    \rangle = \sum_{i=1}^k \langle [\lambda_{E}|_{\partial T_{y_i}}],
    \partial_i A \rangle,
  \end{equation}
  where $\partial_i A$ is the component of $\partial A$ corresponding
  to $H_1(\partial T_{y_i}; \mathbb{R})$. But $T_{y_i}$ is clearly a
  $\lambda_E$-exact Lagrangian submanifold, thus the right-hand side
  of~\eqref{eq:lambda-del-i} vanishes. This completes the proof that
  $\omega_{E'}$ is exact.

  Next, we prove that $\omega_{E'}$ admits a primitive $\lambda_{E'}$
  that extends $\lambda_E|_{E^+}$. We claim that this would follow
  from the assertion that the map induced by inclusion
  $i_*: H_1(E^+; \mathbb{R}) \longrightarrow H_1(E'; \mathbb{R})$ is
  injective. Indeed, fix a small $\epsilon>0$ such that
  $\text{Im}\, x'_j < -(\nu+\epsilon)$ for all $j$, and write
  $E^+_{\epsilon} = E|_{-(\nu+\epsilon) < \text{Im} \, z}$. Denote by
  $i^{\epsilon}: E^+_{\epsilon} \longrightarrow E'$ the
  inclusion. Clearly $i_*$ is injective iff
  $i^{\epsilon}_*: H_1(E^+_{\epsilon}; \mathbb{R}) \longrightarrow
  H_1(E')$ is injective. Fix any primitive $\lambda'$ of $\omega_{E'}$
  and consider the $1$-form $\lambda_E|_{E^+} - \lambda'|_{E^+}$. This
  form is closed because $\omega_E|_{E^+} = \omega_{E'}|_{E^+}$. Since
  $i^{\epsilon}_*$ is injective, the restriction map
  $(i^{\epsilon})^*:H^1(E'; \mathbb{R}) \longrightarrow
  H^1(E_{\epsilon}^+; \mathbb{R})$ is surjective, hence there exists a
  closed $1$-form $\alpha'$ on $E'$ and a smooth function
  $f: E_{\epsilon}^+ \longrightarrow \mathbb{R}$ such that
  $\alpha'|_{E_{\epsilon}^+} = \lambda_E|_{E_{\epsilon}^+} -
  \lambda'|_{E_{\epsilon}^+} + df$. Now cut off the function $f$ in
  between $E^+$ and $E^+_{\epsilon}$ to obtain another function
  $f': E' \longrightarrow \mathbb{R}$ which coincides with $f$ on
  $E^+$ and vanishes outside of $E^+_{\epsilon}$. The desired $1$-form
  $\lambda_{E'}$ is then given by
  $$\lambda_{E'} := \alpha' + \lambda' - df'.$$

  To complete the proof it remains to show that
  \begin{equation} \label{eq:i-injective}
    i_*: H_1(E^+; \mathbb{R}) \longrightarrow H_1(E'; \mathbb{R})
  \end{equation}
  is injective. To this end, denote by $F = \pi^{-1}(w)$ the fiber
  of $\pi: E \longrightarrow \mathbb{R}$ over a regular value $w$ of
  $\pi$ with
  $w \in \{ z \in \mathbb{C} \mid \text{Im}\, z > -\nu \}$.
  
  Assume first that $\dim F>0$. By standard arguments, the
  inclusions $F \subset E^+$ and $F \subset E'$ induce isomorphisms
  $H_1(F) \cong H_1(E^+)$ and $H_1(F) \cong H_1(E')$, where the
  homologies are taken with arbitrary coefficients. Therefore
  $i_*: H_1(E^+) \longrightarrow H_1(E')$ is an isomorphism.

  Assume now that $\dim F = 0$. Choose a small $\epsilon>0$ such that
  all the critical values of $\pi$ are in
  $\{ \text{Im}\, z > -\nu + \epsilon\}$ and write
  $E'^{-} = E'|_{\text{Im}\, z < -\nu +\epsilon}$. Note that
  $E^+ \cap E'^{-}$ is homotopy equivalent to $F$ which is discrete,
  hence $H_1(E^+ \cap E'^{-}; \mathbb{R}) = 0$.  By the Mayer-Vietoris
  sequence for $E' = E^+ \cup E'^{-}$ it follows that
  $i_*: H_1(E^+; \mathbb{R}) \longrightarrow H_1(E'; \mathbb{R})$ is
  injective.

  This completes the proof of the injectivity of $i_*$
  in~\eqref{eq:i-injective} for all possible values of $\dim F$, hence
  also the proof of point~\eqref{i:omega-E'-exact} of the proposition.

  Point~\eqref{i:S-exact} is obvious if $\dim F > 0$ (since in that
  case $\dim(S_i) \geq 2$). Assume that $\dim F = 0$.  In this case
  $N \approx S^1$, and without loss of generality we may assume that
  the number of real critical values of $\pi$ is $k=2$. (This is not
  really essential for the rest of the proof, it just simplifies a bit
  the notation.) Let $\lambda_{E'}$ be a $1$-form from
  point~\eqref{i:omega-E'-exact}, whose existence we have just
  proved. In the course of the argument below we will need to alter
  this $1$-form, so we will denote it by $\lambda'$.

  Let $E^+_{\epsilon}$ be as earlier in the proof. Denote by
  $j_*^{\epsilon}: H_1(E^+_{\epsilon}; \mathbb{R}) \longrightarrow
  H_1(E, \partial T_{x'_1} \cup \partial T_{x'_2}; \mathbb{R})$,
  $i_*^{\epsilon}: H_1(E^+_{\epsilon}; \mathbb{R}) \longrightarrow
  H_1(E'; \mathbb{R})$ the maps induced by the inclusion
  $E^{+}_{\epsilon} \subset E'$. Similarly to the isomorphism
  from~\eqref{eq:iso-kappa} we have also an isomorphism
  $$\kappa: H_1(E, \partial T_{x'_1} \cup \partial T_{x'_2};
  \mathbb{R}) \longrightarrow H_1(E'; \mathbb{R})$$ which we continue
  denoting by $\kappa$ and which is defined by exactly the same means.

  Consider the homology classes $[S_1], [S_2] \in H_1(E'; \mathbb{R})$
  as well as the subspace
  $\textnormal{image\,} i^{\epsilon}_* \subset H_1(E'; \mathbb{R})$.
  We claim that no non-trivial linear combination of $[S_1], [S_2]$
  belongs to $\textnormal{image\,} i^{\epsilon}_*$. This can be easily
  seen by looking at the images of $\kappa^{-1} [S_1] = [T_{x'_1}]$,
  $\kappa^{-1} [S_2] = [T_{x'_2}]$ under the the connecting
  homomorphism
  $$\partial_*: H_1(E, \partial T_{x'_1} \cup \partial T_{x'_2};
  \mathbb{R}) \longrightarrow H_0(\partial T_{x'_1} \cup \partial
  T_{x'_2}; \mathbb{R}) = H_0(\partial T_{x'_1}; \mathbb{R}) \oplus
  H_0(\partial T_{x'_2}; \mathbb{R})$$ and noting that
  $\kappa^{-1}(\textnormal{image\,}i^{\epsilon}_*) =
  \textnormal{image\,} j^{\epsilon}_*$ is sent to $0$ by $\partial_*$.

  In view of the preceding claim we can find a closed $1$-form
  $\theta$ on $E'$ such that:
  \begin{enumerate}
  \item $[\theta] \in H^1(E'; \mathbb{R})$ vanishes on
    $\textnormal{image\,}i^{\epsilon}_*$.
  \item $\langle [\theta], [S_1] \rangle = \int_{S_1} \lambda'$ and
    $\langle [\theta], [S_2] \rangle = \int_{S_2} \lambda'$.
  \end{enumerate}
  By the property of $\theta$ we have
  $(i^{\epsilon})^* [\theta] = 0 \in H^1(E^+_{\epsilon}; \mathbb{R})$,
  hence there exists a smooth function
  $h: E^{+}_{\epsilon} \longrightarrow \mathbb{R}$ such that
  $\theta|_{E^+_{\epsilon}} = dh$. Now, cutoff $h$ near
  $\{ \text{Im}\, z = -\nu-\epsilon \}$ and extend the resulting
  function to a smooth function $h':E' \longrightarrow \mathbb{R}$
  which vanishes on $\{ \text{Im}\, z \leq -\nu - \epsilon\}$ and such
  that $h' = h$ on $E^+ = \{ \text{Im} z > -\nu \}$. Replacing the
  form $\lambda_{E'}$ provided by point~\eqref{i:omega-E'-exact} of
  the proposition by the form
  $$\lambda'' := \lambda' - \theta + dh'$$ we still obtain a primitive
  of $\omega_{E'}$ that coincides with $\lambda_{E}$ over $E^+$ and
  such that the matching spheres $S_1$, $S_2$ are
  $\lambda''$-exact. This completes the proof of
  point~\eqref{i:S-exact} of the proposition in case the fibers of
  $\pi: E \longrightarrow \mathbb{C}$ are $0$-dimensional.
\end{proof}


\section{Floer theory in $E$ versus $E'$} \label{s:Fl-E-E'}

Recall that the extended Lefschetz fibration
$\pi': E' \longrightarrow \mathbb{C}$ from~\S\ref{sb:extended} has
been constructed such that it coincides, together with its associated
structures, with the original Lefschetz fibration
$\pi: E \longrightarrow \mathbb{C}$ over
$\{ z\in \mathbb{C} \mid -\nu < \text{Im} z \}$.

Let $L_0, L_1 \subset E'$ be two marked exact Lagrangians and assume
that
$L_0, L_1 \subset E'|_{\{-\nu < \text{Im} z\}} = E|_{\{-\nu <
  \text{Im} z\}}$. By the arguments from~\cite{Bi-Co:lefcob-pub} the
Floer complexes of $(L_0, L_1)$ coincide, when viewed in $E$ and in
$E'$, provided we choose the right Floer data. More precisely, let $H$
be a Hamiltonian function compactly supported in
$E|_{\{-\nu < \text{Im} z\}}$. Then there exist regular Floer data
$\mathscr{D} = (H, J)$ in $E$ and $\mathscr{D}' = (H, J')$ in $E'$,
with the same Hamiltonian function $H$ such that all the Floer
trajectories for $(L_0,L_1)$ with respect to $\mathscr{D}$ coincide
with those for $\mathscr{D}'$ and they all lie inside
$E|_{\{-\nu < \text{Im} z\}}$. This easily follows from the open
mapping theorem for holomorphic functions, by choosing appropriate
compatible almost complex structures $J$ and $J'$ for which the
projections $\pi$ and $\pi'$ are holomorphic. Consequently we have a
chain isomorphism (induced by the identity map on $\mathcal{O}(H)$)
\begin{equation} \label{eq:CF-E-E'} CF(L_0, L_1; \mathscr{D}; E)
  \longrightarrow CF(L_0, L_1; \mathscr{D}'; E')
\end{equation}
which preserves the action filtration. The $E$ and $E'$ in the
notation of the Floer complexes in the preceding formula indicate the
ambient manifold in which the respective Floer complex is being
considered. Consequently~\eqref{eq:CF-E-E'} induces an action
preserving isomorphism of persistence modules
$$HF^{\leq \bullet}(L_0, L_1; E) \cong HF^{\leq \bullet}(L_0,L_1; E'),$$
hence the spectral invariants and boundary depths of $CF(L_0, L_1)$,
viewed either in $E$ or in $E'$, coincide.

The above can be generalized to the Fukaya categories of $E$ and
$E'$. More specifically, denote by $\fuk(E)$ and $\fuk(E')$ the Fukaya
categories of $E$ and $E'$, whose objects are the closed marked exact
Lagrangian submanifolds in $E$ and $E'$. Let
$\fuk(E; -\nu) \subset \fuk(E)$ be the full subcategory whose objects
are closed exact Lagrangians $L \subset E|_{\{-\nu < \text{Im}
  z\}}$. As explained in~\cite{Bi-Co:lefcob-pub} it is possible to
choose the auxiliary data required for the definitions of $\fuk(E)$
and $\fuk(E')$ in such a way that the inclusion of objects
$\textnormal{Ob}(\fuk(E; -\nu)) \subset \textnormal{Ob}(\fuk(E'))$
extends to a (homologically) full and faithful $A_{\infty}$-functor
$\text{Inc}: \fuk(E; -\nu) \longrightarrow \fuk(E')$. Moreover, if we
view $\fuk(E; -\nu)$ and $\fuk(E')$ as weakly filtered
$A_{\infty}$-categories, we can assume that the functor $\text{Inc}$
is a weakly filtered functor (see~\S\ref{sb:wf-f}
and~\S\ref{sb:wf-ai-mod} for a brief explanation of these concepts,
and~\cite[\S2]{Bi-Co-Sh:lshadows-long} for the precise definitions and
more details).

This has the following consequence for $A_{\infty}$-modules. Let
$L \subset E'$ be a marked exact Lagrangian and assume that
$L \subset E|_{\{-\nu < \text{Im} z\}}$. Denote by $\mathscr{L}^{E'}$
the Yoneda module of $L$, viewed as an $A_{\infty}$-module over
$\fuk(E')$ and by $\mathscr{L}^{E, -\nu}$ the Yoneda module of $L$
over $\fuk(E; -\nu)$. Both modules are weakly filtered in the sense
of~\cite{Bi-Co-Sh:lshadows-long} and with the right choices of
auxiliary data for $\fuk(E; -\nu)$, $\fuk(E')$ we have that
$$\text{Inc}^*(\mathscr{L}^{E'}) = \mathscr{L}^{E, -\nu}$$
as weakly filtered $\fuk(E; -\nu)$-modules.

Next, we compare the Floer theory of the matching spheres $S_j$ in
$E'$ with the Floer theory of the thimbles $\thmb_{x_j}$ in $E$,
defined on page~\pageref{p:thimble-x}. Fix a rectangle
$\mathcal{R} \subset \mathbb{C}$ of the type
\begin{equation} \label{eq:rectangle} \mathcal{R} = \{x + iy \in
  \mathbb{C} \mid x \in (a,b), -\nu < y < \epsilon \}
\end{equation}
such that
$S_j \cap \pi'^{-1}(\mathcal{R}) = \thmb_{x_j} \cap
\pi^{-1}(\mathcal{R})$. (See Figure~\ref{f:ext-lef-2}.)

\begin{figure}[htbp]
  \begin{center}
    \includegraphics[scale=0.68]{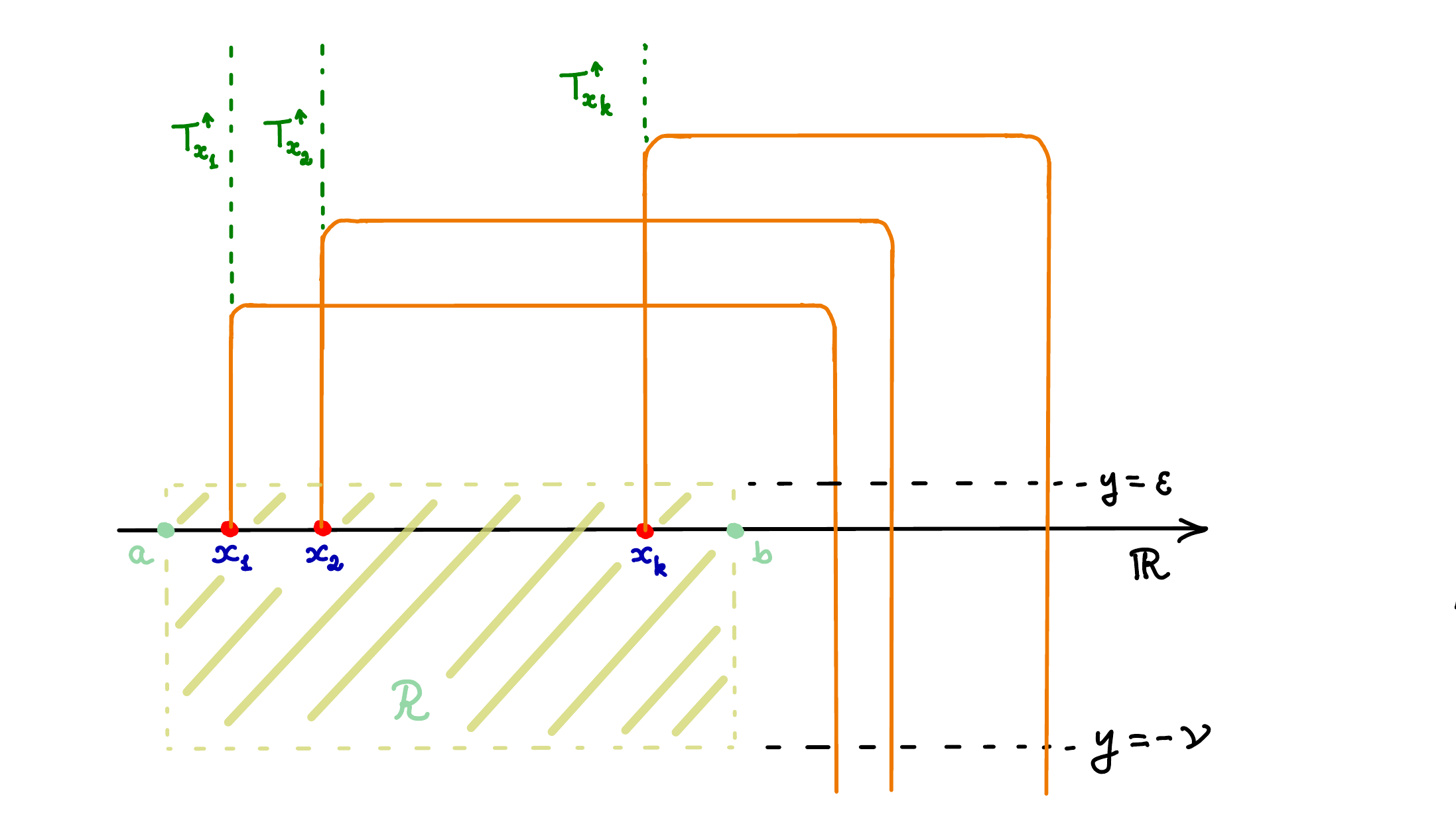}
  \end{center}
  \caption{The rectangle $\mathcal{R}$ and the projection to
    $\mathbb{C}$ of the thimbles $\thmb_{x_j}$.}
  \label{f:ext-lef-2}
\end{figure}

Let $L \subset E'$ be a marked exact Lagrangian and assume that
$\pi'(L) \subset \mathcal{R}$. Let $H$ be a Hamiltonian function
compactly supported in $\pi^{-1}(\mathcal{R})$. Then there exist
almost complex structures $J$ on $E$ and $J'$ on $E'$, compatible with
$\omega_E$ and $\omega_{E'}$ respectively, making the Floer data
$\mathscr{D} = (H,J)$ and $\mathscr{D}'=(H,J')$ regular and such that
the Floer trajectories for $(L,S_j;\mathscr{D}')$ in $E'$ and the
Floer trajectories of $(L, \thmb_{x_j}; \mathscr{D})$ in $E$ coincide
and moreover all these trajectories lie inside
$\pi^{-1}(\mathcal{R})$. This follows again from an open mapping
theorem argument as in~\cite{Bi-Co:lefcob-pub}.

It follows that the identity map on $\mathcal{O}(H)$ gives an action
preserving chain isomorphism
$$CF(L, S_j; \mathscr{D}'; E') \longrightarrow CF(L,
\thmb_{x_j}; \mathscr{D}; E).$$ Here we view $\thmb_{x_j} \subset E$
as a marked exact Lagrangian with primitive function adjusted such
that it coincides with the given primitive function of $S_j$ along
$S_j \cap \pi'^{-1}(\mathcal{R}) = \thmb_{x_j} \cap
\pi^{-1}(\mathcal{R})$.

Denote by $\fuk(E; \mathcal{R}) \subset \fuk(E')$ the full subcategory
whose objects are marked exact Lagrangians $L$ with
$\pi(L) \subset \mathcal{R}$. Similarly to $\text{Inc}$ we have weakly
filtered inclusion $A_{\infty}$-functors
$\text{Inc}_{\mathcal{R}, -\nu}: \fuk(E; \mathcal{R}) \longrightarrow
\fuk(E; -\nu)$ and
$\text{Inc}_{\mathcal{R}, E'}: \fuk(E; \mathcal{R}) \longrightarrow
\fuk(E')$ with
$\text{Inc}_{\mathcal{R}, E'} = \text{Inc} \circ
\text{Inc}_{\mathcal{R}, -\nu}$.

Putting all these constructions together we deduce:

\ocnote{
\begin{lem}\label{lem:inclusion-f}
Let $\mathcal{S}_j$  be the Yoneda module of $S_j$ and let $\tthmb_{x_j}$ be 
the Yoneda module of $\thmb_{x_j}$, the latter being
viewed as a module over $\fuk(E; -\nu)$. With the appropriate choice
of auxiliary data, we have
\begin{equation} \label{eq:S-T} \text{Inc}_{\mathcal{R},
    E'}^*(\mathcal{S}_j) = \text{Inc}_{\mathcal{R},
    -\nu}^*(\tthmb_{x_j})
\end{equation}
as weakly filtered $\fuk(E; \mathcal{R})$-modules.
\end{lem}
}


\section{Cone decompositions in Lefschetz
  fibrations} \label{s:cone-decomp} Recall
from~\cite{Bi-Co:lefcob-pub} that the Yoneda modules associated to
closed Lagrangian submanifolds (or more generally Lagrangian
cobordisms), satisfying appropriate exactness or monotonicity
conditions, in a Lefschetz fibration $E$ can be represented as
iterated cones of modules involving the matching spheres $S_j$ in the
extended Lefschetz fibration $E'$. We will apply these results below,
to the fibrations $E$ and $E'$ constructed
in~\S\ref{sb:rlef}~-~\S\ref{sb:extended} above, while also keeping
track of the action filtrations.

Let $\pi: E \longrightarrow \mathbb{R}$ be a real Lefschetz fibration
with critical values
$x_1, \ldots, x_k, z_1, \widebar{z}_1, \ldots, z_l, \widebar{z}_l$ and
let $\pi':E' \longrightarrow \mathbb{C}$ be the extended Lefschetz
fibration, as in~\S\ref{sb:extended}. Fix $\epsilon>0$ with
$\epsilon < | \text{Im}\, z_j|$ for every $j$. Let $K \subset E$ be a
closed $\lambda_E$-exact Lagrangian submanifold and assume that
$K \subset E|_{\{ |\text{Im}| \, z < \epsilon\}}$. Consider the
matching spheres $S_j \subset E'$ and denote by
$\tau_{S_j}: E' \longrightarrow E'$ the Dehn-twist around $S_j$,
supported in a small neighborhood of $S_j$. Note that $\tau_{S_j}$ is
well defined up to Hamiltonian isotopy (supported near $S_j$) since
the sphere $S_j$, being a matching sphere, has a canonical smooth
identification with $S^n$ ($2n = \dim_{\mathbb{R}}E$) up to smooth
isotopy.

Put $K^{(0)} := K$, $K^{(j)} := \tau_{S_j}(K^{(j-1)})$,
$j = 1, \ldots, k$. We view these Lagrangians as objects of the
$\lambda_{E'}$-exact Fukaya category $\fuk(E')$ of $E'$. Denote by
$\mathcal{K}^{(j)}$ the Yoneda modules associated to $K^{(j)}$,
$j=0, \ldots, k$. Write also $\mathcal{K} := \mathcal{K}^{(0)}$ for
the Yoneda module of $K$ and denote by $\mathcal{S}_j$,
$j=1, \ldots, k$, the Yoneda modules associated to the matching
spheres $S_j$.

By the results of~\cite{Bi-Co:lefcob-pub}, $\mathcal{K}$ is
quasi-isomorphic, in the $A_{\infty}$-category of modules over
$\fuk(E')$, to the following iterated cone of $\fuk(E')$-modules:
\begin{equation} \label{eq:L-cone} \mathcal{K} \cong [\mathcal{B}_1
  \longrightarrow \cdots \longrightarrow \mathcal{B}_k \longrightarrow
  \mathcal{K}^{(k)}],
\end{equation}
where each of the modules $\mathcal{B}_j$, $j=1, \ldots, k$, has
itself an iterated cone decomposition of the following type:
\begin{equation} \label{B_j-cone} \mathcal{B}_j = [\mathcal{S}_j
  \otimes CF(S_j,K) \longrightarrow \mathcal{B}_{j,1} \longrightarrow
  \mathcal{B}_{j,2} \longrightarrow \cdots \longrightarrow
  \mathcal{B}_{j, j-1}].
\end{equation}
In order to describe the modules $\mathcal{B}_{j,d}$,
$1 \leq d \leq j-1$, that appear in~\eqref{B_j-cone} we need a bit of
notation. Denote by $\mathcal{I}_{d, j-1}$ the set of all
multi-indices $\underline{i} = (i_1, \ldots, i_d)$ with
$1 \leq i_1 < i_2 < \cdots < i_d \leq j-1$. We order the elements of
$\mathcal{I}_{d, j-1}$ by the lexicographic order. For each
multi-index $\underline{i} \in \mathcal{I}_{d, j-1}$ put
\begin{equation} \label{eq:Cij}
  \mathcal{C}_{\underline{i}, j} := \mathcal{S}_j \otimes
  CF(S_j, S_{i_d}) \otimes CF(S_{i_d}, S_{i_{d-1}}) \otimes \cdots
  \otimes CF(S_{i_2}, S_{i_1}) \otimes CF(S_{i_1}, K).
\end{equation}
Let $m_{d, j-1} := \# \mathcal{I}_{d, j-1}$ and order the elements of
$\mathcal{I}_{d, j-1} = \{ \underline{i}^{(1)}, \ldots,
\underline{i}^{(m_{d, j-1})}\}$ in such a way that
$\underline{i}^{(1)} \precneqq \underline{i}^{(2)} \precneqq \cdots
\precneqq \underline{i}^{(m_{d, j-1})}$. Then
\begin{equation} \label{eq:Bjd}
  \mathcal{B}_{j, d} = [\mathcal{C}_{\underline{i}^{(1)}, j} \longrightarrow
  \mathcal{C}_{\underline{i}^{(2)}, j} \longrightarrow \cdots
  \longrightarrow \mathcal{C}_{\underline{i}^{(m_{d, j-1})}, j}].
\end{equation}

Having established a cone decomposition of the module $\mathcal{K}$
over the $A_{\infty}$-category $\fuk(E')$ we consider its pull-back to
Fukaya categories associated to $E$. Recall from~\S\ref{s:Fl-E-E'}
that we have the Fukaya categories $ \fuk(E; \mathcal{R})$ and
$\fuk(E; -\nu)$. We take the rectangle $\mathcal{R}$
from~\eqref{eq:rectangle} to be wide enough such that it contains
$\pi(K)$. Recall also the inclusion functor
$$\text{Inc}_{\mathcal{R}, E'}: \fuk(E; \mathcal{R}) \longrightarrow
\fuk(E')$$ that factors as the composition
$\text{Inc}_{\mathcal{R}, E'} = \text{Inc} \, \circ \,
\text{Inc}_{\mathcal{R}, -\nu}$ of the two functors
$$\text{Inc}_{\mathcal{R}, -\nu}: \fuk(E; \mathcal{R}) \longrightarrow
\fuk(E; -\nu), \quad \text{Inc}: \fuk(E; -\nu) \longrightarrow
\fuk(E').$$

By pulling back the cone decomposition~\eqref{eq:L-cone} via
$\text{Inc}_{\mathcal{R}, E'}^*$ we obtain a similar cone
decomposition for $\mathcal{K}$ (now viewed as a module over
$\fuk(E; \mathcal{R})$), where the modules $\mathcal{S}_j$
in~\eqref{B_j-cone} and~\eqref{eq:Cij} are replaced by
$\text{Inc}^*_{\mathcal{R}, -\nu}(\tthmb_{x_j})$,
see~\eqref{eq:S-T}. (Note that the terms involving the Floer complexes
of $S_j$ and and of $S_{i_l}$ remain unchanged.)

Finally, we claim that the pullback
$\text{Inc}_{\mathcal{R}, E'}^* \mathcal{K}^{(k)}$ of the the module
$\mathcal{K}^{(k)}$ which appears last in~\eqref{eq:L-cone} is
acyclic.

We will outline below in~\S\ref{sb:ex-t-Dehn} the proof of the cone
decomposition~\eqref{eq:L-cone}, the
expressions~\eqref{B_j-cone}~-~\eqref{eq:Bjd} as well as the
acyclicity of $\text{Inc}_{\mathcal{R}, E'}^* \mathcal{K}^{(k)}$.
Then in~\S\ref{sb:filt-cd-Lef} and~\S\ref{sb:prfs-filt-cd} we will
refine these results to take into account also the action filtrations.

Before we turn to these details, here is a concrete example showing
how the cone decomposition of $\mathcal{K}$ looks like in case the
number of real critical values of $\pi$ is $k=3$:
\begin{equation*}
  \begin{aligned}
    \mathcal{K} \cong [ & \mathcal{S}_1 \otimes CF(S_1, K)
    \longrightarrow \\
    & \mathcal{S}_2 \otimes CF(S_2, K) \longrightarrow \mathcal{S}_2
    \otimes
    CF(S_2,S_1) \otimes CF(S_1, K) \longrightarrow \\
    & \mathcal{S}_3 \otimes CF(S_3, K) \longrightarrow \mathcal{S}_3
    \otimes CF(S_3,S_2) \otimes CF(S_1,K) \longrightarrow
    \mathcal{S}_3 \otimes CF(S_3,S_2) \otimes CF(S_2,K)
    \longrightarrow \\
    & \mathcal{S}_3 \otimes CF(S_3,S_2) \otimes CF(S_2, S_1) \otimes
    CF(S_1, K) \longrightarrow \mathcal{K}^{(3)}]
  \end{aligned}
\end{equation*}

\subsection{Exact triangles associated to Dehn
  twists} \label{sb:ex-t-Dehn} Let $(X^{2n}, \omega = d\lambda)$ be a
Liouville domain and $S^n \xrightarrow{\; \approx \;} S \subset X$ a
parametrized Lagrangian sphere. In case $n=1$ we additionally assume
that $S$ is $\lambda$-exact. Let $\tau := \tau_S: X \longrightarrow X$
be a symplectomorphism, supported in $\textnormal{Int\,} X$, which
represents the symplectic mapping class of the Dehn twist around
$S$. Note that $\tau$ is an exact symplectomorphism, hence sends exact
Lagrangians to exact Lagrangians.

A well known result of Seidel~\cite{Se:long-exact,
  Se:book-fukaya-categ} says that for every exact Lagrangian
$Q \subset X$ there is the following distinguished triangle in the
derived Fukaya category $\fuk(X)$:

\begin{equation} \label{eq:ex-tr} \xymatrix{
    \mathcal{S} \otimes CF(S,Q) \ar[r] & \tau(\mathcal{Q}) \ar[d] \\
    & \mathcal{Q} \ar[ul] }
\end{equation}
Here $\mathcal{S}$, $\mathcal{Q}$ and $\tau(\mathcal{Q})$ stand for
the $A_{\infty}$-modules corresponding to $S$, $Q$ and $\tau(Q)$ under
the Yoneda embedding.

The above distinguished triangle implies that, up to a
quasi-isomorphism of modules, $\mathcal{Q}$ can be expressed as the
following mapping cone:
\begin{equation} \label{eq:Q-mc}
  \mathcal{Q} \cong [\mathcal{S} \otimes CF(S,Q) \longrightarrow
  \tau(\mathcal{Q})].
\end{equation}
By rotating~\eqref{eq:ex-tr} we obtain also the following
quasi-isomorphism:
\begin{equation} \label{eq:Q-mc-2} \tau(\mathcal{Q}) \cong
  [\mathcal{Q} \longrightarrow \mathcal{S} \otimes CF(S,Q)].
\end{equation}
Note that here and in what follows we work in an ungraded setting,
hence no grading shifts appear in any
of~\eqref{eq:ex-tr}~-~\eqref{eq:Q-mc-2}.

We now turn to the cone decomposition~\eqref{eq:L-cone}, and assume
that $(X, d\lambda) = (E', \lambda_{E'})$ as in~\S\ref{sb:extended}.
The decomposition~\eqref{eq:L-cone} follows by successively
applying~\eqref{eq:Q-mc} and~\eqref{eq:Q-mc-2}. Specifically, we begin
with $\mathcal{K}^{(1)} = \tau_{S_1}(\mathcal{K})$ and obtain
from~\eqref{eq:Q-mc}:
\begin{equation} \label{eq:K-1} \mathcal{K} \cong [\mathcal{S}_1
  \otimes CF(S_1, K) \longrightarrow \mathcal{K}^{(1)}].
\end{equation}
By the same argument we also have
$\mathcal{K}^{(1)} \cong [\mathcal{S}_2 \otimes CF(S_1, K^{(1)})
\longrightarrow \mathcal{K}^{(2)}]$, which together
with~\eqref{eq:K-1} gives:
\begin{equation} \label{eq:K-2} \mathcal{K} \cong [\mathcal{S}_1
  \otimes CF(S_1, K) \longrightarrow \mathcal{S}_2 \otimes CF(S_2,
  K^{(1)}) \longrightarrow \mathcal{K}^{(2)}].
\end{equation}
But by~\eqref{eq:Q-mc-2} we have
$\mathcal{K}^{(1)} \cong [\mathcal{K} \longrightarrow \mathcal{S}_1
\otimes CF(S_1, K)]$. Substituting this into~\eqref{eq:K-2} yields:
\begin{equation} \label{eq:K-3} \mathcal{K} \cong [\mathcal{S}_1
  \otimes CF(S_1, K) \longrightarrow \mathcal{S}_2 \otimes CF(S_2, K)
  \longrightarrow \mathcal{S}_2 \otimes CF(S_2, S_1) \otimes CF(S_1,K)
  \longrightarrow \mathcal{K}^{(2)}]. 
\end{equation}
Continuing in a similar vein, decomposing $\mathcal{K}^{(2)}$,
$\mathcal{K}^{(3)}$ etc.~ we obtain the cone
decomposition~\eqref{eq:L-cone} with items as described
in~\eqref{B_j-cone}~-~\eqref{eq:Bjd}.

It remains to address the acyclicity of the module
$\text{Inc}_{\mathcal{R}, E'}^* \mathcal{K}^{(k)}$. \label{pp:acyclic}
(Recall $K^{(k)} = \tau_{S_k} \cdots \tau_{S_1} (K)$). This follows
from~\cite[\S4.4]{Bi-Co:lefcob-pub}, where it is proved that there is
a Hamiltonian diffeomorphism $\phi: E' \longrightarrow E'$ such that
$\phi(K^{(k)}) \subset E'|_{\{\text{Im}\, z \leq -\nu\}}$. (See
also~\cite{Bi-Co:lcob-fuk-arxiv} for more details.) In particular, for
every Lagrangian submanifold $\vbl \subset \pi'^{-1}(\mathcal{R})$ we
have $CF(\vbl, \phi(K^{(k)})) = 0$.

\subsection{Taking filtrations into account} \label{sb:filt-cd-Lef} We
now go back to the cone decomposition~\eqref{eq:L-cone} and review it
from the perspective of action filtrations.

From now on we assume all the exact Lagrangian submanifolds to be
marked, unless otherwise stated. By a slight abuse of notation, we now
redefine the objects of the Fukaya categories $\fuk(E)$, $\fuk(E')$,
as well as $\fuk(E; \mathcal{R})$, $\fuk(E; -\nu)$, to be {\em marked}
exact Lagrangians, subject to the additional constraints in each of
these categories. These categories now become weakly filtered
$A_{\infty}$-categories, where the filtrations are induced by the
action functional. We refer the reader
to~\cite[\S2]{Bi-Co-Sh:lshadows-long} for the definitions and basic
theory of weakly filtered $A_{\infty}$-categories and weakly filtered
modules over such.

Below we will take the exact Lagrangian
$K \subset E|_{\{ |\text{Im}| \, z < \epsilon\}}$ to have an arbitrary
marking. This marking induces a marking on
$K^{(j)} = \tau_{S_j} \cdots \tau_{S_1} (K)$, $j=1, \ldots, k$,
see~\S\ref{sb:prfs-filt-cd}, page~\pageref{pp:mark-tau}. The
Lagrangian spheres $S_j$ are also assumed to be marked in advance.

Note that all the items in the cone decomposition~\eqref{eq:L-cone},
as detailed in~\eqref{B_j-cone}~-~\eqref{eq:Bjd} are weakly filtered
modules. This is so because the $\mathcal{S}_j$'s and
$\mathcal{K}^{(k)}$ are Yoneda modules over a weakly filtered
$A_{\infty}$-category, and the chain complexes
$CF(S_{i_l}, S_{i_{l-1}})$ and $CF(S_j, K)$ are filtered.

Next, we claim that all the maps in the iterated
cones~\eqref{eq:L-cone},~\eqref{B_j-cone} and~\eqref{eq:Bjd} are
weakly filtered maps. This means, in particular, that when evaluating
these iterated cones modules on a given exact Lagrangian $\vbl$, each
of these maps specializes to a filtered chain map that shifts
filtrations by an amount bounded from above {\em uniformly in $\vbl$}.
More specifically:

\ocnote{
\begin{prop}\label{prop:filtered-cone} In the iterated cone~\eqref{eq:Bjd} 
\begin{equation} \label{eq:Bjd-2} \mathcal{B}_{j, d} =
  [\mathcal{C}_{\underline{i}^{(1)}, j} \xrightarrow{\; \varphi_{1,j}
    \;} [\mathcal{C}_{\underline{i}^{(2)}, j} \xrightarrow{\;
    \varphi_{2,j}\;} [ \cdots \xrightarrow{\; \; \; \;}
  [\mathcal{C}_{\underline{i}^{(m_{d, j-1}-1)}, j} \xrightarrow{\;
    \varphi_{m_{d,j-1}-1, j}\;}
  \mathcal{C}_{\underline{i}^{(m_{d,j-1})}, j}] \cdots ]]],
\end{equation}
each of the module homomorphisms $\varphi_{l,j}$ is weakly filtered,
and shifts action by $\leq s_{\varphi_{l,j}}$, for some
$s_{\varphi_{l,j}} \geq 0$.
\end{prop}
}
This implies that the right-hand side of~\eqref{eq:Bjd-2} is
 filtered using the filtrations of the factors
$\mathcal{C}_{\underline{i}^{(l)}, j}$ and the
recipe~\eqref{eq:cone-filtration}.

In particular, for every exact Lagrangian $\vbl$, the module
homomorphism $\varphi_{l,j}$ specializes to an
$s_{\varphi_{l,j}}$-filtered chain map (still denoted by
$\varphi_{l,j}$):
$$\varphi_{l,j}: \mathcal{C}_{\underline{i}^{(l)}, j}(\vbl)
\longrightarrow [ \mathcal{C}_{\underline{i}^{(l+1)}, j}(\vbl)
\xrightarrow{\; \varphi_{l+1,j} \;} [ \cdots \xrightarrow{\; \; \; \;}
[\mathcal{C}_{\underline{i}^{(m_{d, j-1}-1)}, j}(\vbl) \xrightarrow{\;
  \varphi_{m_{d,j-1}-1,j}\;}
\mathcal{C}_{\underline{i}^{(m_{d,j-1})},j}(\vbl)] \cdots ]].$$ A
crucial point for us will be that the filtration-shifts
$s_{\varphi_{l,j}}$ are independent of $\vbl$.

Having filtered the modules $\mathcal{B}_{j,d}$, the preceding
statements apply also to the maps in the iterated cone
of~\eqref{B_j-cone}, and finally also to the right-hand side
of~\eqref{eq:L-cone}. We will prove Proposition \ref{prop:filtered-cone} in~\S\ref{sb:prfs-filt-cd} below.

Furthermore, we claim that the module quasi-isomorphism
at~\eqref{eq:L-cone} between $\mathcal{K}$ and the (now weakly
filtered) iterated cone on the right-hand side is filtered in the
following sense. 

\ocnote{
\begin{prop}\label{prop:quasi-shift} There exist $s_{\mathcal{K}} \geq 0$ and
weakly-filtered module homomorphisms
$$\varphi: \mathcal{K} \longrightarrow [\mathcal{B}_1 \longrightarrow
\cdots \longrightarrow \mathcal{B}_k \longrightarrow
\mathcal{K}^{(k)}], \quad \psi: [\mathcal{B}_1 \longrightarrow \cdots
\longrightarrow \mathcal{B}_k \longrightarrow \mathcal{K}^{(k)}]
\longrightarrow \mathcal{K}$$ that shift filtrations by
$\leq s_{\mathcal{K}}$ and such that
$$\varphi \circ \psi = \id + \mu_1^{\text{mod}}(h'), \quad 
\psi \circ \varphi = \id + \mu_1^{\text(mod)}(h'')$$ for weakly
filtered pre-module homomorphisms $h', h''$ that shift filtrations by
$\leq s_{\mathcal{K}}$. 
\end{prop}
}

The proof of this statement is again postponed to
~\S\ref{sb:prfs-filt-cd}.
The constant $s_{\mathcal{K}}$ depends on $K$
(and its marking) as well as on the marking on the spheres
$S_1, \ldots, S_k$. 

In particular, the above implies that for every exact Lagrangian
$\vbl$ we have chain maps
\begin{equation} \label{eq:vphi-psi-v}
  \begin{aligned}
    & \varphi_{\vbl}: CF(\vbl,K) \longrightarrow [\mathcal{B}_1(\vbl)
    \longrightarrow \cdots \longrightarrow \mathcal{B}_k(\vbl)
    \longrightarrow CF(\vbl,
    K^{(k)})], \\
    & \psi_{\vbl}: [\mathcal{B}_1(\vbl) \longrightarrow \cdots
    \longrightarrow \mathcal{B}_k(\vbl) \longrightarrow CF(\vbl,
    K^{(k)})] \longrightarrow CF(\vbl,K),
  \end{aligned}
\end{equation}
which are $s_{\mathcal{K}}$-filtered and such that
$\varphi_{\vbl} \circ \psi_{\vbl}$ and
$\psi_{\vbl} \circ \varphi_{\vbl}$ are chain homotopic to the
identities via chain homotopies that shift filtrations by
$\leq s_{\mathcal{K}}$. Once again, it is important to stress that the
bound on the action shift $s_{\mathcal{K}}$ is independent of $\vbl$.

Phrased in the terminology of Definition~\ref{d:dit-mod}, the above
says that the module $\mathcal{K}$ (resp. filtered chain complex
$CF(\vbl,K)$) and the module on the right-hand side
of~\eqref{eq:L-cone} (resp. the filtered chain complex
$[\mathcal{B}_1(\vbl) \longrightarrow \cdots \longrightarrow
\mathcal{B}_k(\vbl) \longrightarrow CF(\vbl, K^{(k)})]$) are at
distance $\leq s_{\mathcal{K}}$ one from the other.

Finally, recall that the pullback module
$\text{Inc}_{\mathcal{R}, E'}^* \mathcal{K}^{(k)}$ is acyclic. We
claim that this acyclicity holds also in the filtered sense. Namely,
there exists a constant $s_{C} = s_{C}(K)$, which depends on $K$, and
a weakly filtered pre-module homomorphism
$h: \text{Inc}_{\mathcal{R}, E'}^* \mathcal{K}^{(k)} \longrightarrow
\text{Inc}_{\mathcal{R}, E'}^* \mathcal{K}^{(k)}$ that shifts action
by $\leq s_{C}$ such that in
$\hom_{\text{mod}_{\fuk(E; \mathcal{R})}}(\text{Inc}_{\mathcal{R},
  E'}^* \mathcal{K}^{(k)}, \text{Inc}_{\mathcal{R}, E'}^*
\mathcal{K}^{(k)})$ we have $id = \mu_1^{\text{mod}}(h)$. In
particular, for every exact Lagrangian
$\vbl \subset \pi^{-1}(\mathcal{R})$ we have:
\begin{equation} \label{eq:beta-sc} \beta(CF(\vbl,\mathcal{K}^{(k)}))
  \leq s_C.
\end{equation}
Here, $\beta(CF(\vbl,\mathcal{K}^{(k)}))$ is the boundary depth of the
acyclic filtered chain complex $CF(\vbl,\mathcal{K}^{(k)})$.

The inequality~\eqref{eq:beta-sc} follows from the last paragraph of~\S\ref{sb:ex-t-Dehn} on
page~\pageref{pp:acyclic}. Indeed, by standard Floer theory we can
take $s_C = 2 \rho(\id, \phi)$, where $\phi: E' \longrightarrow E'$ is
a Hamiltonian diffeomorphism that sends $K^{(k)}$ to
$E'|_{\{\text{Im}\, z \leq -\nu\}}$, and $\rho$ stands for the Hofer
metric on the group of Hamiltonian diffeomorphisms.

\begin{rem} \label{r:sC-constant} The constant $s_{C}$ appearing
  in~\eqref{eq:beta-sc} depends apriori on $K$ (though not on $L$). A
  more careful argument, based on~\cite[\S4.4]{Bi-Co:lefcob-pub},
  shows that the Hamiltonian diffeomorphisms $\phi$, mentioned above,
  can be taken to be at a uniformly bounded (in $K$) Hofer-distance
  from $\id$, as long as we restrict to Lagrangians
  $K \subset E|_{\{ |\text{Im}| \, z < \epsilon\}}$. Consequently the
  constant $s_C$ can be assumed to be independent of $K$.

  However, this additional information will not be used in the rest of
  the paper. The reason is that we will use the filtered cone
  decomposition~\eqref{eq:L-cone} only for one Lagrangian $K$, namely
  $K=N$ - the zero-section of $T^*(N)$ viewed as a Lagrangian in $E$.
\end{rem}

\subsection{Proof of the statements from~\S\ref{sb:filt-cd-Lef}}
\label{sb:prfs-filt-cd}
We continue to assume here all exact Lagrangian submanifolds (and
cobordisms) to be marked.

We begin with a brief digression on inclusion and product functors.
Let $(Y, d\lambda_{Y})$ be a Liouville manifold as
in~\S\ref{sbsb:HF}. Let
$\gamma : \mathbb{R} \longrightarrow \mathbb{R}^2$ be a smooth proper
embedding sending the ends of $\mathbb{R}$ to horizontal rays in
$\mathbb{R}^2$. By abuse of notation we denote by $\gamma$ also the
image of this embedding. By the results
of~\cite{Bi-Co:lcob-fuk,Bi-Co-Sh:lshadows-long} there is a weakly
filtered $A_{\infty}$-functor (called in~\cite{Bi-Co:lcob-fuk}
``inclusion functor'')
$\mathcal{I}_{\gamma}: \fuk(Y) \longrightarrow
\fuk_{\text{cob}}(\mathbb{R}^2 \times Y)$ which sends the object
$\vbl \subset Y$ to
$\mathcal{I}_{\gamma}(\vbl) = \gamma \times \vbl \subset \mathbb{R}^2
\times Y$. Here $\fuk(Y)$ stands for the Fukaya category of closed
$\lambda_Y$-exact Lagrangians in $Y$ and
$\fuk_{\text{cob}}(\mathbb{R}^2 \times Y)$ for the Fukaya category of
exact cobordisms in $\mathbb{R}^2 \times Y$, with respect to the
$1$-form $xdy \oplus \lambda_Y$.

Let $(X, \omega = d\lambda)$ be a Liouville manifold as
in~\S\ref{sbsb:HF}. We denote by $X^{-}$ the manifold $X$ endowed with
the symplectic structure $-\omega$. Take $Y = X \times X^{-}$, endowed
with the symplectic structure $\omega \oplus -\omega$ and Liouville
form $\widetilde{\lambda} := \lambda \oplus -\lambda$ (playing the
role of $\lambda_Y$). Fix
$\widetilde{\lambda}' := xdy \oplus \lambda \oplus -\lambda$ as the
primitive of $\omega_{\mathbb{R}^2} \oplus \omega \oplus -\omega$.

Fix an exact Lagrangians $Q \subset X$. A slight variation on the
inclusion functor $\mathcal{I}_{\gamma}$ is the $A_{\infty}$-functor
$\mathcal{I}_{\gamma, Q}: \fuk(X) \longrightarrow
\fuk_{\text{cob}}(\mathbb{R}^2 \times X \times X^{-})$ which sends an
exact Lagrangians $\vbl \subset X$ to
$\mathcal{I}_{\gamma, Q}(\vbl) := \gamma \times \vbl \times Q$. The
construction of this functor is very similar to the construction of
$\mathcal{I}_{\gamma}$ (for the case $Y = X \times X^{-}$), as
detailed in~\cite{Bi-Co:lcob-fuk}. In fact, $\mathcal{I}_{\gamma, Q}$
factors as
$\mathcal{I}_{\gamma, Q} := \mathcal{I}_{\gamma} \circ \mathcal{P}_Q$,
where $\mathcal{P}_Q: \fuk(X) \longrightarrow \fuk(X \times X^{-})$ is
the obvious functor that sends $\vbl \subset X$ to
$\vbl \times Q \subset X \times X^{-}$.

The main ingredient to show Propositions \ref{prop:filtered-cone} and \ref{prop:quasi-shift} is to
 establish a filtered version of the Seidel's Dehn-twist
triangle~\eqref{eq:ex-tr} (or more precisely~\eqref{eq:Q-mc}).  We pursue this now.

\ocnote{
\begin{lem}\label{lem:Dehn-filtered}
The mapping cone in equation (\ref{eq:Q-mc}) admits a filtered version.
\end{lem}
}
In the course of the proof we will indicate more precisely the relevant shifts involved and their 
dependence on the choices involved in the construction.

\begin{proof}

Let
$(X^{2n}, \omega=d\lambda)$ be a Liouville manifold as
in~\S\ref{sbsb:HF} and $S \subset X$,
$\tau = \tau_S: X \longrightarrow X$ be as at the beginning
of~\S\ref{sb:ex-t-Dehn}. It is possible to choose $\tau$ (a
representative of the Dehn-twist symplectic mapping class) such that
$\tau$ is supported near $S$ and moreover such that
$\tau^*\lambda = \lambda + d h_{\tau}$, where
$h_{\tau}: X \longrightarrow \mathbb{R}$ is a smooth function
compactly supported near $S$. (The latter easily follows from the fact
that given any neighborhood of the zero-section in $T^*(S^n)$, there
is a model Dehn-twist $T^*(S^n) \longrightarrow T^*(S^n)$ supported in
that neighborhood which is $\lambda_{\text{can}}$-exact, and the fact
that the sphere $S$ is $\lambda$-exact.) Note that we have:
$(\tau^{-1})^* \lambda = \lambda - d(h_{\tau} \circ \tau^{-1})$.

\label{pp:mark-tau}
Let $Q \subset X$ be a marked exact Lagrangian with primitive
$h_Q: Q \longrightarrow \mathbb{R}$ for $\lambda|_{Q}$. Then $\tau(Q)$
is also a marked exact Lagrangian. Indeed,
$h_{\tau(Q)}: \tau(Q) \longrightarrow \mathbb{R}$ defined by
$$h_{\tau(Q)}(x) := h_Q(\tau^{-1}(x)) + h_{\tau}(\tau^{-1}(x))$$ is a
primitive of $\lambda|_{\tau(Q)}$. We will use this function to mark
$\tau(Q)$.

We now get back to Dehn-twists, from the perspective of Lagrangian
cobordism. By a result of Mak-Wu~\cite{Mak-Wu:Dehn-twist} there exists
an exact Lagrangian cobordism
$W \subset \mathbb{R}^2 \times X \times X^{-}$ with two negative ends
and one positive end, as follows. The upper negative end is
$S \times S$ and the lower negative end is the graph
$\Gamma_{\tau^{-1}}$ of $\tau^{-1}$. The positive end is the graph of
the identity map (i.e.~the diagonal in $X \times X^{-}$). See
Figure~\ref{f:MW-cob}.

\begin{figure}[htbp]
  \begin{center}
    \includegraphics[scale=0.68]{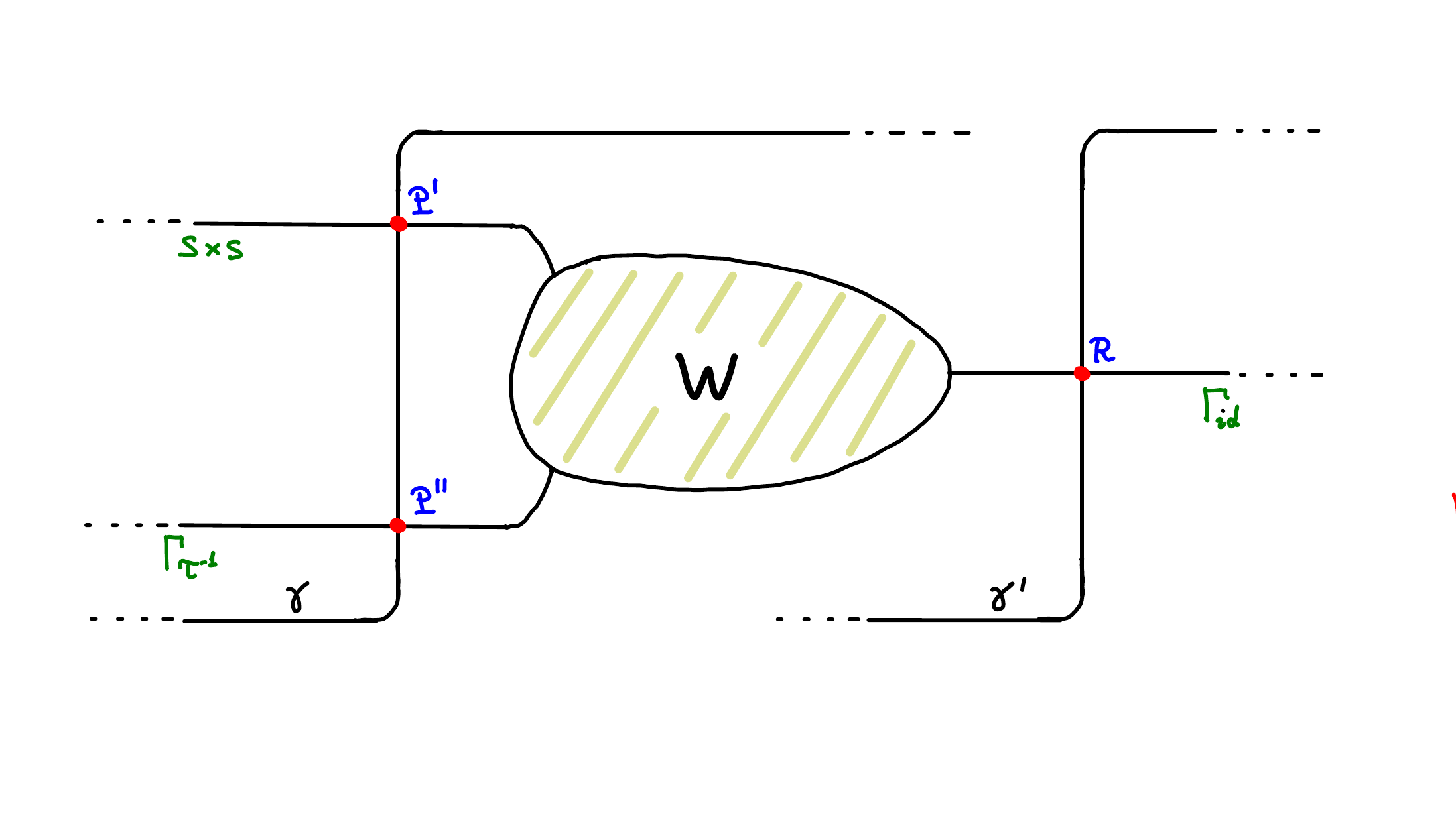}
  \end{center}
  \caption{Projection to $\mathbb{R}^2$ of the Mak-Wu cobordism
    $W \subset \mathbb{R}^2 \times X \times X^{-}$, and the curves
    $\gamma$, $\gamma'$.}
  \label{f:MW-cob}
\end{figure}

Let $\gamma \subset \mathbb{R}^2$ be the curve depicted in
Figure~\ref{f:MW-cob}, and denote by $\mathcal{W}$ the Yoneda module
corresponding to
$W \in \textnormal{Ob}(\fuk_{\text{cob}}(\mathbb{R}^2 \times X \times
X^{-})$. Denote also by $\mathcal{S} \times \mathcal{S}$,
$\tau(\mathcal{Q})$ the Yoneda modules (over $\fuk(X \times X^{-})$
corresponding to the Lagrangians $S \times S$ and $\tau(\mathcal{K})$,
respectively. Ignoring filtrations for the moment, a straightforward
calculation (based on the theory from~\cite{Bi-Co:lcob-fuk}) shows
that the pullback module $\mathcal{I}_{\gamma, Q}^*\mathcal{W}$
coincides with a mapping cone
\begin{equation} \label{eq:IW} \mathcal{I}_{\gamma, Q}^*\mathcal{W} =
  [\mathcal{S} \otimes CF(S, Q) \xrightarrow{\; \varphi \;}
  \tau(\mathcal{Q})]
\end{equation}
for some module homomorphism
$\varphi: \mathcal{S} \otimes CF(S, Q) \longrightarrow
\tau(\mathcal{Q})$.

Consider now the curve $\gamma' \subset \mathbb{R}^2$ from
Figure~\ref{f:MW-cob}. Ignoring filtrations again, it is easy to see
that $\mathcal{I}_{\gamma', Q}^*\mathcal{W} = \mathcal{Q}$, the Yoneda
module corresponding to $Q \subset X$.

The curves $\gamma$ and $\gamma'$ are isotopic via a Hamiltonian
isotopy which is horizontal at infinity. Therefore the modules
$\mathcal{I}_{\gamma, Q}^*\mathcal{W}$ and
$\mathcal{I}_{\gamma', Q}^*\mathcal{W}$ are quasi-isomorphic (in the
category $\text{mod}_{\fuk(X)}$). Thus we have a quasi-isomorphism
\begin{equation} \label{eq:Q-mc-3} \mathcal{Q} \cong [\mathcal{S}
  \otimes CF(S, Q) \xrightarrow{\; \varphi \;} \tau(\mathcal{Q})].
\end{equation}
Our goal now is to derive a coarse filtered version
of~\eqref{eq:Q-mc-3}. More specifically, we will have to address two
thing: explain why the module homomorphism $\varphi$ is filtered, and
then show that the quasi-isomorphism in~\eqref{eq:Q-mc-3} is weighted
in the sense of Definition~\ref{d:dit-mod}.

Note that $\widetilde{\lambda}'$ coincides with $\widetilde{\lambda}$
along each horizontal end of $W$ (because $xdy$ vanishes along
horizontal rays). We also have
$$\widetilde{\lambda}|_{\Gamma_{\id}} = 0, \quad
\widetilde{\lambda}|_{S \times S} = \lambda_S \oplus -\lambda_S, \quad
\widetilde{\lambda}|_{\Gamma_{\tau^{-1}}} = d(h_{\tau} \circ
\tau^{-1}),$$ where $\lambda_S := \lambda|_S$.  Let
$h_W: W \longrightarrow \mathbb{R}$ be a primitive of
$\widetilde{\lambda}'|_W$. By the above, $h_W$ restricts along each of
the ends of $W$ to a primitive function for the restriction of
$\widetilde{\lambda}$ to the Lagrangian corresponding to that end. We
will use these functions, denoted by $h_{W, \Gamma_{\id}}$,
$h_{W, S \times S}$ and $h_{W, \Gamma_{\tau^{-1}}}$, for primitives of
$\widetilde{\lambda}|_{\Gamma_{\id}}$,
$\widetilde{\lambda}|_{S \times S}$ and
$\widetilde{\lambda}|_{\Gamma_{\tau^{-1}}}$ respectively.  Note that
$h_{W, \Gamma_{\id}}$ is constant, and by subtracting this constant
from $h_W$ we may assume without loss of generality that
$h_{W, \Gamma_{\id}} \equiv 0$. (Note that the exact Lagrangian
cobordism $W$ does not come with a preferred marking, and we are free
to choose $h_W$ as we wish.)

Pick any marking on $S$, i.e.~a primitive function
$h_S: S \longrightarrow \mathbb{R}$ for $\lambda_S$. We have:
\begin{equation} \label{eq:h_W-SxS}
  \begin{aligned}
    & h_{W, S \times S} (x,y) = h_S(x) - h_S(y) + C_{W, S \times S},
    \;
    \forall \, (x, y) \in S \times S, \\
    & h_{W, \Gamma_{\tau^{-1}}}(x, \tau^{-1}(x)) =
    h_{\tau}(\tau^{-1}(x)) + C_{W, \Gamma_{\tau^{-1}}}, \; \forall x
    \in X,
  \end{aligned}
\end{equation}
for some constants $C_{W, S \times S}$, $C_{W, \Gamma_{\tau^{-1}}}$.
Fix a primitive $h_{\gamma}: \gamma \longrightarrow \mathbb{R}$ of
$(xdy)|_{\gamma}$. Note that $h_{\gamma}$ is constant along the
positive and negative ends of $\gamma$. Given any marked exact
Lagrangian $\vbl \subset X$, with a primitive function
$h_{\vbl}: \vbl \longrightarrow \mathbb{R}$ for $\lambda|_{\vbl}$, we
will use the function
$h_{\gamma \times \vbl \times Q} := h_\gamma + h_{\vbl} - h_Q$ as a
primitive for $\widetilde{\lambda}'|_{\gamma \times \vbl \times Q}$.

Consider the Floer complex $CF(\gamma \times \vbl \times Q, W)$ with
Floer data consisting of a zero Hamiltonian and any regular almost
complex structure. (We assume here without loss of generality that
$(\vbl \times Q) \pitchfork S \times S$ and
$\vbl \times Q \pitchfork \Gamma_{\tau^{-1}}$.)

Given two exact Lagrangians $L'$, $L''$ in a Liouville manifold
$(Y, d\lambda_Y)$, endowed with primitives
$h_{L'}: L' \longrightarrow \mathbb{R}$,
$h_{L''}: L'' \longrightarrow \mathbb{R}$ for $\lambda_{Y}|_{L'}$ and
$\lambda_Y|_{L''}$, and given a Floer datum for $(L',L'')$ we denote
by $\mathcal{A}(-; (L', L''))$ the action functional associated to the
given Floer datum and the choices of primitives $h_{L'}$, $h_{L''}$.
Here $-$ stands for a path connecting a point from $L'$ to a point in
$L''$.

We will now examine the action functional $\mathcal{A}$ for the pairs
$(\gamma \times \vbl \times Q, W)$, $(\vbl,S)$ and $(S,Q)$. As before,
we use here Floer data with zero Hamiltonian terms. We begin with
calculating $\mathcal{A}$ on the intersection points of
$(\gamma \times \vbl \times Q) \cap W$ (viewed as constant
paths). These intersection points fall into two types:
\begin{enumerate}
\item $(P', x_1, x_2)$, where $P' \in \mathbb{R}^2$ is as depicted in
  Figure~\ref{f:MW-cob} and $x_1, x_2 \in S$.
\item $(P'', x_1, x_2)$, where $P'' \in \mathbb{R}^2$ is as in
  Figure~\ref{f:MW-cob} and $x_1 \in \vbl \cap \tau(Q)$,
  $x_2 = \tau^{-1}(x_1)$.
\end{enumerate}

For the points of the 1'st type we have:
\begin{equation} \label{eq:A-P'}
  \begin{aligned}
    \mathcal{A}(P', x_1, x_2; & (\gamma \times \vbl \times Q, W)) \\
    & = h_S(x_1) - h_S(x_2) + C_{W, S \times S} - h_{\gamma}(p') -
    (h_{\vbl}(x_1)-h_Q(x_2)) \\
    & = (h_S(x_1) - h_{\vbl}(x_1)) + (h_Q(x_2) - h_S(x_2)) + (C_{W, S
      \times S} - h_{\gamma}(P')) \\
    & = \mathcal{A}(x_1; (\vbl,S)) + \mathcal{A}(x_2; (S,Q)) + (C_{W,
      S \times S} - h_{\gamma}(P')).
  \end{aligned}
\end{equation}
Note that the sum of the first two terms in the last equality is
precisely the action-level of the generator
$x_1 \otimes x_2 \in CF(\vbl,S) \otimes CF(S, Q)$.

Turning to the intersection points of the 2'nd type, we have:
\begin{equation} \label{eq:A-P''}
  \begin{aligned}
    \mathcal{A}(P'', x_1, x_2; & (\gamma \times \vbl \times Q) \\
    & = h_{\tau}(\tau^{-1}(x_1)) + C_{W, \Gamma_{\tau^{-1}}} -
    h_{\gamma}(P'') - h_{\vbl}(x_1) + h_Q(\tau^{-1}(x_1)) \\
    & = h_{\tau(Q)}(x_1) - h_{\vbl}(x_1) + C_{W, \Gamma_{\tau^{-1}}} -
    h_{\gamma}(P'') \\
    & = \mathcal{A}(x_1; (\vbl, \tau(Q))) + (C_{W, \Gamma_{\tau^{-1}}}
    - h_{\gamma}(P'')).
  \end{aligned}
\end{equation}

Now recall from~\eqref{eq:IW} that
$$CF(\gamma \times \vbl \times Q, W) = [CF(\vbl,S) \otimes CF(S, Q)
\xrightarrow{\; \varphi \;} \tau(Q)],$$ and that by the results
of~\cite{Bi-Co:lcob-fuk} counts Floer strips going from the
intersection points of type~1 to points of type~2.

From the standard action-energy identity we obtain the following: if
the generator $x \in \vbl \cap \tau(Q)$ of $CF(\vbl, \tau(Q))$
participates in $\varphi(x_1 \otimes x_2)$, then:

\begin{equation} \label{eq:A-vSQ-ineq} \mathcal{A}(x_1; (\vbl,S)) +
  \mathcal{A} (x_2; (S, Q)) + C_{W, S \times S} - h_\gamma(P') \geq
  \mathcal{A}(x; (\vbl, \tau(Q))) + C_{W, \Gamma_{\tau^{-1}}} -
  h_{\gamma}(P'').
\end{equation}


It follows that $\varphi$ shifts action by
\begin{equation} \label{eq:s-vphi} s_{\varphi} \leq h_{\gamma}(P'') -
  h_{\gamma}(P') + C_{W, S \times S} - C_{W, \Gamma_{\tau^{-1}}}.
\end{equation}
The latter quantity is a constant which is independent of $Q$ and
$\vbl$.

Next, consider the curve $\gamma'$ from Figure~\ref{f:MW-cob} and
$\mathcal{I}_{\gamma', Q}: \fuk(X) \longrightarrow
\fuk_{\text{cob}}(\mathbb{R}^2 \times X \times X^{-})$. Recall that up
to a filtration shift we have
$\mathcal{I}_{\gamma', Q}^* \mathcal{W} = \mathcal{Q}$, therefore
$CF(\vbl \times Q, \Gamma_{\id}) \cong CF(\vbl,Q)$, again up to a
filtration shift. We will now determine this shift. To this end,
recall first that $h_{W, \Gamma_{\id}} \equiv 0$. The intersection
points of $(\gamma' \times \vbl \times Q) \cap W$ are of the type
$(R, x,x)$, $x \in \vbl \cap Q$. Calculating the action on such points
we get:
\begin{equation} \label{eq:A-R} \mathcal{A}(R,x,x; (\gamma' \times
  \vbl \times Q, W)) = -h_{\gamma'}(R) - h_{\vbl}(x) + h_Q(x) =
  \mathcal{A}(x; (\vbl,Q)) - h_{\gamma'}(R).
\end{equation}
Therefore the identification
$\mathcal{I}_{\gamma', Q}^* \mathcal{W} = \mathcal{Q}$ holds up to an
action shift of the constant $h_{\gamma'}(R)$.

Finally, there exists a constant $S(W) \geq 0$ that depends only on
$W$ and another constant $C(h_{\gamma}, h_{\gamma'}) \geq 0$ that
depends only on the choices of the primitives $h_{\gamma}$,
$h_{\gamma'}$ such that the following holds. There exist weakly
filtered module homomorphisms
$\phi: \mathcal{I}_{\gamma, Q}^*\mathcal{W} \longrightarrow
\mathcal{I}_{\gamma', Q}^*\mathcal{W}$ and
$\phi': \mathcal{I}_{\gamma', Q}^*\mathcal{W} \longrightarrow
\mathcal{I}_{\gamma, Q}^*\mathcal{W}$ that shift filtrations by
$\leq S(W) + C(h_{\gamma'})$ such that
$\phi' \circ \phi = \id + \mu_1^{\text{mod}}(H)$,
$\phi \circ \phi' = \id + \mu_1^{\text{mod}}(H')$ for some weakly
filtered pre-module homomorphisms $H$, $H'$ that shift filtrations by
$\leq 2(S(W) + C(h_{\gamma'}))$. We refer the reader
to~\cite[\S4]{Bi-Co-Sh:lshadows-long} for more details on this. The
constant $S(W)$ is the {\em shadow} of the cobordism $W$ - namely the
area of the domain in $\mathbb{R}^2$ consisting of the projection of
$W$ to $\mathbb{R}^2$ together with all the bounded connected
components of the complement of this projection.

As a result, we obtain a weakly filtered quasi-isomorphism
\begin{equation} \label{eq:Q-cone} \mathcal{Q} \cong [\mathcal{S}
  \otimes CF(S,Q) \xrightarrow{\; (\varphi, s_{\varphi}) \;}
  \tau(\mathcal{Q})],
\end{equation}
of weight bounded from above by a constant that depends only on $W$
and $\gamma$, $\gamma'$. (See Definition~\ref{d:dit-mod}.)  As seen
above at~\eqref{eq:s-vphi} the amount of shift of $\varphi$ is bounded
from above by a constant $s_{\varphi}$ which does not depend on $Q$.
This concludes the construction of the filtered version of the Seidel
exact triangle. 
\end{proof}

\ocnote{
\begin{rem} There are a number of other ways to construct Seidel's exact triangle associated
to a Dehn twist. Certainly, Seidel's  original construction in \cite{Se:long-exact} and also
 the method in \cite{Bi-Co:lefcob-pub}. These methods  can also be used to deduce filtered 
 versions of the exact traingle. We used here the method
in \cite{Mak-Wu:Dehn-twist} as it appears to provide the fastest approach in our context.
\end{rem}
}

Propositions \ref{prop:filtered-cone} and \ref{prop:quasi-shift} now follow by applying the
procedure indicated at the end of~\S\ref{sb:ex-t-Dehn}, but now using
the filtered version of~\eqref{eq:Q-mc} and~\eqref{eq:Q-mc-2}, as in Lemma \ref{lem:Dehn-filtered}, in
conjunction with the algebraic remarks contained in Proposition~\ref{p:equiv-cones} and the statement
from the beginning of~\S\ref{sb:icones}.

\begin{remnonum}
  The weight of the quasi-isomorphism at~\eqref{eq:Q-cone} as well as
  $s_{\varphi}$ do depend (also) on $W$ (hence on the specific choice
  of the representative $\tau$ of the symplectic mapping class of the
  Dehn-twist), however these choices are made in advance, once and for
  all. The dependencies of this weight and of $s_{\varphi}$ on
  $\gamma$, $\gamma'$ and $h_{\gamma'}$, $h_{\gamma'}$ can in fact be
  eliminated by estimating more sharply the shifts in $\phi$, $\phi'$,
  $H$, $H'$ above. However this is not needed for our purposes.
\end{remnonum}


\section{Proof of the main theorem} \label{s:prf-main}

This section contains two parts. The first, and main part, provides the proof of Theorem \ref{t:main-a}.
The second is concerned with the converse of the statement, as indicated in Remark \ref{r:main-thm} (1).

\subsection{The spectral norm bound in equation (\ref{eq:bound-main-thm})}\label{subsec:proof-main}
For the proof of the main theorem we will need the following Lemma.
Fix a tubular neighborhood $\mathcal{V} = T^*_{\leq r_0}(N)$ of the
zero-section. For $q \in N$ denote by
$F_q = T^*_q(N) \cap \mathcal{V}$ the part of the cotangent fiber over
$q$ that lies inside $\mathcal{V}$. We endow the exact Lagrangians
$F_q$ with the $0$ function as a primitive of
$\lambda_{\text{can}}$. Note that for every marked exact Lagrangian
$L \subset \mathcal{V}$ and every $q \in N$ we have
$HF(L, F_q) \cong \mathbb{Z}_2$, hence
$\sigma_+(CF(L,F_q)) = \sigma_-(CF(L, F_q))$. We denote this number by
$\sigma(CF(L,F_q))$.

\begin{lem} \label{l:diff-fibers} There exist constants
  $C = C(\mathcal{V})>0$ and $C' = C'(\mathcal{V})>0$, that depend
  only on $\mathcal{V}$, such that for every marked exact Lagrangian
  $L \subset \textnormal{Int\,}(\mathcal{V})$ and every
  $q', q'' \in N$ we have
  $$|\sigma(CF(L,F_{q'})) - \sigma(CF(L,F_{q''}))| \leq C, \quad
  |\beta(CF(L,F_{q'})) - \beta(CF(L,F_{q''}))| \leq C'.$$
\end{lem}

\begin{proof}
  The proof is based on standard arguments, hence we will only outline
  it.

  The statements in the Lemma follows from the following somewhat
  stronger statement: {\sl All the $\fuk(\mathcal{V})$-modules
    corresponding to $F_{q}$, $q \in N$, are at a bounded distance one
    from the other in the sense of Definition~\ref{d:dit-mod}.}

  Here is an outline of the proof of the stronger statement.  Since
  $N$ is compact, it is enough to prove the statement locally for
  $q \in N$. Fix $q_0 \in N$ and let $B' \subset N$ be a ball chart
  around $q_0$ and $\overline{B} \subset B'$ a smaller closed ball
  around $q_0$.

  We claim that there exists $r'_0 > r_0$, a compact subset
  $K \subset B'$ and a family of Hamiltonian functions
  $H^{(q)}: [0,1] \times T^*(N) \longrightarrow \mathbb{R}$,
  parametrized by $q \in \overline{B}$, such that the following holds:
  \begin{enumerate}
  \item All the functions $H^{(q)}$, $q \in \overline{B}$, are
    compactly supported in
    $\mathcal{V}' := T^*_{< r'_0}(N) \cap \pi^{-1}(K)$, where
    $\pi: T^*(N) \longrightarrow N$ is the projection.
  \item The family $H^{(q)}$ depends smoothly on $q \in
    \overline{B}$. In particular, the Hofer norm of the elements of
    the family is uniformly bounded in $q$:
    $\sup_{q \in \overline{B}} \int_0^1 \| H^{(q)}_t \|_{\text{osc}}
    \, dt < \infty$. Here $H_t^{(q)}(x) := H^{(q)}(t,x)$ and for a
    compactly supported $H:T^*(N) \longrightarrow \mathbb{R}$,
    $\|H\|_{\text{osc}}$ stands for its $L^{\infty}$-oscillation norm
    $\|H\|_{\text{osc}} := \max H - \min H$.
  \item $\phi_t^{H^{(q)}} (\mathcal{V}) = \mathcal{V}$ for every
    $t \in [0,1]$, $q \in \overline{B}$.
  \item $\phi_1^{(q)}(F_{q_0}) = F_{q}$ for every
    $q \in \overline{B}$.
  \end{enumerate}
  The existence of a family $H^{(q)}$ with the above properties is
  straightforward.

  Let $q \in \overline{B}$ and
  $L \subset \textnormal{Int\,}\mathcal{V}$ a marked exact
  Lagrangian. Without loss of generality assume that
  $L \pitchfork F_{q_0}$ and $L \pitchfork F_q$.  Pick a regular
  almost complex structure $J$ as in~\S\ref{sbsb:HF}.  For a domain
  $\mathcal{U} \subset T^*(N)$ and two transverse marked exact
  Lagrangians $L', L'' \subset \mathcal{U}$ we denote by
  $CF(L',L''; (0,J); \mathcal{U})$ the Floer complex of $(L', L'')$
  with Floer data $(H \equiv 0,J)$ inside the domain $\mathcal{U}$,
  whenever well defined.

  By standard arguments in Floer theory there is a quasi-isomorphism
  \begin{equation} \label{eq:L-Fq} \varphi_{(q)}: CF(L, T^*_{q_0};
    (0,J); T^*(N)) \longrightarrow CF(L,
    \phi_1^{H^{(q)}}(T^*_{q_0}(N)); (0,J); T^*(N)),
  \end{equation}
  of weight $\leq 2C_1(q) + C_2(q)$, where
  $C_1(q) = \int_0^1 \| H^{(q)}_t \|_{\text{osc}} \, dt$ and $C_2(q)$
  is a constant that depends only on $\mathcal{V}'$ and on the
  $C^2$-size of $H^{(q)}$ in a continuous way. See
  Definition~\ref{d:dit-mod} (and the discussion after it) for
  weighted quasi-isomorphisms. Here we endow the exact Lagrangian
  $\phi_1^{H^{(q)}}(T^*_{q_0}(N))$ with a primitive function that is
  $0$ along $\phi_1^{H^{(q)}}(T^*_{q_0}(N)) \cap \mathcal{V} = F_q$.

  The quasi-isomorphisms $\varphi_q$ and its homotopy inverse can be
  constructed either by counting solutions of the Floer equation with
  moving boundary conditions, or alternatively, by applying the
  standard continuation map (comparing the $0$-Hamiltonian with
  $H^{(q)}$) followed by a naturality map as in~\eqref{eq:nat-G}. (The
  generalization in terms of $A_{\infty}$-modules corresponding to
  $T^*_{q_0}(N)$ and $\phi_1^{H^{(q)}}(T^*_{q_0}(N))$ can be
  established by similar methods.)  The bound on the weight of
  $\varphi_q$ follows from standard action-energy estimates in Floer
  theory.

  An important point about the previous weight is that it does not
  depend on $L$ and moreover that
  $\sup_{q \in \overline{B}} (2C_1(q) + C_2(q)) < \infty$.

  By choosing $J$ appropriately near the boundary of $\mathcal{V}$
  (and along $T^*(N) \setminus \mathcal{V}$) an argument based on the
  maximum principle (or alternatively, arguing as in the proof of
  Proposition~\ref{p:loc-glob}) shows that all the Floer trajectories
  contributing to any of the chain complexes
  $CF(L, T^*_{q_0}; (0,J); T^*(N))$ and
  $CF(L, \phi_1^{H^{(q)}}(T^*_{q_0}(N)); (0,J); T^*(N))$ must be
  entirely contained inside $\mathcal{V}$. (An analogous statement
  holds also for Floer polygons contributing to the higher order
  operations of the modules corresponding to $T^*_{q_0}(N)$ and
  $\phi_1^{H^{(q)}}(T^*_{q_0}(N))$ as long as we view them as modules
  over the Fukaya category of $\mathcal{V}$.)

  Since $\phi_1^{H^{(q)}}(T^*_{q_0}(N)) \cap \mathcal{V} = F_{q}$ and
  $T^*_{q_0}(N) \cap \mathcal{V} = F_{q_0}$, the statement we wanted to
  prove follows.
\end{proof}

We are now ready to prove the main theorem.
\begin{proof}[Proof of Theorem~\ref{t:main-a}]
  Fix a small $r_0>0$ and tubular neighborhood
  $\mathcal{V} = T^*_{\leq r_0}(N)$ of $N$. Recall
  from~\S\ref{sb:bcbundle-rlef} the symplectic embedding
  $\kappa: \mathcal{V} \longrightarrow E$ and its image
  $\mathcal{U}:= \kappa(\mathcal{V}) \subset E$.

  We now appeal to the cone decomposition~\eqref{eq:L-cone}
  from~\S\ref{s:cone-decomp} of Yoneda modules over $\fuk(E')$.  We
  apply this to the Lagrangian $K = N$ (i.e. the zero section) and its
  Yoneda module $\mathcal{N}$. Let $L \subset \mathcal{U}$ be any
  exact Lagrangian. \ocnote{The filtered cone decomposition of $\mathcal{N}$, as described 
  in Propositions \ref{prop:filtered-cone} and \ref{prop:quasi-shift}}, 
   gives a filtered cone decomposition of the chain complex
  $CF(L,N; E')$, which by the
  formulas~\eqref{eq:L-cone},~\eqref{B_j-cone},~\eqref{eq:Bjd}
  involves the following types of filtered chain complexes as well as their
  tensor products:
  \begin{enumerate}
  \item $CF(L, S_i;E')$, $CF(S_i, N;E')$, $i=1, \ldots, k$.
  \item $CF(S_{j''}, S_{j'}; E')$, $1\leq j' < j'' \leq
    k$. \label{i:Sj}
  \item $CF(L, N^{(k)}; E')$.
  \end{enumerate}
 The chain complexes in~(\ref{i:Sj}) do not depend on $L$.
  In particular their spectral invariants and boundary depths are
  independent of $L$.
  
  Formulas~\eqref{eq:L-cone}-\eqref{eq:Bjd} together with
  Proposition~\ref{p:sig-beta-tensor} and Lemma~\ref{l:rho-iterated}
  imply that there are constants $A_1, B_1, C_1 > 0$, that do not
  depend on $L$, such that:
  $$\rho(CF(L,N; E')) \leq A_1 
  \widetilde{\rho}(CF(L,S_1; E'), \ldots, CF(L,S_k; E')) +
  B_1\sum_{i=1}^k \beta(CF(L_,S_i); E') + C_1.$$

  Passing from $E'$ to $E$, as described in~\S\ref{s:Fl-E-E'}, we have
  action preserving chain isomorphisms $CF(L,N;E) \cong CF(L,N; E')$
  and $CF(L, \thmb_{x_i}; E) \cong CF(L,S_i; E')$ for every
  $1 \leq i \leq k$. Consequently, the spectral invariants and
  boundary depths of the chain complexes in $E$ coincide with the
  corresponding ones in $E'$.

  Next we appeal to Proposition~\ref{p:loc-glob} (with
  $W_0 = \mathcal{U}$, $V = E$ and $L_0=L$, $L_1 = N$) and to
  Proposition~\ref{p:loc-glob-t} and deduce that
  $$\rho(CF(L,N;\mathcal{U})) \leq A_1 
  \widetilde{\rho}(CF(L,F_{q_1};\mathcal{U}), \ldots, CF(L,F_{q_k};
  \mathcal{U})) + B_1\sum_{i=1}^k \beta(CF(L_,F_{q_i}); E') + C_1,$$
  where $q_i = \kappa^{-1}(x_i) \in N$.

  Put $q := q_1$. By Lemma~\ref{l:diff-fibers} we have that both
  $|\sigma(CF(L, F_q)) - \sigma(CF(L, F_{q_i}))|$ as well as
  $|\beta(CF(L, F_q)) - \beta(CF(L, F_{q_i}))|$ are uniformly bounded
  (with respect to $L$ and $i$), hence there exist constants
  $A_2, B_2> 0$, that do not depend on $L$, such that
  $$\rho(CF(L,N;\mathcal{U})) \leq A_2 + B_2\beta(CF(L, F_q)).$$
  Now, $\gamma(L, N) \leq \rho(CF(L,N;\mathcal{U}))$, hence
  $$\gamma{(L,N) }\leq A_2 + B_2\beta(CF(L, F_q))$$ for all exact
  Lagrangians $L \subset \mathcal{U}$. The last inequality together
  with the triangle inequality for $\gamma$ imply
  inequality~\eqref{eq:bound-main-thm} and conclude the proof of
  Theorem~\ref{t:main-a}.
\end{proof}

\subsection{Boundedness of the spectral metric implies boundedness of
  $\beta(CF(L,F_q)$} \label{sb:specm-betta-bound} Here we outline an
argument showing the statement at point~(\ref{i:vit-beta}) of
Remark~\ref{r:main-thm}. Namely, if the function
$$\mathcal{L}_{\text{ex}, N}(U) \ni L \longmapsto \gamma(N,L)$$ is
bounded, then
$$\mathcal{L}_{\text{ex}, N}(U) \ni L \longmapsto \beta(CF(L, F_q))$$
is bounded too. In other words the conjecture of Viterbo from
page~\pageref{p:vit-conj} implies the boundedness of the boundary
depths $CF(-, F_q)$ over the collection of exact Lagrangians
$L \subset U$ that are exact isotopic to the zero-section $N$.

Here is an outline of the proof. Let
$L \in \mathcal{L}_{\text{ex}, N}(U)$ and assume without loss of
generality that $L \pitchfork N$, $L \pitchfork F_q$. Fix an arbitrary
marking for $L$ and mark $N$ and $F_q$ by taking their primitive
functions to be identically $0$. Put
$$\alpha_+ = c([N]; N,L), \quad \alpha_{-} = c([N]; L,N).$$
We have $\alpha_+ + \alpha_{-} = \gamma(N,L)$. Note that $\alpha_+$
and $\alpha_{-}$ depend on the marking of $L$ but their sum
$\alpha_{+}+\alpha_{-}$ does not. Also note that $\beta(CF(N,L))$ is
independent of the marking of $L$.

We will now need to carry out a chain-level calculation with Floer
complexes. To this end we take the Floer complexes $CF(N,L)$,
$CF(L,N)$, $CF(N,F_q)$ and $CF(L,F_q)$ with Floer data having $0$
Hamiltonian terms. We also fix a Floer datum for $(L,L)$ whose
Hamiltonian term is induced from a $C^2$-small Morse function
$L \longrightarrow \mathbb{R}$ with a unique critical point of top
index, so that the unity $e_L \in HF(L,L)$ has a unique representing
cycle in $CF(L,L)$.

Fix $\epsilon>0$. Choose perturbation data for each of the tuples
$(N,L,N)$, $(L,N,L)$, $(N,L,F_q)$ and $(L,N, F_q)$ which are
compatible with the previous choices of Floer data and such that the
associated $\mu_2$-operations shift action by $\leq \epsilon$.

Let $a \in CF^{\leq \alpha_+}(N,L)$,
$b \in CF^{\leq \alpha_{-}}(L,N)$, be cycles representing the Floer
homology classes $\mathcal{N}^{N}_{L,N}([N])$ and
$\mathcal{N}_{N}^{N,L}([N])$ (see~\S\ref{sbsb:pss-nat}). Consider the
following two filtered chain maps:
\begin{equation} \label{eq:phi-CFLN}
  \begin{aligned}
    & \varphi: CF(L,F_q) \longrightarrow CF(N,F_q), \quad \varphi(x) :=
    \mu_2(a,x), \\
    & \phi: CF(N,F_q) \longrightarrow CF(L,F_q), \quad \phi(y) := \mu_2(b,y).
  \end{aligned}
\end{equation}
By our choices of data, $\varphi$ shifts action by $\leq \alpha_+$ and
$\phi$ by $\leq \alpha_{-}$. Note that $CF(N,F_q) = \mathbb{Z}_2 q$
and it is easy to see that $\varphi \circ \phi = \id$. We claim that
$\phi \circ \varphi$ is chain homotopic to the identity via a chain
homotopy $H$ that shifts action by $\leq \gamma(N,L) + \epsilon$,
where $\epsilon > 0$ can be taken to be arbitrarily small.

Before proving the last claim, let us see how it implies the main
statement we want to prove. For this purpose we would like to use
Lemma~\ref{lem:quasieq-inv} which compares the boundary depths of two
chain complexes that are chain homotopy equivalent (specifically in
our case, $CF(L,F_q)$ and $CF(N,F_q)$). However in order to employ
Lemma~\ref{lem:quasieq-inv} we need the shifts of each of $\varphi$
and $\phi$ to be non-negative and we also need to relate each of these
shifts to the shift of the chain homotopy $H$ which is claimed to be
$\gamma(N,L) + \epsilon$. The ``problem'' is that $\varphi$ and $\phi$
have shifts of $\leq \alpha_{+}$ and $\leq \alpha_{-}$ respectively
and we do not have information on the size of each of them alone - we
only know that $\alpha_{+} + \alpha_{-} = \gamma(N,L)$.

To go about this technical problem we proceed as follows. We shift the
marking of $L$ by a constant such that $\alpha_{-}=0$. Consequently
$\alpha_+$ will now become equal to $\gamma(N,L)$. We thus assume from
now on that $\alpha_{-}=0$ and $\alpha_{+} = \gamma(N,L)$. Under these
circumstances we can now apply Lemma~\ref{lem:quasieq-inv} and obtain
that
$|\beta(CF(L,F_q)) - \beta(CF(N,F_q))| \leq 2 \gamma(N,L) +
2\epsilon$. Since $\beta(CF(N,F_q))=0$ and by assumption $\gamma(N,-)$
is bounded, the main statement follows.

It remains to show the existence of the required chain homotopy $H$
between $\phi \circ \varphi$ and the $\id$. Consider the tuple of
Lagrangians $(L,N,L,F_q)$. Choose Floer perturbation data for this
tuple, which is compatible with the previous choices of Floer data,
and such that the following holds:
\begin{enumerate}
\item $\mu_2(\mu_2(b,a), x) = x$ for every $x \in CF(L, F_q)$. (Note
  that by our choices of Floer data, $\mu_2(b,a) \in CF(L,L)$ is the
  unique cycle representing the unity $e_L \in HF(L,L)$.)
\item The operation
  $\mu_3: CF(L,N) \otimes CF(N,L) \otimes CF(L, F_q) \longrightarrow
  CF(L,F_q)$ shifts action by $\leq \epsilon$.
\end{enumerate}

By standard $A_{\infty}$-identities (applied with
$\mathbb{Z}_2$-coefficients) we have for every $x \in CF(L,F_q)$:
\begin{equation}
  \begin{aligned}
    \phi \circ \varphi(x) = & \mu_2(b, \mu_2(a,x)) = \\
    & \mu_2(\mu_2(b,a), x) + \mu_3(b,a, \mu_1(x)) + \mu_1 \mu_3(b,a,x) = \\
    & x + \mu_3(b,a, \mu_1(x)) + \mu_1 \mu_3(b,a,x).
  \end{aligned}
\end{equation}
The required homotopy $H: CF(L,F_q) \longrightarrow CF(L,F_q)$ is then
$H(x) := \mu_3(b,a,x)$. And it clearly shifts action by
$\leq \gamma(N,L) + \epsilon$. \Qed

\ocnote{
\begin{rem}\label{rem:bounds-arg} A similar argument appears, for a different purpose, in
\cite{Kis-Sh}. At a conceptual level these arguments are a reflection of a Yoneda type lemma
 in the filtered setting that allows 
translation of relations among morphisms  of Yoneda modules  (over the  the $A_{\infty}$ Fukaya category)  in terms of $\mu^{k}$ operations. Such a result, called there the $\lambda$-lemma, appears in \cite{Bi-Co-Sh:lshadows-long}.
\end{rem}
}


\section{Filtered homological algebra} \label{s:fhomalg}

The purpose of this section is to establish a number of algebraic
results that allow control of the spectral range and boundary depth of
filtered complexes through cone-attachments.

\subsection{Background on filtered
  complexes} \label{subs:filtered-complexes}

We consider here filtered modules $C$ over a ring $R$. We assume the
filtration to be indexed by the reals and increasing, namely for every
$\alpha \in \mathbb{R}$ we have a submodule
$C^{\leq \alpha} \subset C$ and
$C^{\leq \alpha}\subset C^{\leq \alpha '}$ for $\alpha\leq \alpha'$.
For simplicity we will always assume that the filtration is
exhaustive, i.e.~$\cup_{\alpha \in \mathbb{R}} C^{\leq \alpha} = C$.

The shift of order $s \in \mathbb{R}$ of a filtered module $C$ is the
filtered module $C[s]$ defined by
$(C[s])^{\leq \alpha}=C^{\leq \alpha+s}$. (Despite the similarity in
notation, this has nothing to do with grading-shifts. In fact in this
paper we work in an ungraded setting.) An $R$-linear map
$f: C \longrightarrow C'$ between two filtered modules is called
$s$-{\em filtered} if $f(C^{\alpha}) \subset (C')^{\leq \alpha+s}$ for
all $\alpha \in \R$. We will refer to such a number $s$ as an
admissible shift for the map $f$. We will also say that $f$ shifts
action by $\leq s$, or sometimes that $f$ is filtered of shift
$s$. Notice that if $f$ is $s$-filtered then it is also $s'$-filtered
for all $s' \geq s$. An $R$-linear map $f:C \longrightarrow C'$ is
called {\em filtered} if it is $s$-filtered for some $s\geq 0$. For
reasons of convenience we will consider only shifts $s$ that are
non-negative. There is no loss of generality in doing that as any map
that shifts action by a negative amount can be viewed as
$0$-filtered. A slight drawback of this convention is that some of the
estimates on invariants of filtered chain complexes developed below
will be less sharp. Since our applications are concerned with coarse
estimates this will not play an important role in our considerations.

Let $C$ be a filtered chain complex or $R$-modules. This means that
$C$ is a filtered module and the differential $d$ of $C$ preserves the
filtration, i.e.~$d(C^{\leq \alpha}) \subset C^{\leq \alpha}$ for every
$\alpha$. To such a chain complex we can associate a persistence
module $H^{\leq \bullet}(C)$ consisting of the homologies of the
subcomplexes of $C$ prescribed by the filtration:
$$H^{\leq \alpha}(C)= H (C^{\leq \alpha}), \quad i^{\beta, \alpha}:
H^{\leq \alpha}(C) \longrightarrow H^{\leq\beta}(C), \ \alpha\leq
\beta, $$ where the maps $i^{\beta, \alpha}$ are induced by the
inclusions $C^{\leq \alpha} \subset C^{\leq \beta}$. We also have the
maps $i^{\alpha}: H^{\leq\alpha}(C)\to H(C)$ induced by the inclusions
$C^{\leq \alpha} \subset C$.

The boundary depth of the filtered complex $C$ is defined as:
$$\beta(C)=\inf \{b\in [0,\infty) \mid
\forall \alpha\in \R,\ \ker (i^{\alpha})=\ker (i^{\alpha+b,
  \alpha})\}.$$ For every $a \in H(C)$ we define the spectral
invariant $\sigma(a)$ by
$$\sigma(a) =
\inf\{\alpha \mid  a \in \textnormal{image\,} i^{\alpha}\}.$$
We also define
\begin{equation*}
  \begin{aligned}
    & \sigma_{+}(C)=\inf \{r\in \R \mid t \geq r \Rightarrow\
    \mathrm{Coker\/}(i^{t})=0\}, \\
    & \sigma_{-}(C)= \sup \{s\in \R \mid t\leq s \Rightarrow \ i^{t}=0
    \}, \\
    & \rho(C)= \sigma_{+} (C)- \sigma_{-}(C).
  \end{aligned}
\end{equation*}
As the notation suggests $\sigma_+(C)$ is the top (or supremal)
spectral invariant of $C$ and $\sigma_{-}(C)$ is the bottom (or
infemal) one. We call $\rho(C)$ the {\em spectral range} of $C$.

\begin{rem}\label{rem:bars}
  The notions above can easily be reformulated in terms of the modern
  terminology of {\em barcodes}~\cite{PRSZ:top-pers}.  For instance
  $\beta (C)$ is the length of the longest finite bar of $C$.
  Further, if the bar code associated to $H^{\leq \bullet}(C)$ is the
  collection $\{[i_{k},j_{k})\}$, then $\sigma_{+}(C)$ is the minimal
  $i_{k}$ among all bars with $j_{k}=\infty$ and $\sigma_{-}(C)$ is
  the maximal $i_{k}$ among the same (infinite) bars.
\end{rem}

We now describe the behavior of $\sigma$ and $\beta$ with respect to
some operations with filtered chain complexes. We begin with the
simple remark that if $f:C \longrightarrow C'$ is a quasi-isomorphism
and is $s$-filtered then we have:
\begin{equation}
  \label{eq:qi1}
  \sigma_{-}(C)\geq \sigma_{-}(C')-s, \quad 
  \sigma_{+}(C)\geq \sigma_{+}(C')-s~.~
\end{equation}
In particular, if $f$ admits an $s$-filtered homological inverse, we
deduce
$$|\sigma_{\pm}(C)-\sigma_{\pm}(C')|\leq s\ , \
|\rho(C)-\rho(C')|\leq 2s~.~$$

In order to relate the boundary depth of two quasi-isomorphic chain
complexes we will need the notion of boundary depth of a map. Let
$f:C \longrightarrow C'$ be a filtered chain map and let $s \geq 0$ be
an admissible shift for $f$. The map $f$ induces a map of persistence
modules
$f^{\bullet}_*:H^{\leq \bullet}(C) \longrightarrow H^{\leq
  \bullet}(C')[s]$, $f_*^{\bullet} = \{f_*^{\alpha}\}$ with
$f_*^{\alpha}: H^{\leq \alpha}(C) \longrightarrow H^{\leq
  \alpha+s}(C')$ induced by $f$.  We define the boundary depth of $f$,
viewed as an $s$-filtered map, by:
$$\beta_s(f)=\inf\{ b \in [0,\infty) \mid
\forall \alpha \in \R,\ \mathrm{Image \/}(f_*^{\alpha})\cap \ker
(i^{\alpha+s})\subset \ker (i^{\alpha+s+b, \alpha+s})\}~.~$$ Clearly
$\beta(C)=\beta_0(id_{C})$, $\beta_s(f)\leq \beta(C')$ and for
$s\leq s'$, $\beta_{s'}(f)=\max\{0,\beta_s(f)-s'+s\}$.

Assume now that $f:C \longrightarrow C'$, $g:C' \longrightarrow C$ are
$s$-filtered chain maps with $g_* \circ f_* =id$ in homology. We have
the inequality:
\begin{equation}\label{eq:htpyeq1}
  \beta(C) \leq \max\{\beta (C')+2s,
  \beta_{2s}(g\circ f - id_{C})\} ~.~
\end{equation}

The simplest way to control the boundary depth of maps as above is by
using filtered homotopies.  Let $f,f':C \longrightarrow C'$ be two
$s$-filtered maps that are homotopic with a homotopy $h:f\simeq f'$
which is $s'$-filtered, then:
\begin{equation}\label{eq:htpy-filt}
  \beta_s(f-f')\leq \min\{0, s'-s\} ~.~
\end{equation}
Assume now that $f:C \longrightarrow C'$, $g:C' \longrightarrow C$ are
$s$-filtered chain maps such that there is an $s$-filtered homotopy
$h:g\circ f\simeq id_{C}$.  In this case,
$\beta_{2s}(g\circ f-id_{C})=0$ and we deduce that
$$\beta(C)\leq \beta(C')+2s~.~$$
Summing up:

\begin{lem}\label{lem:quasieq-inv}
  If $f:C \longrightarrow C'$ and $g:C' \longrightarrow C$ are
  $s$-filtered and there are $s$-filtered chain homotopies
  $h:g\circ f\simeq id_{C}$ and $h':f \circ g \simeq id_{C'}$, then we
  have:
  \begin{equation}\label{eq:he}
    |\beta(C)-\beta(C')|\leq 2s \ , \
    |\sigma_{\pm}(C)-\sigma_{\pm}(C')|\leq s\ , \
    |\rho(C)-\rho(C')|\leq 2s~.~
  \end{equation}
\end{lem}

\subsection{Mapping cones} \label{sb:mcones}

Let $A$, $B$ be filtered chain complexes and $f: A \longrightarrow B$
an $s$-filtered chain map. The filtered mapping cone
$[A \xrightarrow{\; (f,s) \;} B]$ of $f$ is the mapping cone of $f$
endowed with the following filtration:
\begin{equation} \label{eq:cone-filtration} [A \xrightarrow{\; (f,s)
    \;} B]^{\leq \alpha}=A^{\leq \alpha-s}\oplus B^{\leq \alpha}.
\end{equation}
Of course, this choice of filtration is somewhat ad-hoc and there are
other possibilities. Firstly, once can shift the above filtration by
any real number. The reason for the specific choice
in~\eqref{eq:cone-filtration} is to make the inclusion
$B \longrightarrow [A \xrightarrow{\; (f,s) \;} B]$ action
preserving. Secondly, the filtration in~\eqref{eq:cone-filtration}
depends on $s$ (and therefore this parameter appears in the notation).

We will now estimate the boundary depth and spectral range of the
mapping cone in terms of the invariants of its factors.

\begin{lem} \label{lem:spectral-cones} Let
  $C := [A \xrightarrow{\; (f,s) \;} B]$. We have the following
  inequalities:
  \begin{equation}\label{eq:delta-cone}
    \sigma_{-}(C)\geq\min\{\sigma_{-}(B)-\beta(A),\sigma_{-}(A) +s\}, 
  \end{equation}
  
  \begin{equation} \label{eq:delta-cone2} \sigma_{+}(C)\leq
    \max\{\sigma_{+}(B), \sigma_{+}(A)+\beta(B)+s\}
  \end{equation}
  and
  \begin{equation}\label{eq:bdryDepthCone}
    \beta(C)\leq \beta(A)+\beta(B) + 
    \max\{0, \sigma_{+}(A)-\sigma_{-}(B) +s\}~.~
  \end{equation}
\end{lem}

Note that the estimates in the lemma do not depend on the chain map
$f$ (though they do depend on the amount of shift $s$ of $f$).

\begin{proof}[Proof of Lemma~\ref{lem:spectral-cones}]
  The basic ingredient in the proof is provided by the long exact
  sequences:
  $$\cdots \longrightarrow H^{\leq \alpha-s}(A)
  \stackrel{f}{\longrightarrow} H^{\leq \alpha}(B)
  \stackrel{h}{\longrightarrow}H^{\leq
    \alpha}(C)\stackrel{p}{\longrightarrow}H^{\leq\alpha-s}(A)\longrightarrow
  \cdots,$$ where $h$ is induced by inclusion and $p$ by the
  projection. The maps $i^{\alpha}$, $i^{\beta, \alpha}$ relate
  functorially these exact sequences.

  To see~\eqref{eq:delta-cone} let $c\in H^{\leq \alpha}(C)$,
  $\alpha< \min\{\sigma_{-}(B)-\beta(A),\sigma_{-}(A)+s\}$.  Then
  $p(c)\in H^{\leq \alpha-s}(A)$ and as $\alpha-s <\sigma_{-}(A)$,
  then $i^{\alpha -s}(p(c))=0$. This implies that
  $i^{\alpha + b-s, \alpha-s}(p(c))=0$ for all $b>\beta(A)$. We take
  $b$ sufficiently small such that $\alpha < \sigma_{-}(B)-b$.  Thus,
  there is $c'\in H^{\leq \alpha +b}(B)$ such that
  $h(c')=i^{\alpha +b, \alpha}(c)$.  But we also have that
  $\alpha + b< \sigma_{-}(B)$ so that $i^{\alpha + b}(c')=0$.
  Therefore $i^{\alpha}(c)=0$ which shows the the first inequality.

  The proof of~\eqref{eq:delta-cone2} is similar.  Indeed, if
  $\alpha>\max\{\sigma_{+}(B), \sigma_{+}(A)+\beta(B)+s\}$ and
  $c\in H(C)$, then fix $b>\beta(B)$ very close to $\beta(B)$ such
  that $\alpha> \sigma_{+}(A)+b+s$. There exists
  $c'\in H^{\leq \alpha - s-b}(A)$ such that
  $i^{\alpha - s-b}(c')=p(c)$ and moreover
  $f(i^{\alpha-s, \alpha - s-b}(c'))=0$. Let
  $c''=i^{\alpha-s, \alpha - s-b}(c')$. There is
  $c'''\in H^{\leq \alpha}(C)$ such that $p(c''')=c''$. Now
  $p(i^{\alpha}(c''')-c)=0$ therefore there is $\tilde{c}\in H(B)$
  such that $h(\tilde{c})=i^{\alpha}(c''')-c$. But
  $\alpha > \sigma_{+}(B)$ hence there is
  $\tilde{c}'\in H^{\leq \alpha}(B)$ such that
  $i^{\alpha}(\tilde{c}')=\tilde{c}$. It follows that
  $i^{\alpha}(c'''-h(\tilde{c}'))=c$ and thus
  $i^{\alpha}: H^{\leq \alpha}(C) \longrightarrow H(C)$ is surjective.

  Finally, to show~\eqref{eq:bdryDepthCone}, assume
  $r>\beta(A)+\beta(B)+\max\{0, \sigma_{+}(A)-\sigma_{-}(B)\}$ and let
  $c\in H^{\leq \alpha}(C)$ such that $i^{\alpha}(c)=0$. We want to
  show that $i^{\alpha +r, \alpha}(c)=0$. Note that
  $i^{\alpha +b-s, \alpha-s}(p(c))=0$ for $b>\beta(A)$.  Let
  $c'=i^{\alpha +b, \alpha}(c)$. Therefore, there is
  $c''\in H^{\leq \alpha +b}(B)$ with $h(c'')=c'$. In case
  $\alpha+b< \sigma_{-}(B)$, then $i^{\alpha+b}(c'')=0$ and thus for
  $b'>\beta(B)$ we have $i^{\alpha+b+b', \alpha+b}(c'')=0$. This
  implies that $i^{\alpha +b+b', \alpha}(c)=0$ and, by taking $b,b'$
  small enough, this shows $i^{\alpha+r, \alpha}(c)=0$. The other
  possibility to consider is when $\alpha +b\geq \sigma_{-}(B)$. In
  this case let $\hat{c}=i^{\alpha+b}(c'')$. As $h(\hat{c})=0$ there
  is $\hat{c}'\in H(A)$ such that $f(\hat{c}')=\hat{c}$.  Now consider
  $k > \max\{ 0,\sigma_{+}(A)+s-\sigma_{-}(B)\}$.  There exists
  $\hat{c}''\in H^{\leq \alpha+b-s +k}(A)$ such that
  $i^{\alpha+b-s +k}(\hat{c}'')=\hat{c'}$. Now
  $f(\hat{c}'')\in H^{\leq \alpha+b+k}(B)$ and
  $i^{\alpha+b+k}(i^{\alpha+b+k, \alpha+b}(c'')-f(\hat{c}''))=0$.
  Thus
  $i^{\alpha+b+b'+k,\alpha+b+k}(i^{\alpha+b+k,\alpha+b}(c'')-f(\hat{c}''))=0$
  which combined with $h(f(\hat{c}''))=0$ implies that
  $i^{\alpha+b+b'+k, \alpha}(c)=0$ which shows our claim by taking
  $b,b',k$ small enough.
\end{proof}

From inequalities~\eqref{eq:delta-cone},~\eqref{eq:delta-cone2}) we
deduce a simpler (but rougher) estimate for the spectral range of
$C = [A \xrightarrow{\; (f,s) \;} B]$:
\begin{equation}\label{eq:est-sprCone}
  \rho(C)\leq\max\{\sigma_{+}(A),\sigma_{+}(B)\} -
  \min\{\sigma_{-}(A),\sigma_{-}(B)\}+\beta(A)+\beta(B)+s
\end{equation}

It is important to note that one can not, in general, eliminate the
boundary depth from estimates such
as~\eqref{eq:delta-cone},~\eqref{eq:delta-cone2})
or~\eqref{eq:est-sprCone}) nor can one eliminate the spectral values
$\sigma_{+},\sigma_{-}$ from an estimate
like~\eqref{eq:bdryDepthCone}.

\begin{rems} \label{r:s-shift}
  \begin{enumerate}
  \item Above we have considered $s$-morphisms with $s \geq
    0$. Occasionally it makes sense to consider also the case $s<0$
    (such maps not only preserve filtrations but in fact shift them
    downwards by $(-s)$). The
    estimates~\eqref{eq:qi1}~-~\eqref{eq:est-sprCone} can be easily
    adjusted to the case $s<0$. However, for the applications needed in
    this paper it is enough to assume $s \geq 0$.
  \item \label{i:C'-C} Let $A$, $B$ be two filtered chain complexes
    and $f: A \longrightarrow B$ an $s$-filtered chain map, where we
    allow here any $s \in \mathbb{R}$ (also $s<0$). Let $s'\geq
    s$. Then $f$ is also an $s'$-filtered chain map. We can now endow
    the mapping cone of $f$ with two different filtrations,
    following~\eqref{eq:cone-filtration}, once using the shift $s$ and
    once the shift $s'$. Denote the corresponding filtered mapping
    cones by $C := [A \xrightarrow{\; (f,s) \;} B]$ and
    $C' := [A \xrightarrow{\; (f,s') \;} B]$. It easily follows
    (e.g.~from Lemma~\ref{lem:quasieq-inv}) that
    \begin{equation} \label{eq:cone-C-C'} |\sigma_{\pm}(C') -
      \sigma_{\pm}(C)| \leq s'-s, \quad |\rho(C') - \rho(C)|, \;
      |\beta(C') - \beta(C)| \leq 2(s'-s).
    \end{equation}
  \end{enumerate}
\end{rems}

Next we analyze equivalences of mapping cones, taking into account
filtrations. Consider the following diagram:
\begin{equation} \label{eq:d-square}
  \begin{gathered}
    \xymatrix{A' \ar[r]^{(f', s_{f'})} \ar[d]_{(\psi', s_{\psi'})} &
      B' \ar[d]^{(\phi', s_{\phi'})} \\
      A'' \ar[r]_{(f'',s_{f''})} & B''}
  \end{gathered}
\end{equation}
where $A'$, $B'$, $A''$, $B''$ be filtered chain complexes and the
notation on the edges of the square are pairs consisting of a filtered
chain map and an admissible shift. (E.g.~$(f',s_{f'})$ means that
$f': A' \longrightarrow B'$ is an $s_{f'}$-filtered chain map etc.)

We assume that~\eqref{eq:d-square} commutes up to an $s_{h'}$-filtered
chain homotopy $h': A' \longrightarrow B'$
(i.e.~$\phi' \circ f' - f'' \circ \psi' = dh + hd$) for some
$s_{h'} \geq s_{f'}, s_{\phi'}, s_{\psi'}, s_{f''}$. Further, assume
that $\psi'$ and $\phi'$ have filtered homotopy inverses, i.e.~there
exist an $s_{\psi''}$-filtered chain map
$\psi'': A'' \longrightarrow A'$ and an $s_{\phi''}$-filtered chain
map $\phi'': B'' \longrightarrow B'$ with
\begin{equation} \label{eq:psi''-phi''}
  \begin{aligned}
    & \psi'' \circ \psi' = dk'+ k'd, \quad & \psi' \circ \psi'' = dk''
    + k''d, \\
    & \phi'' \circ \phi' = dr'+ r'd, \quad & \phi' \circ \phi'' = dr''
    + r''d, \\
  \end{aligned}
\end{equation}
where $k': A' \longrightarrow A'$, $k'': A'' \longrightarrow A''$,
$r': B' \longrightarrow B'$, $r'': B'' \longrightarrow B''$ are
filtered linear maps. We denote by $s_{k'}$, $s_{k''}$, $s_{r'}$,
$s_{r''}$ admissible shifts for these maps.

Denote by $C(f',s_{f'}) := [A' \xrightarrow{\; (f',s_{f'}) \;} B']$ and
by $C(f'',s_{f''}) := [A'' \xrightarrow{\; (f'',s_{f''}) \;} B'']$ the
filtered mapping cones of $(f', s_{f'})$ and $(f'', s_{f''})$
respectively.

\begin{prop} \label{p:equiv-cones} There exist filtered chain maps
  $\varphi': C(f',s_{f'}) \longrightarrow C(f'',s_{f''})$ and
  $\varphi'': C(f'',s_{f''}) \longrightarrow C(f',s_{f'})$ that fit
  into the following diagrams:
  \begin{equation} \label{eq:d-cones}
  \begin{gathered}
    \xymatrix{A' \ar[r]^{f'} \ar[d]_{\psi'} &
      B' \ar[d]^{\phi'} \ar[r] & C(f', s_{f'}) \ar[d]^{\varphi'}
      \ar[r] & A' \ar[d]_{\psi'} \\
      A'' \ar[r]_{f''} & B \ar[r] & C(f'', s_{f''}) \ar[r] & A''}
  \end{gathered}
  \quad \quad 
  \begin{gathered}
    \xymatrix{A' \ar[r]^{f'} &
      B' \ar[r] & C(f', s_{f'})  \ar[r] & A' \\
      A'' \ar[u]^{\psi''} \ar[r]_{f''} & B \ar[u]_{\phi''} \ar[r] &
      C(f'', s_{f''}) \ar[u]_{\varphi''} \ar[r] & A'' \ar[u]^{\psi''}}
  \end{gathered}
\end{equation}
where the unmarked horizontal maps in both diagrams are the canonical
chain maps associated to cones. These maps are filtered. The left-hand
square in the 2'nd diagram commutes up to a filtered chain homotopy
$h''$.  The 2'nd and 3'rd squares, in each diagram, commute. The
compositions $\varphi'' \circ \varphi'$ and $\varphi' \circ \varphi''$
are chain homotopic to the identities via filtered chain homotopies
$H'$ and $H''$. Moreover, there exist admissible shifts
$s_{\varphi'}$, $s_{\varphi''}$, $s_{H'}$, $s_{H''}, s_{h''}$ for
$\varphi'$, $\varphi''$, $H'$, $H''$, $h''$, and a universal constant
$C$ (that depends neither on the initial diagram nor on any of the
other maps mentioned above) such that
\begin{equation} \label{eq:shifts-diag} s_{\varphi'}, s_{\varphi''},
  s_{H'}, s_{H''}, s_{h''} \leq C (s_{f'} + s_{f''} + s_{\phi'} +
  s_{\phi''} + s_{\psi'} + s_{\psi''} + s_{h'} + s_{k'} + s_{k''} +
  s_{r'} + s_{r''}).
  \end{equation}
\end{prop}
\begin{proof}
  The existence of $\varphi'$, $\varphi''$, $H'$, $H''$ is standard
  homological algebra. In fact, it is straightforward to write down
  explicit formulae for these maps. For example, $\varphi'$ can be
  taken to be $\varphi'(a', b') = (\psi(a'), \phi'(b') + h'(a'))$. One
  then uses the chain homotopies $k'$, $k''$, $r'$, $r''$ to describe
  explicitly $h''$, $\varphi''$ and $H'$, $H''$.

  The only possibly non-standard ingredients are the statements
  concerning the actions shifts and inequality~\eqref{eq:shifts-diag}.
  These can be easily derived from the formulae for $\varphi'$,
  $\varphi''$, $h''$, $H'$, $H''$.
\end{proof}

\begin{rem} \label{r:shift-cones} By deriving explicit formulae for
  $\varphi'$, $\varphi''$, $H'$, $H''$, $h''$ it is possible obtain
  sharper estimates for each of
  $s_{\varphi'}, s_{\varphi''}, s_{H'}, s_{H''}, s_{h''}$ than the
  uniform bound~\eqref{eq:shifts-diag}. In the following we will be
  interested only in coarse estimates on these shifts, hence we will
  not need such sharp estimates.
\end{rem}

\subsection{Iterated cones} \label{sb:icones}
Let $E$, $F$, $G$ be filtered chain complexes and
$f: F \longrightarrow G$ an $s_f$-filtered chain map. Let
$g: E \longrightarrow [F \xrightarrow{\; (f, s_f) \;} G]$ be an
$s_g$-filtered chain map and define
$$C = [E \xrightarrow{\; (g, s_g) \;}
[ F \xrightarrow{\; (f, s_f) \;} G]].$$ There exist a module
homomorphism $g': E \longrightarrow F$ that shifts action by
$\leq s_{g'} := \max\{0, s_g - s_f\}$, and another module homomorphism
$f': [E \xrightarrow{\; (g', s_{g'}) \;} F] \longrightarrow G$ that
shifts action by $\leq s_{f'} := s_f$, such that the module
$$C' = [[ E \xrightarrow{\; (g', s_{g'}) \;} F]
\xrightarrow{\; (f', s_{f'}) \;} G]$$ is isomorphic to $C$ by the map
$C \longrightarrow C'$ induced from the underlying identity
map. Moreover, if $s_g < s_f$ (i.e.~$s_{g'} = 0$) then this map shifts
action by $\leq (s_f-s_g)$ and if $s_g\geq s_f$ (i.e.~$s_{g'}\geq 0$)
it shifts action by $\leq 0$.

\begin{rem}
  The asymmetry in the action shifts comes from our convention to
  consider only non-negative action shifts, i.e.~to regard a map that
  shifts action by a negative amount as shifting action by $\leq
  0$. If we would have allowed for negative action-shifts then we
  could take $s_{g'} = s_g - s_f$ and the identity map
  $C \longrightarrow C'$ would become action preserving. But as
  remarked at the beginning of~\S\ref{subs:filtered-complexes} we will
  stick to the convention that shifts in action are always
  non-negative.
\end{rem}
  
It follows from the above that
\begin{equation} \label{eq:2-cones} |\sigma_{\pm}(C') -
  \sigma_{\pm}(C)| \leq |s_f - s_g|, \quad |\beta(C') - \beta(C)| \leq
  2|s_f-s_g|.
\end{equation}

In the following we will be interested in coarse bounds on spectral
invariants and boundary depths of {\em iterated} cones. Therefore, by
abuse of notation we will often write them as
$K = [ A_r \longrightarrow A_{r-1} \longrightarrow \cdots
\longrightarrow A_1 \longrightarrow A_0]$, whenever the maps are clear
from the context and their action shifts are fixed up to a bounded
change. The spectral invariants and boundary depths of $K$ will then
be determined up to a bounded error.

\subsection{Estimating the spectral range of iterated
  cones} \label{sb:range-icone}

Let $A_{(0)}, \ldots, A_{(k)}$ be a finite collection of filtered
chain complexes of $R$-modules. Assume that each of the $A_{(i)}$'s
has finite spectral range. Define the following values:
\begin{equation} \label{eq:sig-rho-tilde}
  \begin{aligned}
    \widetilde{\sigma}_+(A_{(k)}, \ldots, A_{(0)}) & := \max \{
    \sigma_+(A_{(k)}), \ldots, \sigma_+(A_{(0)}) \}, \\
    \widetilde{\sigma}_-(A_{(k)}, \ldots, A_{(0)}) & := \min \{
    \sigma_-(A_{(k)}), \ldots, \sigma_-(A_{(0)}) \}, \\
    \widetilde{\rho}(A_{(k)}, \ldots, A_{(0)}) & :=
    \widetilde{\sigma}_+(A_{(k)}, \ldots, A_{(0)}) -
    \widetilde{\sigma}_-(A_{(k)}, \ldots, A_{(0)}).
  \end{aligned}
\end{equation}

From inequalities~\eqref{eq:delta-cone}~-~\eqref{eq:est-sprCone}, and
using the notation~\eqref{eq:sig-rho-tilde}, we obtain the following
inequalities for the mapping cone
$C = [A \xrightarrow{\; (f,s) \;} B]$ of an $s$-filtered chain map
$f: A \longrightarrow B$:
\begin{equation} \label{eq:sig-pm-beta}
  \begin{aligned}
    \sigma_+(C) & \leq \phantom{-} \widetilde{\sigma}_+(A,B) +
    \beta(B) + s, \\
    -\sigma_{-}(C) & \leq -\widetilde{\sigma}_{-}(A,B) + \beta(A), \\
    \beta(C) & \leq \beta(A) + \beta(B) + \widetilde{\sigma}_+(A,B) -
    \widetilde{\sigma}_{-}(A,B) + s.
  \end{aligned}
\end{equation}

It follows that both $\rho(C)$ as well as $\beta(C)$ can be bounded
from above by the same expression:
\begin{equation} \label{eq:sig-rho-cone} \rho(C), \beta(C) \leq
  \widetilde{\rho}(A,B) + \beta(A) + \beta(B) + s.
\end{equation}

Turning to the case of iterated cones, let $A_0, \ldots, A_r$ be
filtered chain complexes. Put $C_0 := A_0$. Let
$\varphi_1: A_1 \longrightarrow C_0$ be an $s_1$-filtered chain map
for some $s_1 \geq 0$. Define
$C_1:= [A_1 \xrightarrow{\; (\varphi_1,s_1) \;} C_0]$, filtered as
described in~\eqref{eq:cone-filtration}. Continuing inductively,
assume that we have constructed already the filtered chain complex
$C_i$ for some $1 \leq i \leq r-1$ and let
$\varphi: A_{i+1} \longrightarrow C_i$ be an $s_{i+1}$-filtered chain
map for some $s_{i+1} \geq 0$. Define
$C_{i+1} = [A_{i+1} \xrightarrow{\; (\varphi_{i+1}, s_{i+1}) \;}
C_i]$. We call the final chain complex $C_r$ an iterated cone with
attachments $A_0, \ldots, A_r$ and sometime denote it by
$$C_r = [A_r \longrightarrow [A_{r-1} \longrightarrow \cdots
\longrightarrow [A_2 \longrightarrow [A_1 \longrightarrow A_0]] \cdots
]],$$ omitting references to the chain maps $\varphi_i$ and the
action-shifts $s_i$.

The following Lemma follows easily from~\eqref{eq:sig-pm-beta}.
\begin{lem} \label{l:rho-iterated} There exists (universal) constants
  $a_r, b_r, e_r > 0$, depending only on $r$, such that for every
  iterated cone $C_r$ as above we have:
  \begin{equation} \label{eq:rho-iterated} \rho(C_r) \leq a_r
    \widetilde{\rho}(A_r, \ldots, A_0) + b_r \sum_{j=0}^r \beta(A_j) +
    e_r \sum_{j=1}^r s_j.
  \end{equation}
\end{lem}

\subsection{Weakly filtered $A_{\infty}$-categories and
  modules} \label{sb:wf-ai-mod} Recall that a weakly filtered
$A_{\infty}$-category $\mathscr{C}$ is an $A_{\infty}$-category such
that for every two objects $X,Y \in \textnormal{Ob}(\mathscr{C})$ the
chain complex $(\hom_{\mathscr{C}}(X,Y), \mu_1^{\mathscr{C}})$ is
filtered and additionally each of the higher order operations
$\mu_d^{\mathscr{C}}$, $d \geq 2$, preserves filtrations up to a
(uniform) bounded error. Similarly, filtered modules $\mathcal{M}$
over such categories are $\mathscr{C}$-modules such that for every
object $X \in \textnormal{Ob}(\mathscr{C})$ the chain complex
$(\mathcal{M}(X), \mu_1^{\mathcal{M}})$ is filtered, and the higher
order operations $\mu_d^{\mathcal{M}}$, $d \geq 2$, preserve
filtrations up to (uniform) bounded errors (one for each $d$). One can
define weakly filtered pre-module (resp. module) homomorphisms
$f: \mathcal{M} \longrightarrow \mathcal{N}$ between weakly filtered
modules, by analogy to filtered maps (resp. chain maps). The 1'st
order component $f_1: \mathcal{M}(X) \longrightarrow \mathcal{N}(X)$,
$X \in \textnormal{Ob}(\mathscr{C})$, of such a map is a filtered
linear map (resp. chain map) that shifts filtrations by $\leq s_f$,
where $s_f$ is a constant that does not depend on $X$. An analogous
condition is imposed on the higher order $f_d$ components of $f$.
(Sometimes, by abuse of notation we will omit the the subscript in
$f_1$ and denote the 1'st order component also by $f$.)  Finally,
there is also the notion of weakly filtered $A_{\infty}$-functors
between weakly filtered $A_{\infty}$-categories (in contrast to module
homomorphisms which are allowed to shift filtrations, such functors
are assumed to preserve filtrations, up to bounded errors). We refer
the reader to~\cite{Bi-Co-Sh:lshadows-long} for the basic theory and
formalism of weakly filtered $A_{\infty}$-categories.

\begin{rem} \label{r:wf} A word of caution about terminology
  differences is in order. The notion ``weakly filtered'' appears in
  the literature with two different meanings. In the formalism
  of~\cite{FO3:book-vol1, FO3:book-vol2} ``weakly filtered map''
  stands for a map between filtered chain complexes (or
  $A_{\infty}$-algebras) that preserves filtrations up to a shift,
  whereas in our terminology such maps are called ``filtered'' or
  $s$-filtered if we specify the amount of shift $s$. Our notion of
  ``weakly filtered'' means something else. For example, in the case
  of weakly filtered categories, the 1'st order operations (i.e.~the
  differentials of the $\hom$'s) preserve filtrations, but the higher
  order operations preserve filtrations only up to uniform errors
  (which we call in~\cite{Bi-Co-Sh:lshadows-long} discrepancies), and
  the wording ``weakly'' refers to that. Thus, without these
  discrepancies we would have called such categories ``filtered
  categories''. In a similar vein we have weakly filtered functors,
  modules and (pre)-module homomorphisms.
\end{rem}

The contents of the entire section above
(\S\ref{subs:filtered-complexes}~-~\S\ref{sb:range-icone}) applies
with minor modifications also to the framework of weakly filtered
$A_{\infty}$-modules over a weakly filtered $A_{\infty}$-category
$\mathscr{C}$ rather than just chain complexes. For example, if one
replaces the filtered chain complexes $A$, $B$ by weakly filtered
$\mathscr{C}$-modules $\mathcal{A}$, $\mathcal{B}$ and
$f: A \longrightarrow B$ by a module homomorphism, then one can define
an $A_{\infty}$-mapping cone module
$\mathcal{C} = [\mathcal{A} \xrightarrow{\; f \;} \mathcal{B}]$ which
is weakly filtered in a similar way as
in~\eqref{eq:cone-filtration}. (See~\cite[\S2.4]{Bi-Co-Sh:lshadows-long}
for more details.) The inequalities from~\eqref{eq:cone-C-C'} then
continue to hold with $C'$ and $C$ replaced by $\mathcal{C}'(X)$ and
$\mathcal{C}(X)$ respectively, for every object $X$ in the underlying
$A_{\infty}$-category $\mathscr{C}$. Similar modifications apply
to~\eqref{eq:2-cones} as well as to~\eqref{eq:he}.

It is important to note that in the case of $A_{\infty}$-modules the
preceding inequalities hold uniformly for all objects $X$, since the
shift parameters ($s_f$, $s_g$ etc.) depend only on the modules and
the homomorphisms between them, and not on the choice of a particular
object in the $A_{\infty}$-category.

\begin{rem} \label{r:ch-vs-ai} Through this paper we appeal several
  times to the notions of weakly filtered $A_{\infty}$-categories,
  functors and modules. However, from a purely formal viewpoint this
  is not really necessary. Indeed, we will never use any of the higher
  operations associated to $A_{\infty}$-structures or to special
  features that distinguish such structures from filtered chain
  complexes. Thus in principle one can ``downgrade'' the entire
  algebraic formalism in this paper to filtered chain complexes and
  their persistent homology. The reason we opted for using a bit of
  $A_{\infty}$ formalism is the following. A considerable part of the
  algebra in this paper is devoted to establishing bounds on
  invariants of filtered Floer chain complexes, e.g.~of the type
  $CF(-,-)$, which are uniform in the ``variables'' $(-,-)$, or at
  least one of them. These variables are Lagrangian submanifolds,
  hence are objects of a Fukaya category (which is weakly
  filtered). As explained at several points above, the uniformity of
  various quantities related to action filtration can be more
  concisely expressed using the language of $A_{\infty}$-modules.
\end{rem}

We end this section with a useful definition.
\begin{dfn} \label{d:dit-mod} Let $\mathcal{M}$, $\mathcal{N}$ be two
  weakly filtered $A_{\infty}$-modules. Let
  $f: \mathcal{M} \longrightarrow \mathcal{N}$ be a weakly filtered
  module homomorphism and $w \geq 0$. We say that $f$ is a
  quasi-isomorphism of weight $\leq w$ if the following holds:
  \begin{enumerate}
  \item $f$ shifts filtration by $\leq w$.
  \item There exists a weakly filtered module homomorphism
    $g: \mathcal{N} \longrightarrow \mathcal{M}$ that shifts
    filtration by $\leq w$ and two weakly filtered pre-module
    homomorphisms $h: \mathcal{M} \longrightarrow \mathcal{M}$,
    $k: \mathcal{N} \longrightarrow \mathcal{N}$ that shift
    filtrations by $\leq w$, such that:
    \begin{equation} \label{eq:fg-w} g \circ f = \id +
      \mu_1^{\text{mod}}(h), \quad f \circ g = \id +
      \mu_1^{\text{mod}}(k).
    \end{equation}
  \end{enumerate}
  We say that two weakly filtered modules $\mathcal{M}$ and
  $\mathcal{N}$ are at distance $w$ one from the other if there exists
  a quasi-isomorphism $f: \mathcal{M} \longrightarrow \mathcal{N}$ of
  weight $\leq w$.
\end{dfn}

\ocnote{
\begin{rem}\label{rem:interleaving-dist} Similar notions appear in relation to the so-called 
bottleneck distance in persistance module theory, for instance in \cite{Usher-Zhang:perh}, 
as well as in a somewhat different context in \cite{Bi-Co-Sh:lshadows-long}.
\end{rem}
}

The same definition can be easily adapted to the case when
$\mathcal{M}$ and $\mathcal{N}$ are just filtered chain complexes and
$f: \mathcal{M} \longrightarrow \mathcal{N}$ is a $w$-filtered chain
map. In this case, the analogue of condition~\eqref{eq:fg-w} simply
means that $f\circ g$ and $g \circ f$ are chain homotopic to the
respective identities via $w$-filtered chain homotopies. Note that
despite being called only a ``quasi-isomorphism'', $f$ satisfies a
stronger condition - it is implicitly assumed to have a homotopy
inverse.

\subsection{Spectral range and boundary depth of tensor products}
\label{sb:tensor-prod}

Let $A$, $B$ be finite dimensional filtered chain complexes over a
field $R$. The tensor product (over $R$) chain complex $A \otimes B$
inherits a filtration from $A$ and $B$, where
$(A \otimes B)^{\leq \alpha} \subset A \otimes B$ is generated by the
collection of subspaces $A^{\leq \alpha - s} \otimes B^{\leq s}$,
$s \in \mathbb{R}$.

\begin{prop} \label{p:sig-beta-tensor} For the tensor product chain
  complex $A \otimes B$ we have:
  \begin{equation*}
    \begin{aligned}
      & \sigma_{\pm}(A \otimes B) = \sigma_{\pm}(A) + \sigma_{\pm}(B),
      \quad \rho(A \otimes B) = \rho(A) + \rho(B), \\
      & \beta(A \otimes B) \leq \max \{ \beta(A), \beta(B) \}.
    \end{aligned}
  \end{equation*}
\end{prop}

\begin{proof}
  This follows by direct calculation of the barcode of the persistence
  module $H_*(A \otimes B)$, using the K\"{u}nneth formula for
  persistence modules from~\cite{Po-Sh-St:pers-op}.
\end{proof}


\bibliography{bibliography}

\def\cprime{$'$} \def\cprime{$'$}
\begin{thebibliography}{FOOO2}

\bibitem[Alb]{Alb:PSS}
P.~Albers.
\newblock A {L}agrangian {P}iunikhin-{S}alamon-{S}chwarz morphism and two
  comparison homomorphisms in {F}loer homology.
\newblock {\em Int. Math. Res. Not. IMRN}, (4):Art. ID rnm134, 56, 2008.

\bibitem[BC1]{Bi-Ci:Stein}
P.~Biran and K.~Cieliebak.
\newblock Lagrangian embeddings into subcritical {S}tein manifolds.
\newblock {\em Israel J. Math.}, 127:221--244, 2002.

\bibitem[BC2]{Bi-Co:lcob-fuk-arxiv}
P.~Biran and O.~Cornea.
\newblock Lagrangian cobordism and {F}ukaya categories.
\newblock arXiv version (2018). Can be found at
  \url{http://arxiv.org/pdf/1304.6032}.

\bibitem[BC3]{Bi-Co:lcob-fuk}
P.~Biran and O.~Cornea.
\newblock Lagrangian cobordism and {F}ukaya categories.
\newblock {\em Geom. Funct. Anal.}, 24(6):1731--1830, 2014.

\bibitem[BC4]{Bi-Co:lefcob-pub}
P.~Biran and O.~Cornea.
\newblock Cone-decompositions of {L}agrangian cobordisms in {L}efschetz
  fibrations.
\newblock {\em Selecta Math. (N.S.)}, 23(4):2635--2704, 2017.

\bibitem[BCS]{Bi-Co-Sh:lshadows-long}
P.~Biran, O.~Cornea, and E.~Shelukhin.
\newblock {L}agrangian shadows and triangulated categories.
\newblock Preprint (2018). Can be found at
  \url{http://arxiv.org/pdf/1806.06630v1}.

\bibitem[CE]{El-Ci:Weinstein-book}
K.~Cieliebak and Y.~Eliashberg.
\newblock {\em From {S}tein to {W}einstein and back}, volume~59 of {\em
  American Mathematical Society Colloquium Publications}.
\newblock American Mathematical Society, Providence, RI, 2012.
\newblock Symplectic geometry of affine complex manifolds.

\bibitem[DKM]{DKM:spectral}
J.~Djureti\'{c}, J.~Kati\'{c}, and D.~Milinkovi\'{c}.
\newblock Comparison of spectral invariants in {L}agrangian and {H}amiltonian
  {F}loer theory.
\newblock {\em Filomat}, 30(5):1161--1174, 2016.

\bibitem[EG]{El-Gr:convex}
Y.~Eliashberg and M.~Gromov.
\newblock Convex symplectic manifolds.
\newblock In {\em Several complex variables and complex geometry Part 2 (Santa
  Cruz, CA, 1989)}, volume~52 of {\em Proc. Sympos. Pure Math.}, pages
  135--162, Providence, RI, 1991. Amer. Math. Soc.

\bibitem[FOOO1]{FO3:book-vol1}
K.~Fukaya, Y.-G. Oh, H.~Ohta, and K.~Ono.
\newblock {\em Lagrangian intersection {F}loer theory: anomaly and obstruction.
  {P}art {I}}, volume~46 of {\em AMS/IP Studies in Advanced Mathematics}.
\newblock American Mathematical Society, Providence, RI, 2009.

\bibitem[FOOO2]{FO3:book-vol2}
K.~Fukaya, Y.-G. Oh, H.~Ohta, and K.~Ono.
\newblock {\em Lagrangian intersection {F}loer theory: anomaly and obstruction.
  {P}art {II}}, volume~46 of {\em AMS/IP Studies in Advanced Mathematics}.
\newblock American Mathematical Society, Providence, RI, 2009.

\bibitem[FSS1]{FSS:ex}
K.~Fukaya, P.~Seidel, and I.~Smith.
\newblock Exact {L}agrangian submanifolds in simply-connected cotangent
  bundles.
\newblock {\em Invent. Math.}, 172(1):1--27, 2008.

\bibitem[FSS2]{FSS:cbndl}
K.~Fukaya, P.~Seidel, and I.~Smith.
\newblock The symplectic geometry of cotangent bundles from a categorical
  viewpoint.
\newblock In {\em Homological mirror symmetry}, volume 757 of {\em Lecture
  Notes in Phys.}, pages 1--26. Springer, Berlin, 2009.

\bibitem[HLL]{HLL:monodromy}
S.~Hu, F.~Lalonde, and R.~Leclercq.
\newblock Homological {L}agrangian monodromy.
\newblock {\em Geom. Topol.}, 15(3):1617--1650, 2011.

\bibitem[Kha]{Kh:diameters}
M.~Khanevsky.
\newblock Hofer's metric on the space of diameters.
\newblock {\em J. Topol. Anal.}, 1(4):407--416, 2009.

\bibitem[KM]{Kat-Mil:PSS}
J.~Kati\'c and D.~Milinkovi\'c.
\newblock {P}iunikhin-{S}alamon-{S}chwarz isomorphisms for {L}agrangian
  intersections.
\newblock {\em Differential Geom. Appl.}, 22(2):215--227, 2005.

\bibitem[KMN]{KMN:spec}
J.~Kati\'{c}, D.~Milinkovi\'{c}, and J.~Nikoli\'{c}.
\newblock Spectral invariants in {L}agrangian {F}loer homology of open subset.
\newblock {\em Differential Geom. Appl.}, 53:220--267, 2017.

\bibitem[KS]{Kis-Sh}
A.~Kislev and E.~Shelukhin.
\newblock Bounds on spectral norms and barcodes.
\newblock Preprint (2018). Can be found at
  \url{https://arxiv.org/abs/1810.09865v1}.

\bibitem[Lec]{Leclercq:spectral}
R.~Leclercq.
\newblock Spectral invariants in lagrangian floer theory.
\newblock {\em J. Mod. Dyn.}, 2:249--286, 2008.

\bibitem[LZ]{Le-Za:spec-inv}
R.~Leclercq and F.~Zapolsky.
\newblock Spectral invariants for monotone {L}agrangians.
\newblock {\em J. Topol. Anal.}, 10(3):627--700, 2018.

\bibitem[MW]{Mak-Wu:Dehn-twist}
C.-Y. Mak and W.~Wu.
\newblock Dehn twist exact sequences through {L}agrangian cobordism.
\newblock {\em Compos. Math.}, 154(12):2485--2533, 2018.

\bibitem[Nad]{Na:cotangent}
D.~Nadler.
\newblock Microlocal branes are constructible sheaves.
\newblock {\em Selecta Math. (N.S.)}, 15(4):563--619, 2009.

\bibitem[Oh1]{Oh:symp-top-act-1}
Y.-G. Oh.
\newblock Symplectic topology as the geometry of action functional. {I}.
  {R}elative {F}loer theory on the cotangent bundle.
\newblock {\em J. Differential Geom.}, 46(3):499--577, 1997.

\bibitem[Oh2]{Oh:symp-top-act-2}
Y.-G. Oh.
\newblock Symplectic topology as the geometry of action functional. {II}.
  {P}ants product and cohomological invariants.
\newblock {\em Comm. Anal. Geom.}, 7(1):1--54, 1999.

\bibitem[PRSZ]{PRSZ:top-pers}
L.~Polterovic, D.~Rosen, K.~Samvelyan, and J.~Zhang.
\newblock Topological persistence in geometry and analysis.
\newblock Preprint (2019). Can be found at
  \url{https://arxiv.org/pdf/1904.04044}.

\bibitem[PSS]{Po-Sh-St:pers-op}
L.~Polterovich, E.~Shelukhin, and V.~Stojisavljevi\'{c}.
\newblock Persistence modules with operators in {M}orse and {F}loer theory.
\newblock {\em Mosc. Math. J.}, 17(4):757--786, 2017.

\bibitem[Sei1]{Se:long-exact}
P.~Seidel.
\newblock A long exact sequence for symplectic floer cohomology.
\newblock {\em Topology}, 42(5):1003--1063, 2003.

\bibitem[Sei2]{Se:book-fukaya-categ}
P.~Seidel.
\newblock {\em Fukaya categories and {P}icard-{L}efschetz theory}.
\newblock Zurich Lectures in Advanced Mathematics. European Mathematical
  Society (EMS), Z\"urich, 2008.

\bibitem[She1]{Shl:vit-2}
E.~Shelukhin.
\newblock Symplectic cohomology and a conjecture of {V}iterbo.
\newblock Preprint (2019). Can be found at
  \url{https://arxiv.org/pdf/1904.06798}.

\bibitem[She2]{Shl:vit-1}
E.~Shelukhin.
\newblock Viterbo conjecture for {Z}oll symmetric spaces.
\newblock Preprint (2018). Can be found at
  \url{https://arxiv.org/pdf/1811.05552}.

\bibitem[Ush1]{Usher1}
M.~Usher.
\newblock Boundary depth in {F}loer theory and its applications to
  {H}amiltonian dynamics and coisotropic submanifolds.
\newblock {\em Israel J. Math.}, 184:1--57, 2011.

\bibitem[Ush2]{Usher2}
M.~Usher.
\newblock Hofer's metrics and boundary depth.
\newblock {\em Ann. Sci. \'Ec. Norm. Sup\'er. (4)}, 46(1):57--128 (2013), 2013.

\bibitem[UZ]{Usher-Zhang:perh}
M.~Usher and J.~Zhang.
\newblock Persistent homology and {F}loer--{N}ovikov theory.
\newblock {\em Geom. Topol.}, 20(6):3333--3430, 2016.

\bibitem[Vit1]{Vi:homogen}
C.~Viterbo.
\newblock {S}ymplectic homogenization.
\newblock Preprint (2007). Can be found at
  \url{https://arxiv.org/pdf/0801.0206}.

\bibitem[Vit2]{Vi:generating-functions-1}
C.~Viterbo.
\newblock Symplectic topology as the geometry of generating functions.
\newblock {\em Math. Ann.}, 292(4):685--710, 1992.

\end{thebibliography}

%
%
%

\end{document}